\documentclass[10pt]{amsart}

\usepackage{xr}
\usepackage{url}

\usepackage{hyperref}
\hypersetup{
    colorlinks=true,
    linkcolor=blue,
    citecolor=red,
    pdfpagemode=FullScreen,
    pdftitle={Aletti,Crimaldi,Ghiglietti}
    }

\usepackage[a4paper, total={6.8in, 9in}]{geometry}

\usepackage[utf8]{inputenc}
\usepackage[T1]{fontenc}
\usepackage{amsmath, amssymb}
\usepackage{bbm}
\usepackage{enumitem}
\usepackage{pdflscape}
\usepackage{cancel}

\allowdisplaybreaks

\theoremstyle{plain}

\newtheorem{thm}{Theorem}[section]
\newtheorem{lem}[thm]{Lemma}

\newtheorem{cor}[thm]{Corollary}

\makeatletter
\def\th@newremark{\th@remark\thm@headfont{\bfseries}}
\makeatletter

\theoremstyle{newremark}
\newtheorem{rem}{Remark}[section]

\theoremstyle{definition}

\usepackage{tikz}
\foreach \a in {q,w,e,r,t,y,u,i,o,p,a,s,d,f,g,h,j,k,l,z,x,c,v,b,n,m,Q,W,E,R,T,Y,U,I,O,P,A,S,D,F,G,H,J,K,L,Z,X,C,V,B,N,M}{
  \expandafter\xdef\csname v\a\endcsname {
		{\noexpand\mathbf{\a}}
	}
}

\newcommand{\ind}{\mathbbm{1}}

\newcommand{\Rstar}{\widetilde{R}}

\begin{document}

\title[]{A multi-factorial innovation model with mean-field interaction}

\author[G. Aletti]{Giacomo Aletti}
\address{Universit\`a degli Studi di Milano, Milan, Italy}
\email{giacomo.aletti@unimi.it}

\author[I. Crimaldi]{Irene Crimaldi}
\address{IMT School for Advanced Studies Lucca, Lucca, Italy}
\email{irene.crimaldi@imtlucca.it}

\author[A. Ghiglietti]{Andrea Ghiglietti}
\address{Universit\`a degli Studi di Milano-Bicocca, Milan, Italy}
\email{andrea.ghiglietti@unimib.it (Corresponding author)}

\thanks{G.\ Aletti is a member of ``Gruppo Nazionale per il Calcolo
  Scientifico (GNCS)'' of the Italian Institute ``Istituto Nazionale
  di Alta Matematica (INdAM)''. I.\ Crimaldi and A.\ Ghiglietti are members of ``Gruppo
  Nazionale per l'Analisi Matematica, la Probabilit\`a e le loro
  Applicazioni (GNAMPA)'' of the Italian Institute ``Istituto
  Nazionale di Alta Matematica (INdAM)''. Irene Crimaldi thanks the “Resilienza Economica e Digitale” (RED) 
project (CUP
D67G23000060001) funded by the Italian Ministry of University and Research (MUR)
as “Department of Excellence” (Dipartimenti di Eccellenza 2023-2027, Ministerial
Decree no. 230/2022).}



\date{\today}


\begin{abstract}
We introduce an Indian-buffet-type model for multi-factorial innovation in which each arriving agent may exhibit both previously observed and new features. The number of new features follows a power-law  behavior, while the probability of selecting an old feature combines self-reinforcement, depending on the feature-specific popularity, with a mean-field interaction term depending on the average popularity of all observed features. The model is governed by the usual innovation parameters (mass, discount and concentration), together with two additional parameters: one controlling the strength of reinforcement against a forcing input toward zero, and one regulating the intensity of the mean-field interaction.
\\
\indent Although the growth of the total number of distinct observed features has the same behavior as in the three-parameter Indian buffet process, the mean-field interaction mechanism produces new asymptotic regimes. For aggregate quantities, including the predictive mean, the averaged number of features per agent, the mean inclusion probability, and the mean feature popularity, the phase transition is determined by the comparison between the discount parameter and the weight of the forcing input. For feature-specific quantities, a further transition appears according to the comparison between the interaction level and
a critical threshold. In particular, high interaction leads to an asymptotic synchronization of feature-specific inclusion probabilities.
\\
\indent We establish strong laws and second-order asymptotic results, including central limit theorems in regimes where martingale fluctuations compete with deterministic or random terms. The analysis relies on novel general results for recursive stochastic dynamics, which may be useful beyond the present framework.
\end{abstract}

\maketitle

\section{Introduction}

Understanding the mechanisms by which novelties emerge and propagate 
is of crucial interest across various disciplines, including biology, linguistics, and social sciences. In probabilistic terms, a novelty
or innovation is defined as the first occurrence of an event of interest. Widely employed mathematical models for innovation processes are urn models with infinitely many
colors and species sampling sequences \cite{BAN-THA-2022, lijoi07, lijoi2010, pitman_1996, Sariev23}. However, in these models, at each timestep, each observed agent/item shows only one feature, that may be an old one (that is, already observed in the past) or a new one (that is, appeared for the first time). Instead in the present work, we deal with \emph{multi-factorial innovation processes}, where at each timestep an agent/item is observed and it can exhibit a certain (random) number of features (factors), some old and some new. The total number of features is
not specified in advance but is allowed to grow as new data points (agents/items) are observed and it is inferred from the data.
Our aim is not to propose a new Bayesian nonparametric prior for posterior inference, but rather to study a tractable stochastic dynamics for multi-factorial innovation with mean-field interaction, so that an old feature may be selected not only because it is itself popular,
but also because the overall feature environment is highly active.
\\

\indent The \textit{Indian Buffet Process} (IBP) is a classical model in Bayesian nonparametric feature allocation, designed to describe observations endowed with an unbounded number of latent features. Introduced 
in \cite{GG06} and developed in
 \cite{GG11, TJ}, it defines a probability distribution over infinite binary matrices, where each row corresponds to an agent/item and each column to a feature. This allows each agent/item to simultaneously possess multiple characteristics, distinguishing the IBP from species sampling sequences and clustering models, which typically assign each data point to a single class.
In this sense, the IBP provides a useful benchmark for probabilistic models of feature allocation, especially because it combines an infinite-dimensional feature space with a relatively explicit predictive structure.
\\

\indent The generative metaphor of the Indian buffet provides an intuitive interpretation of the process: a sequence of customers (representing agents/items) enters an Indian restaurant with an infinite buffet of dishes (representing features). The first customer samples a $\mathrm{Poisson}(\alpha)$ number of dishes, while each subsequent customer chooses previously sampled dishes with probability proportional to their popularity, and also tries new dishes according to a $\mathrm{Poisson}(\alpha/t)$ distribution. The resulting binary feature-allocation matrix is exchangeable, meaning that the order of the observations does not affect their joint distribution. This property ensures analytical tractability and enables efficient posterior inference, making the IBP  a natural reference model for exchangeable Bayesian nonparametric feature-allocation priors. 
At the same time, exchangeability and feature-wise independence limit the ability of the classical IBP to describe innovation mechanisms in which the probability of adopting an old  feature depends not only on its own past popularity, but also on the global activity level of the whole system.
\\

\indent Over the years, several generalizations of the IBP have been proposed to extend its applicability, either by preserving exchangeability while enriching the prior structure, or by relaxing exchangeability and independence assumptions. A fundamental theoretical reformulation of the IBP was introduced 
in \cite{TJ}, who demonstrated that the IBP can be viewed as the marginal distribution of a Beta--Bernoulli process. In this framework, a Beta process defines a random measure generating feature probabilities, while a Bernoulli process specifies feature allocation for each observation. This connection provides a rigorous measure-theoretic foundation, linking the IBP to other Bayesian nonparametric priors such as the Dirichlet Process and enabling efficient inference through conjugacy properties. To capture the heavy-tailed distributions often observed in empirical data, power-law variants of the IBP have been developed. Precisely, the IBP can be extended to capture power-law behavior in the total number of observed features through the three-parameter IBP~\cite{broderick2012beta, TG}, which introduces a discount parameter (also known as stability exponent), that regulates the asymptotic behavior of the overall number of different observed features, a mass parameter, that controls the total number of new features exhibited by an agent/item and a concentration parameter, that tunes the number of agents/items per feature. Such a model better fits real-world data, but one fundamental assumption persists: the inclusion probabilities of features are independent across features and determined solely by past feature-specific counts. 

Since the assumption of independence among features limits its expressiveness in settings where correlations or dependencies among features are crucial, some (non-exchangeable) variants of the IBP have been proposed. Among the most influential ones is the \textit{dependent Indian buffet process} \cite{WOG}, which introduces correlations among feature allocations by making the inclusion probability of each feature a function of covariates or of latent variables modeled through Gaussian processes. This allows the IBP to represent smooth temporal or spatial dependencies, although it increases the computational burden of inference. Another notable development is the \textit{hierarchical Indian buffet process} and related \cite{rai-dau, dos-gha-correlated}, which introduces a hierarchical structure where groups of observations share a common global pool of features while maintaining group-specific feature usage. This model has proven particularly effective in multi-task and transfer learning contexts, enabling information sharing across related datasets but at the cost of higher model complexity and more demanding inference procedures. 
\cite{DW-restricted} developed the \textit{restricted Indian buffet process}, which allows control over the distribution of the number of features per observation. This generalization enables more flexible modeling of feature sparsity patterns and partially introduces interactions through restriction mechanisms. However, these restrictions are imposed at the level of the number of active features rather than their pairwise relationships. The \textit{Indian buffet Hawkes process} proposed 
in \cite{IBP2018hawkes} models evolving features over time by combining the IBP with Hawkes processes. The interactions in their model occur in the temporal domain, where previous feature activations influence future events via self-exciting processes. While this captures a dynamic form of interaction, the focus is on temporal excitation rather than structural co-dependency among features. 
\cite{heaukulani2020gibbs} introduced a class of \textit{Gibbs-type Indian buffet processes}, where the underlying random measure is generalized to a Gibbs-type measure. This allows for more flexible prior specifications and includes models that exhibit power-law behavior. Although their framework can model complex feature allocation structures, the interactions among features are implicitly encoded through the prior, rather than through an explicit interaction matrix. Another direction of advancement has been to incorporate structural or spatial dependencies among agents/items. For example, the \textit{distance-dependent Indian buffet process} \cite{gershman2015distance} modifies the IBP by making the probability of feature sharing a function of the distance between data points (agents/items), allowing it to model local correlations. Similarly, the \textit{phylogenetic Indian buffet process} \cite{MGJ} introduces a tree-structured dependency among agents/items, capturing shared evolutionary histories in a Bayesian framework. More recently, the \textit{attraction Indian buffet distribution} \cite{warr2021attraction} has been developed to allow pairwise attraction between agents, which increases the likelihood of feature sharing. Another extension of the IBP is the \textit{Indian buffet process with random weights}, introduced 
in \cite{BCPR-IBP}. In this formulation, each agent is associated with a random weight, that represents the relevance of the agent and enters the feature probabilities. This additional stochastic layer generalizes the Beta--Bernoulli construction and connects the IBP with the theory of randomly reinforced stochastic processes. The random-weight IBP maintains some of the analytical properties of the classical model while offering greater flexibility in capturing heterogeneity among the agents. Finally, recent Bayesian nonparametric works have also considered feature-allocation models in connection with prediction of unseen features. For instance, 
\cite{CFMB2024} introduced scaled-process priors for Bayesian nonparametric estimation of unseen genetic variation. Their stable-Beta scaled-process prior enriches the posterior distribution of the number of unseen features, replacing the Poisson posterior structure induced by completely random measure priors with a tractable negative-binomial posterior depending on both the sample size and the number of observed distinct features. This line of work is close to ours at the level of feature-allocation modeling and multi-feature data, but it has a different purpose: it concerns exchangeable Bayesian nonparametric priors and posterior prediction, whereas the present paper studies a non-exchangeable recursive innovation dynamics with explicit mean-field interaction among features.
\\

\indent IBP-type and related feature-allocation models have been used in several domains where observations may exhibit multiple latent or observed features.
As examples, we mention bioinformatics~\cite{KG}, network theory \cite{bol-cri-mon-2016, MGJ2009, SCJ}, causal inference~\cite{WGG}, modeling of choices~\cite{GJR} and similarity judgements~\cite{NG}. 
Feature-allocation ideas have also been used in Bayesian nonparametric methods for unseen-feature prediction, with particular emphasis on genetic and genomic variation (see the recent paper~\cite{CFMB2024} and the references therein).
\\

\indent Summing up, while the IBP assumes exchangeability and inclusion probabilities independent across features, several recent models have sought to extend the framework by incorporating dependencies. Indeed, ignoring such interactions can lead to oversimplified models that fail to capture the observed structure in real data. Most of these models focus on dependencies among the \textit{rows} of the binary matrix - i.e., the data points (the agents/items) themselves - rather than among the \textit{columns} - i.e., the features. The few models, mentioned above, that deal with feature interaction, offer a more realistic representation of evolving phenomena, but they lose the elegant mathematical formulation of the original IBP and make a detailed theoretical analysis difficult. Moreover they require more complex inference techniques, such as sequential Monte Carlo or dynamic variational methods, and often come at the cost of increased computational complexity. 

The present paper takes a different route. We introduce an IBP-inspired, non-exchangeable stochastic process in which feature inclusion probabilities evolve through a combination of feature-specific self-reinforcement and a mean-field interaction term. The mean-field term represents the global level of activity of the observed feature system: a feature may be selected not only because it is itself popular, but also because the overall feature environment is highly active. This mechanism gives a simple and interpretable way to model cross-feature dependence while preserving enough recursive structure to allow a sharp asymptotic analysis.
Our proposed model directly introduces interactions in the inclusion probabilities, which permits clearer interpretability and a more faithful representation of domains where features are influenced both by their own popularity and by the global level of activity of the feature system, 
such as in biological networks, cultural evolution, or recommendation systems. Additionally, 
our model, characterized by an explicit dependence among features, allows us to 
conduct a detailed theoretical study, proving various asymptotic results, such as strong laws and central limit theorems for several key quantities, such as the total number of observed features ($D_t$), the averaged number of features exhibited by the agents/items 
($\overline{T}_t$), the averaged feature inclusion probability ($\overline{P}_t$), the averaged number of agents/items per feature ($\overline{K}_t$) and also, for feature-specific quantities, such as the inclusion probability ($P_{t,j}$) and the popularity ($K_{t,j}$) for an observed feature $j$. 
While the asymptotic behavior of $D_t$ is the same as in the classical IBP with three parameters~\cite{TG}, 
we obtain new phenomena such as 
the power-law behavior of the averaged quantities and of the dish-specific quantities, not seen in the classical IBP, where, except in the extreme cases $\beta=1$, the empirical mean $\overline{T}_t$ converges almost surely to a strictly positive real random variable and the feature-popularity grows linearly. 
The mean-field interaction mechanism also produces a second level of phase transition for feature-specific quantities. In particular, when the interaction parameter is sufficiently large, the feature-specific inclusion probabilities synchronize asymptotically, whereas in the low-interaction regime the long-time behavior remains feature dependent.
\\

\indent We also highlight the novel Theorem~\ref{th-general} and~Theorem~\ref{th-CLT-general} in the supplementary, which are stated and proven in a general setting and so they could be also applied in other contexts. In a framework of recursive stochastic dynamics, the first one provides conditions to guarantee that the limit random variable of the suitable rescaled process is non-zero with probability one, while the second one deals with the related second-order asymptotic. 
These results are not merely technical tools for the present model. They address a class of recursive dynamics in which a polynomially rescaled process converges almost surely to a random limit, and where the second-order behavior is determined by the competition between martingale fluctuations and a slowly decaying deterministic forcing term. This setting is different from the standard stochastic-approximation framework centered around a deterministic attracting equilibrium, and it naturally appears in reinforced and innovation-type processes.
\\ 

The main contribution of the paper is therefore twofold. First, we introduce a tractable interacting multi-factorial innovation process and identify the phase transitions governing both averaged and feature-specific quantities. Second, as a consequence, we develop recursive-dynamics tools that yield strong laws, exact rates, positivity of random limits and central limit theorems for the relevant observables.
\\

The rest of the paper is organized as follows. In Section~\ref{sec:notation_model_overview}, we illustrate some adopted notation, present the new model, which is the object of this work, and provide an overview of the main results of our study. In Section~\ref{preliminaries}, we collect several preliminary properties of the model, that will be repeatedly used throughout the paper, and we describe the asymptotic behavior of the total number $D_t$ of the observed features. Section~\ref{results-mean} is devoted to the asymptotic behavior of the average quantities $Z_t$, $\overline{T}_t$, $\overline{K}_t$ and $\overline{P}_t$. In Section~\ref{results-dish-specific}, we derive the asymptotic behavior of the quantities $K_{t,j}$ and $P_{t,j}$ associated with each individual feature $j$. Section~\ref{sec:future_research} concludes the paper with some final remarks and a discussion of possible directions for future research.
Finally, the paper is enriched by a wide supplementary, collecting, among other things, novel general results for recursive dynamics (we refer to Section~\ref{app-tech-res}),  
and non-trivial technical results for the considered model (we refer to Section~\ref{sec:technical_results_model}). Sections, theorems and remarks reported in the supplementary document are referenced with capital letters,
while those presented in the main document are referenced with numbers.

\section{Notation, Model, and Overview of the Results}\label{sec:notation_model_overview}
In this section, we introduce the basic notation used throughout the paper, present the interacting multi-factorial innovation model, and conclude with an overview of the main results established in this work.

\subsection{Notation}
In the sequel, given two possibly random sequences $(\rho_t)$ and $(V_t)$, we use: 
\begin{itemize}
\item the expression $\rho_t=O(V_t)$ with $V_t > 0$ in order to indicate  that 
$\limsup_t |\rho_t|/V_t<+\infty$ with probability one;
\item  the expression 
$\rho_t=o(V_t)$ in order to indicate that we have 
$\rho_t / V_t \stackrel{a.s.}\longrightarrow 0$;
\item the symbols $O(V_t)$ and $o(V_t)$ to denote generic, possibly random, sequences with 
the above properties. 
\end{itemize}
Note that, if $V_t=v_t$ is deterministic, the symbols $O(v_t)$ and $o(v_t)$ can indicate both random or deterministic sequences. 
This will usually be clear from the context, but where it is not and it is important to highlight it, we'll write it explicitly.
We adopt the notation  $o_P(\cdot)$ when the above limit is in the sense of the convergence in probability, instead of a.s. convergence.

\subsection{The model}

Fix $\alpha>0$, $0\leq \beta\leq 1$, $\theta>0$, $0<w\leq 1$ and $0\leq \iota\leq 1$. The first three parameters (i.e. $\alpha$, $\beta$ and $\theta$) correspond to the three parameters 
(mass, discount and concentration) 
that characterize the Indian buffet process introduced in~\cite{TG}; while the two last parameters (i.e. $w$ and $\iota$) are the new ones, introduced in order to tune the {\em forcing input} toward zero in the inclusion probabilities and the {\em interaction among the dishes}, respectively. 
\\

 \indent The dynamics is as follows. Time is discrete and indexed by customer arrivals.
  Customer 1 tries $N_1\stackrel{d}=\,$Poi$(\lambda_0)$
  dishes with $\lambda_0=\alpha$. 
  For each time-step $t\geq 1$, let ${\mathcal O}_t$ be the collection of dishes experimented
  by the first $t$ customers,
  that we call the ``old'' dishes at time-step $t$. Denoting by ${\mathcal F}=({\mathcal F}_t)_t$
  the (complete) natural filtration associated to the model, then:

\begin{itemize}
\item Customer $t+1$ selects a subset ${\mathcal O}_{t+1}^*\subseteq {\mathcal O}_t$.
  Each  $j\in {\mathcal O}_t$ is included or not into ${\mathcal O}_{t+1}^*$ independently of the other members
  of ${\mathcal O}_t$ with ${\mathcal F}_t$-conditional probability (called {\em inclusion probability}) 
  \begin{equation}\label{inclusion-probab}
    \begin{split}
  P_{t,j}=P(X_{t+1,j}=1|{\mathcal F}_t)=w\left[(1-\iota)\frac{K_{t,j}}{\theta+t}+
        \iota \frac{1}{D_t}\sum_{i\in {\mathcal O}_t} \left(\frac{K_{t,i}}{\theta+t}\right)\right]\,.
  \end{split}
\end{equation}
    where $X_{n,i}$ is the indicator of the event $\{$customer $n$ selects dish $i\}$,
    $K_{t,i}=\sum_{n=1}^t X_{n,i}$ is the number of customers who tried dish $i$ until time-step $t$  and
    $D_t=card({\mathcal O}_t)$.

\item In addition to ${\mathcal O}_{t+1}^*$, customer $t+1$ also tries
  (independently of the past ${\mathcal F}_t$ and of the choice of ${\mathcal O}_{t+1}^*$) a random number 
  $N_{t+1}\stackrel{d}=\,$Poi$(\lambda_t)$ of new dishes,
  where Poi$(\lambda_t)$
  denotes the Poisson distribution with mean $\lambda_t=\alpha/(t+1)^{1-\beta}$. 
\end{itemize}
Note that in \eqref{inclusion-probab} a mean-field interaction term is present: feature \(j\) interacts with each other features through the empirical average popularity of all currently observed features. Thus the model
is based on a global interaction field.
We can say that the case when $w=1$ and $\iota=0$ is related to the standard IBP of~\cite{TG}, although 
in the standard model we have 
$P_{t,j}=(K_{t,j}-\beta)/(\theta+t)$ and 
$\lambda_t=\alpha \Gamma(\theta+1)\Gamma(\theta+\beta+t)/[\Gamma(\theta+\beta)\Gamma(\theta+1+t)]$. These definitions for $P_{t,j}$ and $\lambda_t$ make the standard IBP to be exchangeable. However, despite these differences, it holds true (as in our model with $w=1$ and $\iota=0)$ that $P_{t,j}\sim K_{t,j}/(\theta+t)$ and $\lambda_t\sim C/(t+1)^{1-\beta}$, for some constant $C>0$, as $t\to +\infty$, and so the asymptotic properties of our model with $w=1$ and $\iota=0$ are similar to the ones of the standard IBP. Instead, as we will see,  
 the second-order asymptotic behavior presents significant differences. Note that, from an applicative point of view, our simplified way to define the inclusion probabilities $P_{t,j}$  and the parameter $\lambda_t$ allows for a better interpretation and a clearer identification of the role played by each single parameter. This is the reason behind our choice to define them as above. 
 Hence, the introduced model should  be viewed not as an exchangeable IBP variant intended for Bayesian nonparametric inference, but as an IBP-inspired stochastic model suitable for innovation processes with possibly more than one feature (factor) per agent.
\\
\indent Moreover, in the above introduced model, as in standard IBP, for each fixed time-step $t+1$, 
the random variables $\{X_{t+1,j}:\, j\in {\mathcal O}_t\}$ are conditionally independent
given the past; while, differently with respect to the standard model, for each fixed $j\in{\mathcal O}_t$, 
the random variable $X_{t+1,j}$ is not independent
of the random variables $\{X_{n,i}:\, i\neq j,\, n\leq t\}$. Finally, as in the standard model,
the random variable $N_{t+1}$ is independent of $N_1,\dots, N_t$ and of all the
choices done for the old dishes, i.e. of~$\{X_{n,j}:  j\in{\mathcal O}_n,\, n\leq t+1\}$.
\\
\indent The parameters  $\alpha$ (known as the mass parameter) and $\beta$ (known as the discount parameter or the stability exponent) drive the growth of the number of tested dishes 
(see Subsection~\ref{subsec-D}), while the parameter $\theta$ (known as the concentration parameter) regulate the initial condition and the behavior of the model in the initial phase (i.e. for a finite number of time-steps).   
Regarding the parameter~$w$, we observe that the inclusion probability can be seen as a convex combination of the reinforcement probability with weight $w\in (0,1]$ and 
the null probability with weight $(1-w)\in [0,1)$. Hence, the parameter $w$ tunes the weight of the two mechanisms: the {\em reinforcement} and the {\em forcing input toward zero}. 
The first one can be explained by the popularity principle: the more popular is a feature (dish), the higher is the probability that a future agent (customer) selects it.  
The last mechanism can be explained thinking about the fact that the features (dishes) could be subject to a fad and so their probability to be selected is "forced" to vanish.  
Finally, the reinforcement term is itself a convex combination of a term that depends on the averaged popularity of the appeared features (tested dishes) with weight $\iota\in [0,1]$ and 
a term that depends only on the specific popularity of the feature (dish) $j$ with weight $(1-\iota)\in [0,1]$. 
Hence, the {\em interaction} present in the model dynamics is of the {\em mean-field type}, whose intensity is ruled by the parameter $\iota$ and in the inclusion probability 
$P_{t,j}$ the total weight of the {\em self-reinforcement} is $w[(1-\iota)+ \iota/D_t]$, while 
the weight of each {\em cross-reinforcement} (that is the reinforcement coming from a feature $i\in{\mathcal O}_t\setminus\{j\}$) 
is $w\iota/D_t$.  
\\
\indent Generally speaking, a mean-field interaction is 
used when each agent perceives the overall average context, not just the number of similar agents, 
so that the agent's behavior is influenced by an aggregate perception of the system, rather than by
direct contacts. In other words, the system  of interest is simplified by replacing the complex pairwise interactions with an effective average field that
represents the collective effect of all the others. Hence, in the IBP framework, 
instead of calculating how every feature (dish) affects every other one (which is extremely difficult when the features are in a large number), 
we assume that each feature is affected by only an average field generated by all the others appeared in the system. 
Just to give an example in order to facilitate the interpretation and visualize the framework, consider 
the following setting. An agent (customer) enters a community and adheres to some existing proposals (old dishes), 
while also introducing new ones (new dishes). Each agent decides to join an existing proposal not only based on how many agents have already joined that proposal, 
but also on the overall average adoption across all proposals - in other words, on a global perception of participation across all the existing proposals.  
This can be seen as a mean-field interaction process in an ``idea space''. 
Each agent supports existing ideas not only based on how many others support them,
but also based on the overall average level of support across all ideas.

\subsection{Overview of the Results}

Regarding the results we are going to prove, we observe that the behavior of the number $D_t$ of the tested dishes along the time-steps is not affected by 
the interaction among the dishes and hence the related results are standard. We collect them in the following Section~\ref{preliminaries} for the reader's convenience;
 while the main results of this work concern the asymptotic behaviors of 
  the dish(feature)-specific inclusion probability $P_{t,j}$, the dish(feature)-specific popularity $K_{t,j}$, the number $T_t$ of dishes tested by customer $t$, that is 
  $T_t=\sum_{j\in {\mathcal O}_t} X_{t,j} + N_t$ and the corresponding averaged quantities
\begin{equation}\label{eq:average_quantities}
\overline{P}_t=\frac{1}{D_t}\sum_{j\in {\mathcal O}_t} P_{t,j},\quad
\overline{K}_t=\frac{1}{D_t}\sum_{j\in{\mathcal O}_t}K_{t,j}
\quad\mbox{and}\quad
\overline{T}_t=\frac{1}{t}\sum_{n=1}^t T_n\,,
\end{equation}
that corresponds to the mean dish(feature)-inclusion probability, the mean dish(feature)-popularity and the mean number of dishes (features) per customer (agent/item), respectively, until time-step $t$. 
We also analyze the predictive mean $Z_t=E[T_{t+1}|{\mathcal F}_t]$ of the number of dishes tested (features exhibited) by the future customer (agent/item) $t+1$ given 
the past ${\mathcal F}_t$.  
These results are presented in Section~\ref{results-mean} (for the averaged quantities) and in Section~\ref{results-dish-specific} (for the specific-dish quantities). 
In particular, we highlight that the asymptotic behaviors of the averaged quantities only depend on the relationship between the parameters $\beta$ and $w$ (see 
Table~\ref{table-Z-T-P-K-medio}), while 
the interaction intensity $\iota$ plays a role in the results describing the asymptotics for the specific-dish quantities (see Table~\ref{table-K-j} and Table~\ref{table-P-j}).   This is a consequence of the specific form of the interaction considered in this work: the mean-field interaction influences the processes at the level of 
    individual dishes but vanishes when considering the averaged quantities. 
   This represents an advantage from the applicative point of view  because this means that we can use the parameters $\alpha$ and $\beta$ to fit the growth of the observable number $D_t$ 
   of distinct dishes (features), we can fit the behavior of the observable averaged quantities $\overline{T}_t$ by the parameter $w$, while we can use the parameter $\iota$ to fit the observable $j$-specific quantities $K_{t,j}$. 
\\
\indent For the first-order results in Table~\ref{table-Z-T-P-K-medio}, the main considerations are the following.
The predictive mean $Z_t$ of the number of dishes chosen at the future time $t+1$ converges to zero when both $\beta$ and $w$ are less than 1. The decay is sub-linear (specifically, a power-law or a power-law with a logarithmic correction), with a power-law exponent depending on the weight $(1-w)$ of the forcing input toward zero or on the power-law exponent $(1-\beta)$ governing the decrease of the expected number $\lambda_t$ of new features. 
When $\max\{\beta, w\} = 1$, it converges to a finite strictly positive quantity if $\beta \neq w$, whereas it diverges to infinity if $\beta = w$.
The averaged quantity $\overline{T}_t$ exhibits exactly the same asymptotic behaviors, but the limit random variables coincide only when $w=1$ and so, in that case, we can affirm that the empirical mean $\overline{T}_t$ can be used as a strongly consistent estimator of the predictive mean $Z_t$.   These results for $w=1$ are consistent with the ones known for the standard IBP of \cite{TG} (see~\cite{BCPR-IBP}), while the others are a consequence of the introduction in the model of the forcing input toward zero.
The mean inclusion probability~$\overline{P}_t$ always converges to zero. 
This result is of particular interest, as it explains the sparsity of the observed feature–matrix.
Although $\overline{P}_t$ always converges to zero, the mean dish-popularity $\overline{K}_t$  always diverges to $+\infty$, except when the $w<\beta$, i.e.~when the weight $(1-w)$ of the forcing input toward zero is higher than the power-law exponent 
$(1-\beta)$ of the expected number of new dishes. Specifically, when $w> \beta$, we have  a power-law growth (with or without a logarithmic correction) with an exponent related to the difference between $w$ and $\beta$, a logarithmic growth when $w=\beta$ and a convergence to a suitable constant when $w<\beta$.
\\
\indent Regarding the first-order results in Table~\ref{table-K-j} and Table~\ref{table-P-j}, we can make the following remarks. 
After its appearance, the inclusion probability $P_{t,j}$ of dish~$j$ tends to zero in all cases except when $\iota = 0$ and $w = 1$ 
(that, as already observed, essentially corresponds to the standard IBP of~\cite{TG}). All the dish-specific inclusion probabilities~$P_{t,j}$ converge at the same rate, meaning that the rate does not depend on $j$. The level $\iota$ of interaction influences the convergence rate at which each process $(P_{t,j})_t$ converges towards zero. 
Specifically, the higher the interaction parameter $\iota$, the faster the convergence to zero (recall that when $\iota = 0$ and $w = 1$, $P_{t,j}$ does not even tend to zero). 
Additionally, the level of interaction determines how closely the processes $P_{t,j}$ are asymptotically related to each other. 
Moreover, comparing Table~\ref{table-Z-T-P-K-medio} and Table~\ref{table-P-j}, we can see that, in the most of cases,  the  
averaged inclusion probability  $\overline{P}_t$ is asymptotically lower than the specific inclusion probability $P_{t,j}$ of any dish $j$. 
The level $\iota$ of interaction quantifies the difference between $\overline{P}_t$ and the single $P_{t,j}$. Specifically:
    \begin{itemize}
        \item \emph{High interaction ($\iota=1$ or $\beta/w<\iota<1$):} The processes $P_{t,j}$, once rescaled, converge towards the same random variable, with the limit not depending on $j$.
    In other words, the processes $(P_{t,j})_t$ synchronize. The process $\overline{P}_t$ and the processes $P_{t,j}$ converge at the same rate, but when 
    $0<\beta/w<\iota<1$ their limits differ.
        \item  \emph{Critical regime  ($0<\iota=\beta/w<1$):} As above, the processes $(P_{t,j})_t$ synchronize. The process $\overline{P}_t$ and the processes $P_{t,j}$ converge at different rates.
        \item  \emph{Low interaction ($\iota=0$ or $0<\iota<\min\{\beta/w,1\}$):}  The processes $P_{t,j}$, once rescaled, converge towards possibly different random variables depending on $j$. 
        Moreover, the averaged quantity $\overline{P}_t$ and the processes $P_{t,j}$ converge at different rates. 
    \end{itemize}
It may seem surprising that, when all the processes $P_{t,j}$ synchronize, their common behavior still differs from the one of~$\overline{P}_t$, which is their average. Specifically: in the case of high interaction, 
$\overline{P}_t$ converges at the same rate as each $P_{t,j}$ does, 
but, when rescaled, its limit is strictly smaller than the common limit toward which all the rescaled $P_{t,j}$ converge, and, in the critical case, they have even different rates. The reason for this apparent paradox lies in the fact that the total number of dishes involved in the average diverges to infinity, while the rescaled processes $(P_{t,j})_t$ do not converge uniformly to their common limit. Consequently, although older dishes approach the common limit, there will always be a large number of more recent dishes whose values remain significantly below it.\\
\indent The $j$-specific popularity $K_{t,j}$ always increases to $+\infty$. The growth is logarithmic or according to a power-law and, always sub-linear, except when $w=1$ and $\iota=0$ (consistently with the behavior in the standard IBP of \cite{TG}). We can observe that the smaller the interaction intensity $\iota$ is, the faster is the growth to $+\infty$. This is coherent with the previous observation on $P_{t,j}$.\\
\indent We complete the analysis of the model with some asymptotic results of the second-order. Specifically, for the case $w>\beta$, we establish some central limit theorems (see Theorem~\ref{thm:second_order-Z} and Corollary~\ref{cor-clt}), 
that could be useful for the estimation of the random variable \( Z^*_\infty\), present in Tables~\ref{table-Z-T-P-K-medio}-\ref{table-P-j}, 
via confidence intervals based on the observable quantity \( \overline{T}_t \).  
Similarly, in the case of low-interaction, we establish a central limit theorem (see Theorem~\ref{thm:second_order-P_j}),  
that could be useful for the estimation of the random variable \( K^*_{\infty,j}\), present in Table~\ref{table-K-j} and Table~\ref{table-P-j}, 
via confidence intervals. In particular, in this result, we identify three regimes according to the relationships between the interaction intensity $\iota$ and the ratio $\beta/w$: the larger $\iota$ is, the slower the convergence to zero of the difference between the rescaled process $K_{t,j}$ and its limit $K^*_{\infty,j}$ is.\\

We conclude this overview by 
highlighting the two novel general results,  Theorem~\ref{th-general} and Theorem~\ref{th-CLT-general} in supplementary, that could also be applied
in other frameworks. In a context of a recursive dynamics, Theorem~\ref{th-general}
provides conditions to guarantee that the limit random variable of a
suitable rescaled process is almost surely non-zero.  
This fact is important since it allows to conclude that the scaling factor 
 provides the exact rate of convergence. Note that the standard arguments 
 in literature serve to prove only the convergence of the rescaled process to a real random variable, while  we 
 succeed to prove that the limit random variable is non-zero with probability one  
 by means of a non-standard technique.  Theorem~\ref{th-CLT-general} 
 deals with the related second-order asymptotic and it faces within a framework excluded 
  from the usual central limit theorems applied in the literature when dealing with recursive dynamics, such as the ones in~\cite{Zhang2016}.

\renewcommand{\arraystretch}{1.5}
\begin{table}[htbp!]
\caption{Almost sure asymptotic behavior of:\\ (i) the predictive mean $Z_t$,\\ (ii) the empirical mean $\overline{T}_t$,\\ (iii) the averaged inclusion probability $\overline{P}_t$,\\ (iv) the averaged number of customers per dish $\overline{K}_t$.\\
The real random variable $Z^*_\infty$ is a.s.\ strictly positive}.
\label{table-Z-T-P-K-medio}
\centering
\small
\begin{tabular}{l|l|l|l|l|}
 & $Z_t=E[T_{t+1}\mid{\mathcal F}_t]$ & 
$\overline{T}_t=\sum_{n=1}^tT_n/t$ &
$\overline{P}_t=\sum_{j\in\mathcal{O}_t} P_{t}/D_t$ &
$\overline{K}_t=\sum_{j\in\mathcal{O}_t} K_{t,j}/D_t$ \\
\hline
$0=\beta<w$ & $t^{-(1-w)}Z^{*}_\infty$ &
$t^{-(1-w)}\frac{Z^{*}_\infty}{w}$ &
$\frac{t^{-(1-w)}}{\ln(t)}\frac{Z^{*}_\infty}{\alpha}$ &
$\frac{t^{w}}{\ln(t)}\frac{Z^{*}_\infty}{\alpha w}$
\\
$0<\beta<w$ & $t^{-(1-w)}Z^{*}_\infty$ &
$t^{-(1-w)}\frac{Z^{*}_\infty}{w}$ &
$t^{-(1-w+\beta)}\frac{Z^{*}_\infty\beta}{\alpha}$ &
$t^{w-\beta}\frac{Z^{*}_\infty\beta}{\alpha w}$
\\
$0<\beta=w$ & $t^{-(1-w)}\ln(t)\alpha w$ &
$t^{-(1-w)}\ln(t)\alpha$ &
$t^{-1}\ln(t)w^{2}$ &
$\ln(t)w$
\\
$0<w<\beta$ & $t^{-(1-\beta)}\frac{\alpha \beta}{\beta-w}$ &
$t^{-(1-\beta)}\frac{\alpha}{\beta-w}$ &
$t^{-1}\frac{w\beta}{\beta-w}$ &
$\frac{\beta}{\beta-w}$\\
\hline
\end{tabular}
\end{table}

\renewcommand{\arraystretch}{1.5}
\begin{table}[htbp!]
\caption{Almost sure asymptotic behavior of the number of customers per dish. 
The real random variables $Z^*_\infty$ and $K^*_{\infty,j}$ are a.s.\ strictly positive.
}
\label{table-K-j} 
\centering
\small
\begin{tabular}{l|l|l|l|l|}
$K_{t,j}$ & $\iota=1$ & $\frac{\beta}{w}<\iota <1$ &
$0<\iota=\frac{\beta}{w} <1$ & $\iota=0$ or $0 < \iota<\min\{\frac{\beta}{w},1\}$\\
\hline
$0=\beta<w$ & $\tfrac{t^{w}}{\ln(t)}\frac{Z^{*}_\infty}{\alpha w}$ &
$\frac{t^{w}}{\ln(t)}\frac{1}{\alpha w}\, Z^*_\infty$
& \mbox{not possible}
& $t^{w} K^*_{\infty,j}$\\ 
$0<\beta<w$ & $t^{(w-\beta)}\frac{Z^{*}_\infty\beta}{\alpha(w-\beta)}$ &
$t^{w-\beta}\frac{\iota}{(\iota w - \beta)} \frac{\beta}{\alpha}\, Z^*_\infty$
& $t^{w-\beta}\ln(t)\iota \frac{\beta}{\alpha}\, Z^*_\infty$ & $t^{w(1-\iota)} K^*_{\infty,j}$ \\ 
$0<\beta=w$ & $\ln^2(t)\frac{w^{2}}{2}$ &\mbox{not possible} & \mbox{not possible} & $t^{w(1-\iota)} K^*_{\infty,j}$ \\ 
$0<w<\beta$ & $\ln(t)\frac{w\beta}{\beta-w}$ &\mbox{not possible} &\mbox{not possible} & $t^{w(1-\iota)} K^*_{\infty,j}$\\ 
\hline
\end{tabular}
\end{table}

\renewcommand{\arraystretch}{1.5}
\begin{table}[htbp!]
\caption{Almost sure asymptotic behavior of the inclusion probability. 
The real random variables $Z^*_\infty$ and $K^*_{\infty,j}$ are a.s.\ strictly positive.}
\label{table-P-j} 
\centering
\footnotesize
\begin{tabular}{l|l|l|l|l|}
$P_{t,j}$ & $\iota=1$ & $\frac{\beta}{w}<\iota <1$ &
$0<\iota=\frac{\beta}{w} <1$ & $\iota=0$ or $0 < \iota<\min\{\frac{\beta}{w},1\}$\\
\hline
$0=\beta<w$ & $\frac{t^{-(1-w)}}{\ln(t)}\frac{Z^{*}_\infty}{\alpha}$ &
$\frac{t^{-(1-w)}}{\ln(t)}\frac{1}{\alpha}Z_\infty^*$
&\mbox{not possible}  & $\frac{1}{t^{(1-w)}} w K^*_{\infty,j}$ \\ 
$0<\beta<w$ & $\frac{1}{t^{(1-w+\beta)}}\frac{Z^{*}_\infty\beta}{\alpha}$ &
$\frac{1}{t^{(1-w+\beta)}}\frac{\iota(w-\beta)}{(\iota w-\beta)}\frac{\beta}{\alpha}Z_\infty^*$
& $\frac{\ln(t)}{t^{(1-w+\beta)}}(1-\iota)w \iota\frac{\beta}{\alpha}Z_\infty^*$ & $\frac{1}{t^{(1-w(1-\iota))}}(1-\iota)w K^*_{\infty,j}$ \\ 
$0<\beta=w$ & $\frac{\ln(t)}{t}w^{2}$ & \mbox{not possible} & \mbox{not possible} & $\frac{1}{t^{(1-w(1-\iota))}}(1-\iota)w K^*_{\infty,j}$ \\ 
$0<w<\beta$ & $\frac{1}{t}\frac{w\beta}{\beta-w}$ & \mbox{not possible} & \mbox{not possible}& $\frac{1}{t^{(1-w(1-\iota))}}(1-\iota)w K^*_{\infty,j}$\\ 
\hline
\end{tabular}
\end{table}

\section{Preliminaries}\label{preliminaries}
Recall that 
$K_{t,j}=\sum_{n=1}^t X_{n,j}$ is the number of customers who test dish $j$ until time-step $t$ and  let $T_t$ be the number of dishes tested by customer $t$. We have 
\begin{equation*}
  \sum_{n=1}^t T_n
  =\sum_{n=1}^t (\sum_{j\in{\mathcal O}_{n-1}} X_{n,j} + N_n)
    =\sum_{j\in{\mathcal O}_t} K_{t,j}\,.
  \end{equation*}
Hence, setting 
$\overline{T}_t=\sum_{n=1}^t T_n/t$, we can write 
\begin{equation}\label{eq-Kmedio-e-Tmedio}
\overline{K}_t=\frac{\sum_{j\in{\mathcal O}_t} K_{t,j}}{D_t}=\frac{t}{D_t}\overline{T}_t
\end{equation}
and, setting $S_0=0$ and $S_t=\sum_{j\in {\mathcal O}_t} P_{t,j}$ for $t\geq 1$,  we have 
\begin{equation}\label{eq-S}
  \begin{split}
S_t&=\sum_{j\in {\mathcal O}_t} P_{t,j}=\frac{w(1-\iota)\sum_{j\in{\mathcal O}_t} K_{t,j}+
        w\iota \sum_{i\in {\mathcal O}_t} K_{t,i}}{\theta+t} \\
 &=
\frac{w\sum_{j\in{\mathcal O}_t} K_{t,j}}{\theta+t}
= \frac{w\sum_{n=1}^t T_n }{\theta+t}=w \frac{t}{\theta+t}\overline{T}_t\,.
  \end{split}
  \end{equation}
From \eqref{eq-S} we also obtain 
    \begin{equation}\label{eq-P-medio-2}
    \overline{P}_t=\frac{\sum_{j\in{\mathcal O}_t}P_{t,j}}{D_t}=
    \frac{w\sum_{j\in{\mathcal O}_t} K_{t,j}}{(\theta+t)D_t}=
    w\frac{\overline{K}_t}{\theta+t}
    \end{equation}
and so we can rewrite the inclusion probability $P_{t,j}$ as 
    \begin{equation}\label{eq-P-2}
    P_{t,j} = (1-\iota) w\frac{K_{t,j}}{\theta+t}+
        \iota \overline{P}_t\,.
    \end{equation} 
Moreover, the predictive mean  $Z_t=E[T_{t+1}|{\mathcal F}_t]$ 
is the sum of two terms, one related to the old dishes and the other related to the rate at which new dishes appear. 
Precisely, since $T_{t+1}=\sum_{j\in{\mathcal O}_t} X_{t+1,j}+N_{t+1}$, we have 
\begin{equation}\label{eq-Z}
Z_t=E[T_{t+1}|{\mathcal F}_t]= \sum_{j\in {\mathcal O}_t} P_{t,j}+\lambda_t=S_t+\lambda_t\,,
\end{equation}
where $\lambda_t=\alpha/(t+1)^{1-\beta}$ and $S_t$ is defined in \eqref{eq-S}. 
A simple calculation yields 
\begin{equation}\label{eq-diff-Z}
 Z_{t+1}-Z_{t}=  -\frac{(1-w)Z_t}{\theta+t+1}+\frac{w(T_{t+1}-Z_{t})}{\theta+t+1}
+\lambda_{t+1}-\lambda_{t}+\frac{\lambda_{t}}{\theta +t+1}
\end{equation}
so that we obtain
\begin{equation*}
 E\bigl[Z_{t+1}\mid\mathcal{F}_t\bigr] - Z_t
 =-\frac{(1-w)Z_t}{\theta+t+1}+\lambda_{t+1}-\lambda_{t}+\frac{\lambda_{t}}{\theta +t+1}.
\end{equation*}
Note also that we have 
\begin{equation}\label{eq-sp-quad-T}
E\bigl[T_{t+1}^2 \mid\mathcal{F}_t\bigr] = Z_t+Z_t^2-R_t\,,
\quad\mbox{where}\quad
R_t=\sum_{j\in {\mathcal O}_t} P_{t,j}^2,
\end{equation}
(and so $0\leq R_t\leq S_t\leq Z_t$), because
\begin{equation*}
\begin{split}
E\bigl[T_{t+1}^2 \mid\mathcal{F}_t\bigr]&=
E\bigl[(\sum_{j\in{\mathcal O}_t} X_{t+1,j}+N_{t+1})^2\mid \mathcal{F}_t\bigr]\\
&=\sum_{j\in{\mathcal O}_t} P_{t,j}+2\sum_{j<j'}P_{t,j}P_t(j')+\lambda_t+\lambda_t^2+2\lambda_t\sum_{j\in{\mathcal O}_t} P_{t,j}\\
&=\sum_{j\in{\mathcal O}_t} P_{t,j}+\lambda_t+
(\sum_{j\in{\mathcal O}_t} P_{t,j}+\lambda_t)^2-\sum_{j\in {\mathcal O}_t} P_{t,j}^2\\
&=Z_t+Z_t^2-R_t\,.
\end{split}
\end{equation*}

\indent Finally, we recall that $D_t$ denotes the number of dishes tested until time-step $t$ and its asymptotic behavior is described in the following 
subsection. 

\subsection{Asymptotic behavior of $D_t$} \label{subsec-D}
We recall that $D_t=card({\mathcal O}_t)=\sum_{n=1}^t N_n$, where the random variables $N_n$ are independent and 
Poisson-distributed with parameter $\lambda_{n-1}$ so that $D_t$ has Poisson distribution with parameter $\Lambda_t=\sum_{n=1}^t\lambda_{n-1}=\sum_{n=1}^t \alpha/n^{1-\beta}\sim \alpha \ln(t)$ when $\beta=0$ 
and $\sim (\alpha/\beta)t^\beta$ when  $\beta\in (0,1]$. (Remember that these are also the asymptotic behaviors of the mean and the variance of $D_t$, since they coincide with the parameter $\Lambda_t$.)\\

By standard martingale arguments and Kronecker's lemma, we have the following result regarding the first-order asymptotic behavior of $D_t$:
\begin{equation}\label{eq-D-as} 
\begin{cases}
\frac{D_t}{t^\beta}\stackrel{a.s.}\longrightarrow \frac{\alpha}{\beta}\quad&\mbox{when } 0<\beta\leq 1\,;\\
\frac{D_t}{\ln(t)}\stackrel{a.s.}\longrightarrow \alpha\quad&\mbox{when } \beta=0\,.
\end{cases}
\end{equation}

\begin{rem}[Central limit theorem]\label{rem-clt-D}
Setting
\begin{equation*}
\begin{split}
\lambda(\beta)&=\alpha\,\text{ if }\,\beta=0\quad\text{and}
\quad\lambda(\beta)=\frac{\alpha}{\beta}\,\,\text{ if }\,
\beta\in (0,1],
\\
a_t(\beta)&=\ln(t)\,\text{ if }\,\beta=0\quad\text{and}
\quad 
a_t(\beta)=t^\beta\,\text{ if }\,\beta\in (0,1]\,,
\end{split}
\end{equation*}
and recalling that 
\begin{equation*}
  \begin{split}
\sqrt{a_t(\beta)}\,\Bigl\{\tfrac{\Lambda_t}{a_t(\beta)}-
\lambda(\beta)\Bigr\}
=
\tfrac{\alpha\sum_{n=1}^t n^{\beta-1}-\lambda(\beta)a_t(\beta)}{\sqrt{a_t(\beta)}}
=
\tfrac{O(1)}{\sqrt{a_t(\beta)}}
\longrightarrow 0\,,
\end{split}
\end{equation*}
by a martingale central limit theorem (see \cite[Theorem~3.2, page 58]{HH}) 
applied to the martingale difference array 
$Y_{t,n}=(N_n-\lambda_{n-1})/\sqrt{a_t(\beta)}$, with $t\geq 1$ and $1\leq n\leq t$,
we can easily obtain that 
\begin{gather*}
\sqrt{a_t(\beta)}\,\Bigl(\frac{D_t}{a_t(\beta)}-\lambda(\beta)\Bigr)
\stackrel{stably}\longrightarrow\mathcal{N}\bigl(0,\,\lambda(\beta)\bigr)\,.
\end{gather*}
Moreover, by a functional central limit theorem for martingales (see \cite[Theorem~2.5]{dur-res}) applied to the same martingale difference array $(Y_{t,n})$, 
 we can get the functional version of the above limit result.
\end{rem}

Finally, the following remark provides an asymptotic bound for 
the difference $|D_t-\Lambda_t|$. 
\begin{rem}[Law of the iterated logarithm]\label{LIL-D}
\rm For each $\beta\in [0,1]$, we have with probability one that
\begin{equation*}\label{eq-lil}
\limsup_{t\rightarrow+\infty} \frac{|D_t-\Lambda_t|}{\sqrt{2 \Lambda_t\ln(\ln (\Lambda_t))}} = 1\,.
\end{equation*}
Its proof is based on the classical Kolmogorov's law of the iterated logarithm and standard arguments and it is collected in Section~\ref{app-LIL} of the supplementary. 
\end{rem}

\section{Asymptotic behavior of 
$\overline{T}_t$, 
$\overline{P}_t$ and $\overline{K}_t$}
\label{results-mean}

This section is devoted to deriving the asymptotic properties of the average quantities defined in~\eqref{eq:average_quantities}, namely
\[
\overline{P}_t = \frac{1}{D_t} \sum_{j \in \mathcal{O}_t} P_{t,j}, \quad
\overline{K}_t = \frac{1}{D_t} \sum_{j \in \mathcal{O}_t} K_{t,j}, \quad \text{and} \quad
\overline{T}_t = \frac{1}{t} \sum_{n=1}^t T_n\,.
\]
These processes are closely related to the process of predictive means
\[
Z_t = E[T_{t+1} \mid \mathcal{F}_t] = S_t + \lambda_t\,,
\]
as shown by the preliminary relations~\eqref{eq-Kmedio-e-Tmedio}, \eqref{eq-S} and \eqref{eq-P-medio-2} in Section~\ref{preliminaries}. Therefore, the theorems in this section focus on the asymptotic behavior of the process $(Z_t)_t$, from which the corresponding results for the averaged quantities can be directly derived (see Table~\ref{table-Z-T-P-K-medio} and Corollary~\ref{cor-clt}). 
Firstly, we provide the first-order asymptotic properties in Subsection~\ref{sec:first_order_asymp_average_quantities}, and then we turn to the second-order asymptotics in Subsection~\ref{sec:second_order_asymp_average_quantities}.

\subsection{First-order asymptotic results}\label{sec:first_order_asymp_average_quantities} 
 From the recursive dynamics of $(Z_t)_t$ in~\eqref{eq-diff-Z}, one might expect that its first-order asymptotic behavior can always be derived via standard stochastic approximation techniques. However, the situation is more subtle. Indeed, the dynamics is essentially governed by two competing contributions:
\begin{equation}\label{eq:heuristic_1}
    -\frac{(1-w)}{\theta+t+1} Z_t
\qquad\text{and}\qquad
(\lambda_{t+1}-\lambda_{t})+\frac{\lambda_t}{\theta+t+1}\,.
\end{equation}
Heuristically, the first term, driven by the parameter $(1-w)$ (that is the weight of the forcing input toward zero), acts as a linear drift toward zero (for $w<1$), suggesting a decay rate of order $t^{-(1-w)}$. The second term, induced by the decay (for $\beta<1$) of the expected number $\lambda_t$ of new dishes, suggests instead a rate of order $t^{-(1-\beta)}$.
The interplay between these two effects determines the asymptotic behavior of $Z_t$: when the two rates differ, the slower decay dominates; while, in the critical case $1-w = 1-\beta$, the two contributions asymptotically balance each other, leading to a cancellation of leading-order terms and resulting in a logarithmic correction, so that the decay becomes slightly slower than $t^{-(1-w)}$ (equivalently, $t^{-(1-\beta)}$). When $\max\{w,\beta\}=1$, the process $(Z_t)_t$ does not converge to zero: it converges to a random variable when $\beta<1$, to a constant when $w<1$ and it diverges to $+\infty$ in the case $w=\beta=1$. These facts in the case $w=1$ are consistent with the results known for the standard IBP of \cite{TG} (see~\cite{BCPR-IBP}).
\\
\indent It is worth emphasizing that, once the processes are rescaled according to the correct rate, depending on which term in~\eqref{eq:heuristic_1} dominates, the proofs become substantially different in the two regimes.
Determining the dominant term depends on the values of $w$ and $\beta$, which requires splitting the analysis into two separate theorems: Theorem~\ref{th-Z-T-medio} for the case $w > \beta$, and Theorem~\ref{th-Z-T-medio-altro-caso} for the case $w \le \beta$.
\\
\indent In the first case, when $w > \beta$, the second term in~\eqref{eq:heuristic_1} behaves essentially as a remainder, and thus $(Z_t)_t$ becomes a non-negative almost super-martingale and, 
besides proving the convergence of $t^{1-w}Z_t$ (which follows from a classical convergence result for non-negative almost super-martingales), we will prove (by a non  classical technique) that the corresponding limit random variable is strictly positive, concluding that $t^{1-w}$ is the exact scaling factor. 
\\
\indent In the second case, when $w \le \beta$, the second term in~\eqref{eq:heuristic_1} can not be neglected and we obtain 
a deterministic limit for the suitable rescaled process $(Z_t)_t$ via stochastic approximation.
 In particular, if $\beta = 1$, the sequence $(\lambda_t)_t$ does not decay, causing the failure of the convergence to zero for $Z_t$. 
Finally, when $w = \beta = 1$, the additional absence of drift leads to the divergence $Z_t \to +\infty$.

The proofs of all the results presented in this section are collected in Subsection~\ref{sec:proofs_first_order_average} of the supplementary and are essentially based on Theorem~\ref{thm:app-cor}, Theorem~\ref{th-general} and Lemma~\ref{lem-unif-int-new}. In particular, the last one is a non immediate technical lemma needed to control the moments of the unbounded random variables $Z_t$. 
Also relation~\eqref{eq-sp-quad-T}  plays an important role. 
\\

We first present the convergence result in the case $w>\beta$.

\begin{thm}[Weak forcing input toward zero]\label{th-Z-T-medio} 
Let $\beta<w$ (i.e. $(1-w)<(1-\beta)$). 
Then we have 
\begin{equation*}
t^{1-w} Z_t\stackrel{a.s.}\longrightarrow Z^*_\infty\,,
\end{equation*}
where $Z^*_\infty$ is a finite and strictly positive random variable.
\end{thm}
From Theorem~\ref{th-Z-T-medio} and using the preliminary relations~\eqref{eq-Kmedio-e-Tmedio}, \eqref{eq-S} and \eqref{eq-P-medio-2}, we can easily derive that
\begin{equation*}
t^{1-w}\overline{T}_t\stackrel{a.s.}\longrightarrow \frac{1}{w}Z^*_\infty,
\quad
t^{1-w}\sum_{j\in{\mathcal O}_t}P_{t,j}
\stackrel{a.s.}\longrightarrow
Z^*_\infty,
\quad
\frac{1}{t^{w}}\sum_{j\in{\mathcal O}_t}K_{t,j}
\stackrel{a.s.}\longrightarrow
\frac{1}{w} Z^*_\infty\,,
\end{equation*}
that along with the behavior of $D_t$ described in~\eqref{eq-D-as}  leads to the results for $\overline{P}_t$ and $\overline{K}_t$ reported in Table~\ref{table-Z-T-P-K-medio}.
\\

\indent We now state the result in the case $w\leq \beta$.
\begin{thm}[Strong forcing input toward zero and Critical case] \label{th-Z-T-medio-altro-caso}
Let $w\leq \beta$ (i.e.~$(1-w)\geq (1-\beta)$). Then we have 
\begin{equation*}
\begin{cases}
t^{1-\beta}\,Z_t\stackrel{a.s.}\longrightarrow 
\frac{\alpha\beta}{\beta-w}\quad &\mbox{when } w<\beta\,;\\
\frac{t^{1-\beta}}{\ln(t)}\,Z_t\stackrel{a.s.}\longrightarrow
\alpha\beta\quad &\mbox{when } w=\beta\,.
\end{cases}
\end{equation*}
\end{thm}
From Theorem~\ref{th-Z-T-medio-altro-caso} and using the preliminary relations~\eqref{eq-Kmedio-e-Tmedio}, \eqref{eq-S} and \eqref{eq-P-medio-2}, we can easily derive that: when $w<\beta$
\begin{equation*}
t^{1-\beta}\overline{T}_t\stackrel{a.s.}\longrightarrow 
\frac{\alpha}{\beta-w},
\quad
t^{1-\beta}\sum_{j\in{\mathcal O}_t}P_{t,j}
\stackrel{a.s.}\longrightarrow
\frac{\alpha w}{\beta-w},
\quad
\frac{1}{t^{\beta}}\sum_{j\in{\mathcal O}_t}K_{t,j}
\stackrel{a.s.}\longrightarrow
\frac{\alpha}{\beta-w}\,,
\end{equation*}
while when $w=\beta$
\begin{equation*}
\frac{t^{1-\beta}}{\ln(t)}\overline{T}_t\stackrel{a.s.}\longrightarrow 
\alpha,
\quad
\frac{t^{1-\beta}}{\ln(t)}\sum_{j\in{\mathcal O}_t}P_{t,j}
\stackrel{a.s.}\longrightarrow
\alpha w,
\quad
\frac{1}{t^{\beta}\ln(t)}\sum_{j\in{\mathcal O}_t}K_{t,j}
\stackrel{a.s.}\longrightarrow
\alpha\,,
\end{equation*}
that along with the behavior of $D_t$ described in~\eqref{eq-D-as}  leads to the results of $\overline{P}_t$ and $\overline{K}_t$ reported in Table~\ref{table-Z-T-P-K-medio}.

\begin{rem}[Convergence in $L^p$]
By Lemma~\ref{lem-unif-int-new}, we obtain that the a.s. convergence of 
 the suitably rescaled sequence $Z_t$ (and the suitably rescaled sequence $\overline{T}_t$) 
stated in Theorem~\ref{th-Z-T-medio} and Theorem~\ref{th-Z-T-medio-altro-caso}
 is also in $L^p$ for each  $p\in [1,+\infty)$. Moreover, by Fatou's lemma, the limit random variable $Z^*_\infty$ that arises in the case $w>\beta$ (see Theorem~\ref{th-Z-T-medio}) satisfies 
  $E\big[e^{z Z^*_\infty}\big]<+\infty$ for some $z>0$.
\end{rem}

\subsection{Second-order asymptotic in the case when the forcing input toward zero is weak 
($\beta<w$)}
\label{sec:second_order_asymp_average_quantities}

In this section, we establish some Central Limit Theorems (CLTs) for the random variable \( Z^*_\infty > 0 \), introduced in Theorem~\ref{th-Z-T-medio} 
and related to the a.s. limits of the rescaled predictive mean $Z_t$ and of 
the rescaled averaged quantities \( \overline{T}_t \), \( \overline{P}_t \) and \( \overline{K}_t \) in the regime \( w > \beta \).
 These results are of particular interest when $w=1$, as they enable the estimation of the limit predictive mean \( Z^*_\infty \) 
 via the construction of confidence intervals based on the observable empirical means \( \overline{T}_t \). 
\\
\indent As discussed in Section~\ref{sec:first_order_asymp_average_quantities}, the dynamics of \( (Z_t)_t \) in~\eqref{eq-diff-Z} is essentially governed by the two competing contributions in~\eqref{eq:heuristic_1}, one driven by the weight \( (1-w) \)  of the forcing input toward zero and the other by \( (1-\beta) \).  
In the case \( w > \beta \) considered here, 
the first term determines the decay rate \( t^{-(1-w)} \) of \( Z_t \), while the second one plays no role in the first-order behavior, acting essentially as a remainder term.  
By contrast, in the second-order asymptotic behavior (see Theorem~\ref{thm:second_order-Z}), the parameter $\beta$ plays an important role even in the case $w>\beta$. Indeed,
by~\eqref{eq-diff-Z} we can observe that the convergence of \( t^{1-w} Z_t \) towards \( Z^*_\infty \) is governed by two competing contributions:
\begin{equation*}
t^{1-w}\frac{\Delta M_{t+1}}{\theta+t+1}
\qquad\text{and}\qquad
  t^{1-w}\frac{\lambda_t}{\theta+t+1}\,,
\end{equation*}
where $\Delta M_{t+1}=w(T_{t+1}-Z_t)$ is an unbounded martingale difference. 
Heuristically, the first term suggests a martingale CLT with scaling \( t^{-w/2} \), leading to a zero-mean Gaussian limit kernel. 
The second term is deterministic and exhibits a decay of order \( t^{-(w-\beta)} \).
The interplay between these two contributions determines the second-order asymptotic behavior of \( t^{1-w}Z_t \). 
Specifically, taking into account~\eqref{eq:compact_tcl_Z_t}: if \( (w - \beta) > w/2 \) (i.e., \( \beta < w/2 \)), 
   the second term is negligible also at the second order, and the asymptotics are entirely driven by the martingale term which leads to the convergence of  $t^{w/2}(t^{1-w}Z_t-Z_\infty^*)$ 
   to a Gaussian kernel with zero mean; 
 if \( (w - \beta) = w/2 \) (i.e., \( \beta = w/2 \)), 
   we still obtain that  $t^{w/2}(t^{1-w}Z_t-Z_\infty^*)$ converges toward a Gaussian kernel, but the second term contributes with a constant shift in the mean equal to \( -\alpha \);
   if \( 0 < w - \beta < w/2 \) (i.e., \( w/2 < \beta < w \)),  
   the deterministic term dominates so that we obtain $(t^{1-w}Z_t-Z_\infty^*)\stackrel{P}\sim -Ct^{-(w-\beta)}$, with a suitable constant $C>0$, and the Gaussian fluctuations arise if we consider 
   the quantity $(t^{1-w}Z_t-Z_\infty^* + Ct^{-(w-\beta)})$ multiplied by the scaling factor $t^{w/2}$. 
It may appear surprising that the asymptotic contribution of the second term is strictly negative in the last two cases. 
However, this is fully consistent with the fact that, in the dynamics~\eqref{eq-diff-Z}, the second term represents a strictly positive contribution.
Accordingly, it affects the second-order asymptotics by increasing the value of \( Z^{*}_\infty \), and hence by reducing the deviation \( (t^{1-w} Z_t - Z^{*}_\infty) \), relative to a dynamics with a random zero-mean remainder term, as typically encountered in similar settings. Moreover, it is worth emphasizing that the variance of the Gaussian limit kernel is always random. Moreover, its random component coincides with \( Z^*_\infty \), except in the degenerate case of no-interaction (\( \iota = 0 \)) and no-forcing input toward zero (\( w = 1 \)). As we will see in Section~\ref{results-dish-specific}, this is also the only regime in which the dish-specific inclusion probabilities \( P_{t,j} \) do not converge a.s. to zero (see Table~\ref{table-P-j}).\\
\indent Finally, note that  in Theorem~\ref{thm:second_order-Z} the convergence holds true also in the sense of the a.s. conditional convergence, which is a form of convergence introduced in~\cite{crimaldi-2009} and, subsequently, employed also by others in the urn model  and Bayesian literature, e.g.~\cite{fortini-petrone-2020, Zhang-2014}. Specifically, in~\cite{fortini-petrone-2020}, it is shown how to use a CLT with this form of convergence in order to construct a credible interval in a Bayesian framework. In~\cite{Zhang-2014}, instead, it is illustrated how to use a CLT with this form of convergence in order to prove that 
the limit random variable has no point masses.

The proofs of all the results presented in this section are collected in Subsection~\ref{sec:proofs_second_order_average} of the supplementary and are essentially based on the general Theorem~\ref{th-CLT-general} and on the technical  Lemma~\ref{lem-unif-int-new} and Lemma~\ref{lem-RsuS}.

\begin{thm}[CLT for the case of weak forcing input toward zero]\label{thm:second_order-Z}
Assume $\beta<w$, i.e. $(1-w)<(1-\beta)$, and let
\[
\Sigma_\infty = 
\begin{cases}
    Z^*_\infty - \Rstar_\infty >0 & \text{if $w=1$ and $\iota = 0$}\,;
    \\
    Z^*_\infty>0 & \text{otherwise}\,.
\end{cases} 
\]
where $Z^*_\infty$ is the a.s.\ limit of $t^{1-w}Z_t$ introduced in Theorem~\ref{th-Z-T-medio} and $\Rstar_\infty$ is the a.s.\ limit of $R_t$ as provided~in Lemma~\ref{lem-RsuS}. Then, we have
\begin{equation}\label{eq:compact_tcl_Z_t}
 {t}^{w/2}(t^{1-w}Z_t+\tfrac{\alpha\beta}{w-\beta}t^{-(w-\beta)}-Z^*_\infty)\stackrel{stably}\longrightarrow 
    {\mathcal N}(0,w\Sigma_\infty)\,,
\end{equation}
where this convergence also holds true in the sense of the a.s.\ conditional convergence with respect to the filtration~$\mathcal F$. Moreover, 
as a consequence, when $(w-\beta)<w/2$ (i.e. $w/2<\beta<w$), we have
\begin{equation*}
 {t}^{w-\beta}(t^{1-w}Z_t-Z^*_\infty)\stackrel{P}\longrightarrow -\tfrac{\alpha \beta}{w-\beta}\,. 
\end{equation*}
\end{thm}

\begin{rem}[No point-masses for the limit random variable] 
\rm  In the case $w<1$ or $\iota>0$, with a similar argument used in~\cite{Zhang-2014} (see also~the proof of \cite[Theorem~2.5]{cri-dai-lou-min}) 
based on the a.s. conditional convergence toward a Gaussian kernel, we can obtain from~\eqref{eq:compact_tcl_Z_t} that, for any $\beta< w$,  $P(Z_\infty^*=z)=0$ for each $z>0$.
\end{rem}

In the next result, we provide the second order asymptotics for the averaged quantities: the one for $\overline{T}_t$ is a direct consequence of \eqref{eq:compact_tcl_Z_t}, while the ones for $\overline{P}_t$ and $\overline{K}_t$ are also related to the second-order asymptotic of the number $D_t$ of distinct tested dishes (observed features). 

\begin{cor}\label{cor-clt} 
If $\beta<w$, then we have 
\begin{equation}\label{eq-clt-compact-Tmedio}
t^{w/2}\Big(t^{1-w}\overline{T}_t+\tfrac{\alpha}{(w-\beta)}t^{-(w-\beta)}-\tfrac{1}{w}Z^*_\infty\Big)
\stackrel{stably}\longrightarrow 
{\mathcal N}(0, \Sigma_\infty/w)\,
\end{equation}
(where the convergence also holds true in the sense of the a.s. conditional convergence with respect to $\mathcal F$) and, as a consequence, when $w-\beta<w/2$ (i.e. $w/2<\beta<w$), we have 
\indent 
\begin{equation*}
t^{w-\beta}\Big(t^{1-w}\overline{T}_t-\tfrac{1}{w}Z^*_\infty\Big)\stackrel{P}\longrightarrow -\tfrac{\alpha}{(w-\beta)}\,.
\end{equation*}
Moreover, using the notation introduced in Remark~\ref{rem-clt-D}, we have 
\begin{equation}\label{clt-Pmedio}
\sqrt{a_t(\beta)}\left(t^{1-w}a_t(\beta)\overline{P}_t+\tfrac{\alpha w}{\lambda(\beta)(w-\beta)}t^{-(w-\beta)}-\tfrac{1}{\lambda(\beta)}Z^*_\infty\right)\stackrel{stably}\longrightarrow 
\mathcal{N}(0,\tfrac{(Z_\infty^*)^2}{\lambda(\beta)^3})
\end{equation}
and, as a consequence, when $(w-\beta)<\beta/2$ (i.e. $(2/3)w<\beta<w$), we have 
\begin{equation*}
t^{w-\beta}\left(t^{1-w+\beta}\,\overline{P}_t-\tfrac{\beta}{\alpha}Z^*_\infty\right) 
\stackrel{P}\longrightarrow -\tfrac{\beta w}{w-\beta}\,.
\end{equation*}
Finally, we have 
\begin{equation}\label{clt-Kmedio}
\sqrt{a_t(\beta)}\left(\tfrac{a_t(\beta)}{t^w}\overline{K}_t+\tfrac{\alpha }{\lambda(\beta)(w-\beta)}t^{-(w-\beta)}-\tfrac{1}{\lambda(\beta)w}Z^*_\infty \right)
\stackrel{stably}\longrightarrow \mathcal{N}\left(0,\tfrac{(Z_\infty^*)^2}{\lambda(\beta)^3 w^2}\right)
\end{equation}
and, as a consequence, when $(w-\beta)<\beta/2$  (i.e. $(2/3)w<\beta<w$), we have 
$$
t^{w-\beta}\left( \tfrac{1}{t^{w-\beta}}\overline{K}_t-\tfrac{\beta}{\alpha w}Z_\infty^* \right)
\stackrel{P}\longrightarrow -\tfrac{\beta}{w-\beta} \,.
$$
\end{cor}
Note that the above limit random  variance $\Sigma_\infty/w$ can be estimated by the observable process $(T_t)_t$. Indeed, when $w<1$ or $\iota>0$, it can be seen as the a.s. limit of $t^{1-w}\overline{T}_t$; while, when $w=1$ and $\iota=0$, by~\eqref{eq-sp-quad-T} and Lemma~\ref{lemma-kro}, it coincides with the a.s. limit of $\sum_{n=1}^t T_n^2/t-(\overline{T}_t)^2$.

\begin{rem}[Comparison with the standard IBP of~\cite{TG}] 
\rm From \cite[Theorem~8]{BCPR-IBP}, using our notation, we get that the standard IBP of \cite{TG} satisfies for every $\beta<1$
$$
\sqrt{t}\left(Z_t-Z_\infty^*\right)\stackrel{stably}\longrightarrow {\mathcal N}(0, \Sigma_\infty)\quad\mbox{and}\quad
\sqrt{t}\left(\overline{T}_t-Z_\infty^*\right)\stackrel{stably}\longrightarrow {\mathcal N}(0, \Sigma_\infty)\,,
$$
with $\Sigma_\infty>0$. This is obviously a difference with our model in the case  $w<1$, but, importantly,  it is a difference 
with our model even when $w=1$ and $\iota=0$. Indeed, as illustrated above, we have a non-centred Gaussian kernel in the case $\beta=1/2$ and a convergence in probability  with a rate depending on $(1-\beta)$ to a suitable constant when $1/2<\beta<1$. This difference is due to the fact that, also in the case $w=1$ and $\iota=0$, the inclusion probabilities $P_{t,j}$ and the parameter $\lambda_t$ in our model differ from the ones in the standard IBP, where the inclusion probabilities and the parameter $\lambda_t$ are defined so that the sequence $(Z_t)$ of the predictive means is a martingale. 
\end{rem}

\section{Asymptotic behavior of $P_{t,j}$ and $K_{t,j}$}
\label{results-dish-specific}

In this section, we aim to derive the asymptotic behavior of the quantities $K_{t,j}$ and $P_{t,j}$ related to each individual observed 
dish $j$. To this end, let $\tau_j$ be the first time at which a dish $j$ appears. 
 The observed dishes are the ones with $\tau_j<+\infty$. 
Then, define the process $B_{t,j}=K_{t,j}/(\theta+t)$ 
(which is equal to zero for $t< \tau_j$ and $>0$ for $t\geq \tau_j$)
whose asymptotic behavior will be essential to establish the 
one of $P_{t,j}$ and $K_{t,j}$. For any observed dish $j$,  
for $t\geq \tau_j$, we have 
\begin{equation*}
\begin{split}
B_{t+1,j}-B_{t,j}&=
\frac{(K_{t,j}+\Delta K_{t+1,j})}{(\theta+t+1)}-\frac{K_{t,j}}{(\theta+t)}\\
&=-\frac{1}{\theta+t+1} B_{t,j} +\frac{1}{\theta+t+1}\Delta K_{t+1,j}\,,
\end{split}
\end{equation*}
that is 
\begin{equation}\label{eq:dynamics_B_tj}
B_{t+1,j}=
\left(1-\frac{1}{\theta+t+1}\right) B_{t,j} +\frac{1}{\theta+t+1}\Delta K_{t+1,j}\,,
\end{equation}
where 
\begin{equation*}
\begin{split}
E[\Delta K_{t+1,j}\mid {\mathcal F}_t]= P_{t,j}
&= (1-\iota) w B_{t,j} +\iota \overline{P}_t.
\end{split}
\end{equation*}
A simple calculation yields to
\begin{equation}\label{eq:dynamics_B_tj_2}
B_{t+1,j}-B_{t,j}=
-\frac{(1-(1-\iota)w)}{\theta+t+1}B_{t,j}
+\frac{\iota\overline{P}_t}{\theta+t+1}+\frac{\Delta M_{t+1,j}}{\theta+t+1}\,,
\end{equation}
where $\Delta M_{t+1,j}=(\Delta K_{t+1,j}-P_{t,j})$ is a martingale difference.\\

Firstly, we establish the first-order asymptotic properties in Subsection~\ref{sec:first_order_asymp_dish_specific_quantities}, and then we turn to the second-order asymptotics in Subsection~\ref{sec:second_order_asymp_dish_specific_quantities}.

\subsection{First-order asymptotic results}\label{sec:first_order_asymp_dish_specific_quantities}

The dynamics of $B_{t,j}$ in~\eqref{eq:dynamics_B_tj_2} is essentially governed by two competing contributions:
\begin{equation}\label{eq:heuristic_2}
-\frac{(1-(1-\iota)w)}{\theta+t+1}B_{t,j}
\qquad\text{and}\qquad
\frac{\iota\overline{P}_t}{\theta+t+1}\,.
\end{equation}
Heuristically, the first term, driven by the parameter $(1-(1-\iota)w)$, acts as a linear drift toward zero (for $(1-\iota)w<1$), suggesting a decay rate of order $t^{-(1-(1-\iota)w)}$. 
The second term, proportional to the interaction intensity~$\iota$ and to the random averaged quantity~$\overline{P}_t$, suggests instead a rate as derived by means of Theorems~\ref{th-Z-T-medio} and Theorem~\ref{th-Z-T-medio-altro-caso}, and 
described in Table~\ref{table-Z-T-P-K-medio}. 
The interplay between these two effects determines the asymptotic behavior of $B_{t,j}$ (and hence of $K_{t,j}$ and $P_{t,j}$): when the two rates differ, the slower decay dominates; while, in the critical case, when the two contributions asymptotically balance each other, the leading-order terms disappear and a logarithmic correction appears.
Once the processes are rescaled according to the right rate, depending on which term in~\eqref{eq:heuristic_2} dominates, the proofs become substantially different in the various regimes.
the dominant term depends on the values of $\iota$ with respect to the ratio $\beta/w$ and,  
to better highlight how these results change with the level of interaction $\iota$, we will categorize the results into three distinct cases:
\begin{itemize}
    \item[(a)] \emph{Low interaction}: where the interaction parameter $\iota$ is either zero (no interaction) or smaller than the threshold $\beta/w$;
    \item[(b)] \emph{High interaction}: where the interaction parameter $\iota$ equals one (maximum interaction) or it is smaller than one but exceeds the threshold $\beta/w$;
    \item[(c)] \emph{Critical regime}: where the interaction parameter $\iota$ exactly equals the threshold $\beta/w$.
\end{itemize}
In the first case (a), the second term in~\eqref{eq:heuristic_2} behaves essentially as a remainder, and thus $(B_{t,j})_t$ becomes a non-negative almost super-martingale and 
we will prove (by the non classical general result~Theorem~\ref{th-general}) that its right scaling factor is $t^{(1-(1-\iota)w)}$. In the other cases (b)-(c), the second term in~\eqref{eq:heuristic_2} is of the same order as the first one, and we obtain the desired results  
via a general stochastic approximation theorem with a random ``attractor'' (see Theorem~\ref{thm:app-cor}).
\\
\indent The fact that the presence of interaction ($\iota > 0$) causes each $P_{t,j}$ to approach zero (even when the weight $(1-w)$ of the forcing input toward zero is null) can be understood by recalling that, according to~\eqref{eq-P-2}, $P_{t,j}$ is a linear combination involving the term $\iota \overline{P}_t$, which always tends to zero (see Table~\ref{table-Z-T-P-K-medio}).  Hence, a higher value of $\iota$ amplifies the effect of $\overline{P}_t$, accelerating the convergence of $P_{t,j}$ to zero.

\subsubsection{Statement of the results}
In this section we are going to provide the first-order asymptotic results about the dish-specific quantities $K_{t,j}$ and $P_{t,j}$.  
The first theorem holds in the case when the interaction is low.

\begin{thm}[Case (a): Low interaction] \label{th-P-K-low_interaction}
When 
\begin{itemize}
    \item[(1)] $\iota=0$ (no interaction at all) or
    \item[(2)] $0<\iota<\min\{\beta/w,1\}$,
\end{itemize}
for each observed dish $j$, we have 
$$
\frac{K_{t,j}}{t^{(1-\iota)w}}\stackrel{a.s.}\longrightarrow K^*_{\infty,j}
$$
and 
$$
t^{1-(1-\iota)w}\,P_{t,j}\stackrel{a.s.}\longrightarrow (1-\iota)w\, K^*_{\infty,j}\,,
$$
where $K^*_{\infty,j}$ is a finite and strictly positive random variable.
\end{thm}
This result for $w=1$ and $\iota=0$ is coherent with the one for the standard IBP of \cite{TG}. 
\\

\indent The next theorem holds in the case when the interaction is high.

\begin{thm}[Case (b): High interaction] \label{th-P-K-high_interaction}
When 
\begin{itemize}
\item[(1)] $\iota=1$, for each observed dish~$j$, we have 
$$P_{t,j}=\overline{P}_t\ \ \forall\, t\geq1 \qquad\mbox{and}\qquad \frac{K_{t,j}}{\sum_{n=1}^t\overline{P}_{n-1}}\stackrel{a.s.}\longrightarrow 1\,.$$
\item[(2)] $0<\beta/w<\iota<1$, for each observed dish~$j$, we have 
$$
\frac{K_{t,j}}{t^{w-\beta}}\stackrel{a.s.}\longrightarrow 
K_\infty= \frac{\iota}{(\iota w - \beta)} \frac{\beta}{\alpha}\, Z^*_\infty
$$
and
$$
t^{1-w+\beta}\,P_{t,j}\stackrel{a.s.}\longrightarrow 
\frac{\iota(w-\beta)}{(\iota w-\beta)}\frac{\beta}{\alpha}\,Z_\infty^*\,,
$$
where $Z_\infty^*>0$ a.s.\ is defined in Theorem~\ref{th-Z-T-medio};

\item[(3)] $0=\beta/w<\iota<1$ (i.e. $\beta=0$ and $0<\iota<1$), for each observed dish~$j$, we have 
$$
\frac{K_{t,j}}{t^{w}/\ln(t)}\stackrel{a.s.}\longrightarrow 
K_\infty= \frac{1}{\alpha w}\, Z^*_\infty
$$
and
$$
t^{1-w}\ln(t)\,P_{t,j}\stackrel{a.s.}\longrightarrow 
\frac{1}{\alpha}\,Z_\infty^*\,.
$$
\end{itemize}
\end{thm}

Finally, the last theorem holds in the critical case.

\begin{thm}[Case (c): Critical regime] \label{th-P-K-medium_interaction}
When $0<\iota=\beta/w <1$, for each observed dish~$j$, we have
$$
\frac{K_{t,j}}{t^{w-\beta}\ln(t)}\stackrel{a.s.}\longrightarrow 
K_\infty= \iota \frac{\beta}{\alpha}\, Z^*_\infty
$$
and
$$
\frac{t^{1-w+\beta}}{\ln(t)}\,P_{t,j}\stackrel{a.s.}\longrightarrow 
(1-\iota)w \iota\frac{\beta}{\alpha}\,Z_\infty^*\,.
$$
\end{thm}

\subsubsection{Proofs}

\begin{proof}[Proof of Theorem~\ref{th-P-K-low_interaction} (Low interaction)]
Considering the process $B_{t,j}=K_{t,j}/(\theta+t)$ and 
denoting $B^{*}_{t,j}=\zeta_{t}(1-(1-\iota)w)B_{t,j}$, from the dynamics~\eqref{eq:dynamics_B_tj_2} we obtain
 \begin{equation}\label{eq-dyn-B-star}
  B^{*}_{t+1,j}=\Big(1+O\big(\tfrac{1}{t^2}\big)\Big)B^{*}_{t,j}
  +\tfrac{\zeta_{t+1}(1-(1-\iota)w)}{\theta+t+1} (\Delta M_{t+1,j} + \iota \overline{P}_t) 
  \end{equation}
and so, by the asymptotics for $\overline{P}_t$ (see Table~\ref{table-Z-T-P-K-medio}) 
and the classical Theorem~\ref{th-almost-supermart}, the process $(B^{*}_{t,j})_t$ is a non-negative
almost super-martingale that converges a.s.\ to a finite r.v. $B^{*}_\infty$. However, this is not enough for our scopes. 
Indeed, in that way, we do not know that  such random limit $B^{*}_\infty$ is a.s.\ strictly positive, and this is crucial for proving 
the right rate of convergence. For this reason, we are going to apply the more sophisticated 
Theorem~\ref{th-general} (case i) together with Remark~\ref{rem-general} with $\tau=\tau_j$) to dynamics~\eqref{eq:dynamics_B_tj} with $X_t=B_{t,j}$, $Y_{t+1}=\Delta K_{t+1,j}\in \{0,1\}$, $\delta=(1-\iota ) w\in (0, 1]$. To this purpose, we note that $\rho_t=0$ when $\iota=0$;
while $\rho_t=\iota \overline{P}_t > 0$ when $\iota>0$. To handle the case $\iota>0$, from Table~\ref{table-Z-T-P-K-medio} we get $\overline{P}_t=O(1/t^{1-w+\beta})$ when $0<\beta<w\leq 1$, $\overline{P}_t=O(\ln(t)/t)$ when $0<\beta=w\leq 1$ and $\overline{P}_t=O(1/t)$ when $0<w<\beta\leq 1$ (the case $\beta=0$ is not possible with $\iota>0$ because $\beta>\iota w$ in the case of low interaction). 
Therefore, setting $\delta=(1-\iota)w$, we have 
$$
\sum_t \frac{\rho_t}{(\theta+t)^{\delta-\epsilon}} 
<+\infty\;\mbox{a.s.\ for some }0<\epsilon<\delta\,,
$$
provided that  $\iota=0$ or $0<\iota <\beta/w$.
Moreover, we trivially (since $Y_{t+1}\in \{0,1\}$) have 
$$
E[Y_{t+1}^2\mid {\mathcal F}_t]\leq E[Y_{t+1}\mid {\mathcal F}_t]\,.
$$
 Hence, under our conditions, we find 
$$
t^{1-(1-\iota)w}\, B_{t,j}\stackrel{a.s.}\longrightarrow B^*_{\infty,j}\in (0,+\infty)
$$
and so 
 the first statement of the theorem follows with $K^*_{\infty,j}=B^*_{\infty,j}$.  
Finally, it follows, by \eqref{eq-P-2}, that  
$$
t^{1-w(1-\iota)}\,P_{t,j}\stackrel{a.s.}\longrightarrow (1-\iota)w\,B^*_{\infty,j}=(1-\iota)w\,K^*_{\infty,j}\in (0,+\infty)\,,
$$
which is the second statement of the theorem.
\end{proof}

\begin{proof}[Proof of Theorem~\ref{th-P-K-high_interaction} (High interaction)]
We will prove the two cases separately.\\ 
\indent {\em Proof of case (1):} 
From \eqref{eq-P-2} we immediately get $P_{t,j}=\overline{P}_t$ for each $j\in{\mathcal O}_t$ and so the asymptotic behavior of $P_{t,j}$ is described in Table~\ref{table-Z-T-P-K-medio}, 
while, by Lemma~\ref{williams-lemma},  for $K_{t,j}$ we have $K_{t,j}/\sum_{n=1}^t\overline{P}_{n-1}\stackrel{a.s.}\longrightarrow 1$ because $K_{t,j}=\sum_{n=1}^t X_{n,j}$
with $E[X_{n+1,j}\mid{\mathcal F}_{n}]=P_{n,j}=\overline{P}_n$ and $\sum_n \overline{P}_n=+\infty$ a.s..
\\
    \indent {\em Proof of case (2):} Recall that we have $B_{t,j}=K_{t,j}/(\theta+t)$ with dynamics~\eqref{eq:dynamics_B_tj_2}:
$$
B_{t+1,j}=B_{t,j}-\frac{(1-(1-\iota)w)B_{t,j}-\iota\overline{P}_t}{\theta+t+1}
+\frac{\Delta M_{t+1,j}}{\theta+t+1}\,,
$$
where
$$\Delta M_{t+1,j}=\Delta K_{t+1,j}-P_{t,j}=\Delta K_{t+1,j}-(1-\iota)wB_{t,j}-\iota\overline{P}_t\,.$$
When $\iota w>\beta>0$, the  term $\iota \overline{P}_t$ is not negligible as in Theorem~\ref{th-P-K-low_interaction} and so we have to study the dynamics of $B_{t,j}$ paired with the one of $\overline{P}_t$. From Table~\ref{table-Z-T-P-K-medio} we know that, since $w>\beta$, we have $t^{1-w+\beta}\overline{P}_t\stackrel{a.s.}\rightarrow {P}^*_{\infty}>0$.  
Now, let $B_{t,j}^*=\zeta_t(1-w+\beta)B_{t,j}$, 
$\overline{P}_t^*=\zeta_t(1-w+\beta)\overline{P}_t$ and
$\Delta M_{t+1,j}^*=\zeta_{t+1}(1-w+\beta)\Delta M_{t+1,j}$. 
Then, using \eqref{eq-zeta-1}, we obtain 
\begin{align*}
B_{t+1,j}^*&=\Big(B_{t,j}^{*}-\tfrac{(1-(1-\iota)w)B_{t,j}^{*}-\iota\overline{P}_t^{*}}{\theta+t+1}\Big)
\tfrac{\zeta_{t+1}(1-w+\beta)}{\zeta_t(1-w+\beta)} + 
\tfrac{\Delta M_{t+1,j}^*}{\theta+t+1}\\
&=
\Big(B_{t,j}^{*}-\tfrac{(1-(1-\iota)w)B_{t,j}^{*}-\iota\overline{P}_t^{*}}{\theta+t+1}\Big)
\Big(1+\tfrac{1-w+\beta}{\theta+t+1}\Big) + 
\tfrac{\Delta M_{t+1,j}^*}{\theta+t+1}\\
&=B_{t,j}^{*}
-\tfrac{(\iota w-\beta)B_{t,j}^{*}-\iota\overline{P}_t^{*}}{\theta+t+1}+
\tfrac{\Delta M_{t+1,j}^*}{\theta+t+1}+\rho_t\,,
\end{align*}
with $\rho_t=O\big(\frac{B_{t,j}+\overline{P}_t}{t^{1+w-\beta}}\big)$ being ${\mathcal F}_t$-measurable.  
This suggests that $(B^{*}_t)_t$ is not a non-negative almost super-martingale, as in Theorem~\ref{th-P-K-low_interaction}, but instead it evolves according to a stochastic approximation dynamics, which should convergence to~$\iota {P}_\infty^{*}/(\iota w-\beta)$. For proving this convergence, 
 we are going to apply~Theorem~\ref{thm:app-cor} with $X_t=B^*_t$, together with Lemma~\ref{lem:alternative_to_epsilon}. 
 To this end, we observe that, by \eqref{eq-sp-quad-T} and since $0\leq P_{t,j}\leq 1$,
\begin{equation*}
\begin{split}
E[(\Delta M^*_{t+1,j})^2\mid{\mathcal F}_t]&=
\zeta(1-w+\beta)^2P_{t,j}(1-P_{t,j})\\
&\leq \zeta(1-w+\beta)^2P_{t,j}\\
&=\zeta(1-w+\beta)[(1-\iota) w B^*_{t,j} +\iota \overline{P}^*_t]\\
&=O(t^{1-w+\beta+\epsilon}) \quad\forall \epsilon>0\,,
\end{split}
\end{equation*}
where we have used $B^*_t=o(t^\epsilon)$ that follows by Lemma~\ref{lem:alternative_to_epsilon} applied to $X_t=B^*_t$ with
$A_t=\overline{P}^*_t$ and $b=(\iota w-\beta)>0$.
(Indeed, $A_t=\overline{P}^*_t$ is a.s.\ convergent by Table~\ref{table-Z-T-P-K-medio} and, since both $B_{t,j}$ and $\overline{P}_t$ are bounded by one and $w>\beta$, we a.s. have 
$\sum_t X_t/t^2=\sum_t O(B_{t,j}/t^{1-w+\beta})<+\infty$ and
$\sum_t |\rho_{t}|=\sum_t O\left(\frac{B_{t,j}+\overline{P}_t}{t^{1+w-\beta}}\right) < +\infty$.)
Hence, choosing a suitable $\epsilon$, we get 
$$
\sum_t \frac{1}{(\theta+t+1)^2}E[(\Delta M^*_{t+1,j})^2\mid{\mathcal F}_t]<+\infty
\quad\hbox{a.s.}
$$
and so, by~Theorem~\ref{thm:app-cor}, we have 
$$ B_{t,j}^{*} = \zeta_t(1-w+\beta)B_{t,j}\stackrel{a.s.}\longrightarrow 
 \frac{\iota}{(\iota w - \beta)}\, P_{\infty}^*\,.$$
  From Table~\ref{table-Z-T-P-K-medio} (case $w>\beta>0$) we get $P_{\infty}^*=\frac{\beta}{\alpha}Z_\infty^*>0$ a.s.\ and so the first statement of case (2) of the theorem follows.  
Finally, by \eqref{eq-P-2}, we get 
\begin{equation*}
\begin{split}
t^{1-w+\beta}\,P_{t,j}\stackrel{a.s.}\longrightarrow 
 \tfrac{(1-\iota)w\,\iota}{(w\iota - \beta)}\, P_{\infty}^*+\iota P_{\infty}^*
=\tfrac{\iota(w-\beta)}{(\iota w-\beta)}\, P_{\infty}^*
= \tfrac{\iota(w-\beta)}{(\iota w-\beta)}\frac{\beta}{\alpha}Z_\infty^*\,,
\end{split}
\end{equation*}
which is the second statement of the case (2) of the theorem.
\\
\indent {\em Proof of case (3):} 
When $\iota w>\beta=0$, analogously to case (2), the term $\iota \overline{P}_t$ is not negligible in the dynamics~\eqref{eq:dynamics_B_tj_2} of $B_{t,j}$ and so we have to study the dynamics of $B_{t,j}$ paired with the one of $\overline{P}_t$. From Table~\ref{table-Z-T-P-K-medio} we know that, since $w>\beta=0$, we have $t^{1-w}\ln(t)\overline{P}_t\stackrel{a.s.}\rightarrow P_{\infty}>0$, so the dynamics we need to consider is 
$$
B_{t+1,j}=B_{t,j}-\frac{(1-(1-\iota)w)B_{t,j}-\iota\overline{P}_t}{\theta+t+1}
+\frac{\Delta M_{t+1,j}}{\theta+t+1}\,,
$$
where $\Delta M_{t+1,j}=\Delta K_{t+1,j}-P_{t,j}=\Delta K_{t+1,j}-(1-\iota)wB_{t,j}-\iota\overline{P}_t$. 
Now, let $B_{t,j}^*=\zeta_t(1-w)\ln(\theta+t)B_{t,j}$, 
$\overline{P}_t^*=\zeta_t(1-w)\ln(\theta+t)\overline{P}_t$ and
$\Delta M_{t+1,j}^*=\zeta_{t+1}(1-w)\ln(\theta+t+1)\Delta M_{t+1,j}$. 
Then, using \eqref{eq-zeta-3} from Section~\ref{sec:Technical_lemmas_deterministic} of the supplementary, we obtain 
\begin{align*}
B_{t+1,j}^*&=\Big(B_{t,j}^{*}-\tfrac{(1-(1-\iota)w)B_{t,j}^{*}-\iota\overline{P}_t^{*}}{\theta+t+1}\Big)
\tfrac{\zeta_{t+1}(1-w)}{\zeta_t(1-w)}
\tfrac{\ln(\theta+t+1)}{\ln(\theta+t)} + 
\tfrac{\Delta M_{t+1,j}^*}{\theta+t+1}\\
&=
\Big(B_{t,j}^{*}-\tfrac{(1-(1-\iota)w)B_{t,j}^{*}-\iota\overline{P}_t^{*}}{\theta+t+1}\Big)
\Big(1+\tfrac{1-w}{\theta+t+1} 
-\tfrac{1}{(\theta+t+1)\ln(\theta+t+1)}+ O\big(\tfrac{1}{t^2\ln(t)}\big)\Big)\\
&\qquad +\tfrac{\Delta M_{t+1,j}^*}{\theta+t+1}\\
&=B_{t,j}^{*}
-\iota \tfrac{wB_{t,j}^{*}-\overline{P}_t^{*}}{\theta+t+1}-
\tfrac{B_{t,j}^{*}}{(\theta+t+1)\ln(\theta+t+1)} + 
\tfrac{\Delta M_{t+1,j}^*}{\theta+t+1}
+ \rho_t\\
&=B_{t,j}^{*}
+ \frac{1}{\theta+t+1}\Big(\iota\overline{P}_t^{*}-\Big(\iota w+\tfrac{1}{\ln(\theta+t+1)}\Big)B_{t,j}^{*}\Big)+ 
\tfrac{\Delta M_{t+1,j}^*}{\theta+t+1}
+ \rho_t\,,
\end{align*}
with $\rho_t=O\big(\frac{B_{t,j}+\overline{P}_t}{t^{1+w}}\big)$ being ${\mathcal F}_t$-measurable. 
We want to apply~Theorem~\ref{thm:app-cor} with $X_t=B^{*}_t$, together with Lemma~\ref{lem:alternative_to_epsilon}.
To this end, analogously to the proof of case (2), we have that 
\begin{equation*}
\begin{split}
E[(\Delta M^*_{t+1,j})^2\mid{\mathcal F}_t]&=
\zeta(1-w)^2\ln^2(\theta+t+1)P_{t,j}(1-P_{t,j})\\
&\leq \zeta(1-w)^2\ln^2(\theta+t+1)P_{t,j}\\
&=\zeta(1-w)\ln(\theta+t+1)[(1-\iota) w B^*_{t,j} +\iota \overline{P}^*_t]\\
&=O(t^{1-w+\epsilon}\ln(t)) \quad\forall \epsilon>0\,,
\end{split}
\end{equation*}
where we have used $B^*_t=o(t^\epsilon)$ that follows by Lemma~\ref{lem:alternative_to_epsilon} applied to $X_t=B^*_t$, with
$A_t=\iota\overline{P}^*_t$ and $b_t=\iota w + 1/\ln(\theta+t+1)$ (note that we a.s. have $\sum_t X_t/t^2=\sum_t O(B_{t,j}\ln(t)/t^{1+w})<+\infty$ and 
$\sum_t|\rho_{t}|=\sum_t O\big(\frac{B_{t,j}+\overline{P}_t}{t^{1+w}}\big) < +\infty$ as $w>0$). Hence, choosing a suitable $\epsilon$, we get
$$
\sum_t \frac{1}{(\theta+t+1)^2}E[(\Delta M^*_{t+1,j})^2\mid{\mathcal F}_t]<+\infty
\quad\hbox{a.s.}\,.
$$
Then, by~Theorem~\ref{thm:app-cor}, we have 
$$ B_{t,j}^{*}= \zeta_t(1-w)\ln(\theta +t)B_{t,j}\stackrel{a.s.}\longrightarrow 
 \frac{1}{w}\, P_{\infty}^*.$$
  From Table~\ref{table-Z-T-P-K-medio} (case $w>\beta=0$) we get $P_{\infty}^*=\frac{1}{\alpha}Z_\infty^*>0$ a.s.\ and so the first statement of the case (3) of the theorem follows.  
Finally, by \eqref{eq-P-2}, we get 
\begin{equation*}
t^{1-w}\ln(t)\,P_{t,j}\stackrel{a.s.}\longrightarrow 
(1-\iota)w \frac{1}{w}\, P_{\infty}^*+\iota P_{\infty}^*
=P_{\infty}^*=\frac{1}{\alpha}Z_\infty^*\,,
\end{equation*}
which is the second statement of the case (3) of the theorem.
\end{proof}

\begin{proof}[Proof of Theorem~\ref{th-P-K-medium_interaction} (Critical regime)]
    The structure of this proof is analogous to the proof of Theorem~\ref{th-P-K-high_interaction} as, also in the case $\iota w=\beta>0$, we have to study the dynamics~\eqref{eq:dynamics_B_tj_2} of $B_{t,j}=K_{t,j}/(\theta+t)$ paired with the one of $\overline{P}_t$. From Table~\ref{table-Z-T-P-K-medio} we know that, since $w>\beta$ (because $0<\iota<1$ and $\iota w=\beta$), we have $t^{1-w+\beta}\overline{P}_t\stackrel{a.s.}\rightarrow P^*_{\infty}>0$.  
Now, let $B_{t,j}^{*}=\frac{\zeta_t(1-w+\beta)}{\ln(\theta+t)}B_{t,j}$, 
$\overline{P}_t^*=\zeta_t(1-w+\beta)\overline{P}_t$ and
$\Delta M_{t+1,j}^{*}=\zeta_{t+1}(1-w+\beta)\Delta M_{t+1,j}$. 
Then, using \eqref{eq-zeta-1} and \eqref{eq-zeta-3} from Section~\ref{sec:Technical_lemmas_deterministic} of the supplementary, we obtain
\begin{align*}
B_{t+1,j}^{*} &=\left(1-\tfrac{(1-(1-\iota)w)}{\theta+t+1}\right)B_{t,j}^{*}
\tfrac{\zeta_{t+1}(1-w+\beta)}{\ln(\theta+t+1)}
\tfrac{\ln(\theta+t)}{\zeta_t(1-w+\beta)} \\
& \quad + \tfrac{\iota\overline{P}_t^{*}}{(\theta+t+1)\ln(\theta+t+1)}
\tfrac{\zeta_{t+1}(1-w+\beta)}{\zeta_t(1-w+\beta)} + 
\tfrac{\Delta M_{t+1,j}^{*}}{(\theta+t+1)\ln(\theta+t+1)}\\
&=\left(1-\tfrac{(1-(1-\iota)w)}{\theta+t+1}\right)B_{t,j}^{*}
\Big(1+\tfrac{1-w+\beta}{\theta+t+1}-
\tfrac{1}{(\theta+t+1)\ln(\theta+t+1)} 
+ O\big(\tfrac{1}{t^2\ln(t)}\big)\Big)\\
&+
\tfrac{\iota\overline{P}_t^{*}}{(\theta+t+1)\ln(\theta+t+1)}
\left(1+\tfrac{1-w+\beta}{\theta+t+1}\right)
+ \tfrac{\Delta M_{t+1,j}^{*}}{(\theta+t+1)\ln(\theta+t+1)}\\
&=B_{t,j}^{*}
-\tfrac{B_{t,j}^{*}-\iota\overline{P}_t^{*}}{(\theta+t+1)\ln(\theta+t+1)}+
\tfrac{\Delta M_{t+1,j}^{*}}{(\theta+t+1)\ln(\theta+t+1)} + \rho_t\,,
\end{align*}
with $\rho_t=O\big(\frac{B_{t,j}+\overline{P}_t}{t^{1+w-\beta}\ln(t)}\big)=O\big(\frac{B_{t,j}+\overline{P}_t}{t^{1+w-\beta}}\big)$ being ${\mathcal F}_t$-measurable. 
We want to apply~Theorem~\ref{thm:app-cor} with $X_t=B^{*}_t$, together with Lemma~\ref{lem:alternative_to_epsilon}.
To this end, analogously to the proof of Theorem~\ref{th-P-K-high_interaction}, we have that 
$$
E[(\Delta M^*_{t+1,j})^2\mid{\mathcal F}_t] \leq
\zeta(1-w+\beta)[(1-\iota) w B^*_{t,j} +\iota \overline{P}^*_t]
=O(t^{1-w+\beta+\epsilon}),
$$
for any $\epsilon>0$, where we have used $B^*_t=o(t^\epsilon)$ that follows by Lemma~\ref{lem:alternative_to_epsilon} applied to $X_t=B^*_{t,j}$, with
$A_t=\overline{P}^*_t$ and $b=0$ (note that we a.s. have $\sum_t X_t/t^2=\sum_tO(B_{t,j}/t^{1+w-\beta})<+\infty$ and  
$\sum_t|\rho_{t}|=\sum_t O\left(\frac{B_{t,j}+\overline{P}_t}{t^{1+w-\beta}}\right) < +\infty$ as $w>\beta$). Hence, choosing a suitable $\epsilon$, we get
$$
\sum_t \frac{1}{(\theta+t+1)^2\ln^2(\theta+t+1)}E[(\Delta M^*_{t+1,j})^2\mid{\mathcal F}_t]<+\infty
\quad\hbox{a.s.}\,.
$$
Then, by~Theorem~\ref{thm:app-cor}, we have 
$$ B_{t,j}^{*} = \frac{\zeta_t(1-w+\beta)}{\ln(\theta+t)}B_{t,j}\stackrel{a.s.}\longrightarrow 
\iota P_{\infty}^*\,.$$
   From Table~\ref{table-Z-T-P-K-medio} (case $w>\beta>0$) we get $P_{\infty}^*=\frac{\beta}{\alpha}Z_\infty^*>0$ a.s.\ and so the first statement of the theorem follows.  
   Finally, by \eqref{eq-P-2}, we get 
\begin{equation*}
\frac{t^{1-w+\beta}}{\ln(\theta+t)}\,P_{t,j}\stackrel{a.s.}\longrightarrow 
(1-\iota)w \iota P_{\infty}^*
= (1-\iota)w \iota\frac{\beta}{\alpha}Z_\infty^*\,,
\end{equation*}
which is the second statement of the theorem.
\end{proof}

\subsection{Second-order asymptotics in the case of low interaction}\label{sec:second_order_asymp_dish_specific_quantities}

We focus here on the low-interaction regime, when    
each dish-specific process $K_{t,j}$ converges to a limit random variable, proportional to $K^*_{\infty,j} > 0$ (see Theorem~\ref{th-P-K-low_interaction}).  
Therefore, it could result useful to establish a CLT for these random variables $K^*_{\infty,j}$, in order to construct confidence intervals and estimate them from the observable quantities $K_{t,j}$ and $\overline{T}_t$. This result is presented in Theorem~\ref{thm:second_order-P_j}.  
Its proof relies on the auxiliary process $B_{t,j} = K_{t,j}/(\theta+t)$, introduced in Section~\ref{sec:first_order_asymp_dish_specific_quantities} and on the results proven for the averaged quantity $\overline{P}_t$.   
By~\eqref{eq-dyn-B-star} from the proof of Theorem~\ref{th-P-K-low_interaction}, we can observe that the convergence of $t^{1-(1-\iota)w} B_{t,j}$ toward $K^*_{\infty,j}$ is governed by the two competing contributions:
\begin{equation}\label{eq:heuristic_2CLT}
 t^{1-(1-\iota)w} \frac{\Delta M_{t+1,j}}{\theta+t+1}
\qquad\text{and}\qquad
t^{1-(1-\iota)w}\frac{\iota \overline{P}_t}{\theta+t+1}\,,
\end{equation}
where $\Delta M_{t+1,j}=(\Delta K_{t+1,j}-P_{t,j})$.
Heuristically, the first term suggests a martingale CLT with scaling $t^{-(1-\iota)w/2}$, leading to a zero-mean Gaussian limit kernel.  
The second term, which is absent in the case of no interaction ($\iota=0$), depends on the random averaged quantity $\overline{P}_t$ and so it exhibits a different order of decay according to the values of $\beta$ and $w$ (see Table~\ref{table-Z-T-P-K-medio} for the first-order asymptotics derived by means of Theorems~\ref{th-Z-T-medio} and Theorem~\ref{th-Z-T-medio-altro-caso}, and  see 
Corollary~\ref{cor-clt} for the second-order asymptotics obtained from Theorem~\ref{thm:second_order-Z}).
In presence of interaction (we here mean low interaction, i.e. \(0<\iota<\min\{\beta/w,1\}\)), the interplay between these two contributions determines the second-order asymptotic behavior of \( t^{1-(1-\iota)w} B_{t,j} \).  
Specifically: if $\iota=0$ OR $0<\iota < \min\{2\beta/w-1,1\}$ 
(note that the case $\beta\geq w$ and $0\leq \iota<1$ is contained in the union of these two cases), 
   the second term is negligible also at the second order and $t^{(1-\iota)w/2}(K_{t,j}/t^{(1-\iota)w}-K^*_{\infty,j})$ converges to 
    a Gaussian kernel with zero mean; if \( 0<\iota = 2\beta/w-1<1 \) (that implies $\beta<w$), we still obtain that $t^{(1-\iota)w/2}(K_{t,j}/t^{(1-\iota)w}-K^*_{\infty,j})$ converges 
    to a Gaussian kernel,  but the second term contributes with a strictly negative  random shift in the mean;
    if $\max\{0, (2\beta/w-1)\}<\iota <\beta/w<1$,  
   the second term dominates so that $(K_{t,j}/t^{(1-\iota)w}-K^*_{\infty,j})\stackrel{P}\sim -CZ_\infty^* t^{-(\beta-w \iota)}$, with a suitable constant $C>0$, 
   and, if we also have $\iota<3\beta/w-1$, then the Gaussian fluctuations arise for 
   $(K_{t,j}/t^{(1-\iota)w}-K^*_{\infty,j} + CZ_\infty^* t^{-(\beta-w \iota)})$ 
   multiplied by the scaling factor $t^{(1-\iota)w/2}$. 
 Analogously to the discussion in Section~\ref{sec:second_order_asymp_average_quantities},
also here the asymptotic contribution of the second term is strictly negative (when it is not negligible), because, in the dynamics~\eqref{eq-dyn-B-star}, that term represents a strictly positive contribution.  
This affects the second-order asymptotics by increasing the value of $K^{*}_{\infty,j}$ and, consequently, by reducing the difference $( K_{t,j}/t^{w(1-\iota)} - K^{*}_{\infty,j})$.   
Moreover, the variance of the Gaussian limit kernel is always random and equal to the dish-specific limit $K^{*}_{\infty,j}$. 
Finally, it is interesting to discuss how the convergence rate in the CLT changes according to the level $\iota$ of interaction. 
Since the second term in~\eqref{eq:heuristic_2CLT} is proportional to $\iota\overline{P}_t$, it is natural to expect that the larger the parameter $\iota$ is, the slower the convergence of the CLT will be. Indeed, we can observe from Theorem~\ref{thm:second_order-P_j} that the 
decay rate of the difference $( K_{t,j}/t^{w(1-\iota)} - K^{*}_{\infty,j})$ is of the form  $1/t^{f(\iota)}$, where $f(\iota)$ is a decreasing function in the parameter $\iota$.  
In particular, we will see that $f(\iota)$ tends to zero as $\iota$ approaches~$\min\{\beta/w,1\}$ (that is its maximum possible value in the case of low-interaction), 
while, for $0<\iota\downarrow 0$, we have  that $f(\iota)$ converges to $\min\{\beta,w/2\}>0$ and, finally, $f(0)=w/2$. This means that 
 a discontinuity in the convergence rate of the CLT arises when $\beta < w/2$.

\begin{thm}[CLT for the case of low interaction]
\label{thm:second_order-P_j}
Assume
\begin{itemize}
    \item[(1)] $\iota=0$ (no interaction at all) or
    \item[(2)] $0<\iota<\min\{\beta/w,1\}$.
\end{itemize}
Let
$K^*_{\infty, j}$ be the a.s.\ limit of $K_{t,j}/t^{w(1-\iota)}$ introduced in Theorem~\ref{th-P-K-low_interaction}.
Then:
\begin{itemize}
\item[(i)]  when  
$\iota=0$ OR 
$0<\iota<\min\{2\beta/w-1,1\}$  
(that implies 
$(1-\iota)w/2<\beta-w\iota$) OR 
$0<\iota=2\beta/w-1<1$ (that implies  
$(1-\iota)w/2=\beta-w\iota$),  for each observed dish $j$, 
we have 
\begin{equation}\label{clt-K-i}
\begin{split}
&t^{(1-\iota)w/2}\left(t^{-(1-\iota)w}K_{t,j} + I_{\{0<\iota=(2\tfrac{\beta}{w}-1)<1\}}\tfrac{\iota(\beta/\alpha)}{(\beta-w\iota)}\,Z^*_\infty \,t^{-(\beta-w\iota)}-K^*_{\infty,j}\right)\\
&\qquad \stackrel{stably}\longrightarrow 
{\mathcal N}(0,\, K^*_{\infty,j})\,,
\end{split}
\end{equation}
and the convergence also holds true in the sense of the a.s. conditional convergence with respect to $\mathcal F$, 
\item[(ii)] when  
$\max\{0,(2\beta/w-1)\}<\iota<\beta/w<1$ (that implies 
$(1-\iota)w/2>\beta-w\iota$),  for each observed dish $j$,
we have
\begin{equation*}
 {t}^{\beta-w\iota}(t^{-w(1-\iota)}K_{t,j}-K^*_{\infty,j})\stackrel{P}\longrightarrow -\tfrac{\iota\beta/\alpha}{\beta-w\iota}Z^*_\infty\,.
\end{equation*}
Moreover, when we also have $(w-\beta)-(1-\iota)w/2< \beta/2$ (i.e. $\iota<3\beta/w-1$),  for each observed dish $j$, we have 
\begin{equation*}
t^{(1-\iota)w/2}\left(t^{-(1-\iota)w}K_{t,j} +
\tfrac{\iota(\beta/\alpha)}{(\beta-w\iota)}\,Z^*_\infty \,t^{-(\beta-w\iota)}-K^*_{\infty,j}\right)\stackrel{stably}\longrightarrow 
{\mathcal N}(0,\, K^*_{\infty,j})\,.
\end{equation*}
\end{itemize}
\end{thm}
Note that, in both cases $(i)$ and $(ii)$, 
the random quantity $Z_\infty^*$ can be seen as the a.s. limit of $w t^{1-w}\overline{T}_t$ 
(for case $(ii)$, recall~\eqref{eq-clt-compact-Tmedio}).

\begin{proof}[Proof of Theorem~\ref{thm:second_order-P_j}]
We will first prove the CLT for the process $B_{t,j}=K_{t,j}/(\theta+t)$ and then 
we will derive from it the one for $K_{t,j}$. 
Indeed, since the dynamics of $B_{t,j}$ in~\eqref{eq:dynamics_B_tj_2}, i.e.
$$
B_{t+1,j}=
\left(1-\frac{(1-(1-\iota)w)}{\theta+t+1}\right)B_{t,j}
+\frac{(\Delta K_{t+1,j}-P_{t,j})}{\theta+t+1}
+\frac{\iota\overline{P}_t}{\theta+t+1}
$$
is equal to the general dynamics~\eqref{eq-general-dyn-real} in Section~\ref{subsec:CLT} of the supplementary, with 
$\delta=w(1-\iota)\in (0,1]$, $\Delta M_{t+1}=(\Delta K_{t+1,j}-P_{t,j})$, $\rho_{t+1}=\iota\overline{P}_t$, 
$$\rho=\begin{cases}
   w-\beta, & \text{if $w>\beta$ AND $0<\iota<1$}\,;\\
    w(1-\iota)\varepsilon_0 & \text{if $w\leq \beta$ OR $\iota=0$}\,;
\end{cases}
$$
for an arbitrary $\varepsilon_0\in (0,1)$, and
$$\rho_\infty=\begin{cases}
   \iota(\beta/\alpha)Z^*_\infty, & \text{if $w>\beta$ AND $0<\iota<1$}\,;\\
    0 &  \text{if $w\leq \beta$ OR $\iota=0$}\,;
\end{cases}$$
(because, by Theorem~\ref{th-Z-T-medio} and Theorem~\ref{th-Z-T-medio-altro-caso}, $t^{1-\rho}\,\iota\overline{P}_t$ converges a.s.\ to $\rho_\infty$), and 
since we are assuming $w\iota<\beta$, in all cases we have $0<\rho<\delta$ and so  
the CLT for $B_{t,j}$ can be obtained by applying Theorem~\ref{th-CLT-general} with $X_t=B_{t,j}$ and $X^*_\infty=K^*_{\infty,j}$.
Indeed, by Theorem~\ref{th-P-K-low_interaction}, we have $t^{1-(1-\iota)w}B_{t,j}\stackrel{a.s.}\sim t^{-(1-\iota)w} K_{t,j}\stackrel{a.s.}\to K^*_{\infty,j}$ 
and, in addition, 
condition~\eqref{AA:eq:momento-M1} of Theorem~\ref{th-CLT-general} follows by~\eqref{eq:bound Delta K} in Lemma~\ref{lem:Delta K}, because for any $p\geq1 $ and $\epsilon>0$, we have 
$$
E\Big[ \sum_{t=1}^{\infty} 
\frac{|\Delta {M}_{t+1}|^p}{t^{(1-\iota)w+\epsilon}}
\Big] =
\sum_{t=1}^{\infty} \frac{1}{t^{1+\epsilon}}
E[ t^{1-(1-\iota)w}|\Delta {K}_{t+1}|^p] 
<+\infty\,.
$$
Moreover,
condition~\eqref{AA:eq:momento-convergence} of Theorem~\ref{th-CLT-general} follows by~Theorem~\ref{th-P-K-low_interaction}, because 
$$
 t^{1-(1-\iota)w} E[(\Delta M_{t+1})^2|\mathcal{F}_t]=
t^{1-(1-\iota)w}P_{t,j}(1-P_{t,j})\stackrel{a.s.}\longrightarrow V_\infty^{*} = w(1-\iota)K^*_{\infty,j}\,.
$$ 
Then, from Theorem~\ref{th-CLT-general}, choosing $\varepsilon_0<1/2$ (so that we always have $\rho<\delta/2$ when $w\leq \beta$ OR $\iota=0$), we obtain that:
\begin{itemize}
    \item[(i)] For $\rho\leq \delta/2=(1-\iota)w/2$, i.e. for $w\leq \beta$ OR $\iota=0$ OR 
    $w>\beta$ and $0<\iota<2\beta/w-1$ (that implies $\iota<\beta/w$) OR $0<\iota=2\beta/w-1<1$ (that implies $\iota<\beta/w<1$),  
    we have 
\begin{equation*}
t^{(1-\iota)w/2}\left(t^{1-(1-\iota)w}B_{t,j} + 
\tfrac{\rho_\infty}{(1-\iota)w-\rho}\,t^{-[(1-\iota)w-\rho]}-K^*_{\infty,j}\right)
\stackrel{stably}\longrightarrow 
{\mathcal N}(0,\, K^*_{\infty,j})\,,
\end{equation*}
where $\rho$ and $\rho_\infty$ take different values according to the different cases and so we obtain \eqref{clt-K-i}. Moreover, the convergence also holds true in the sense of the a.s. conditional convergence with respect to $\mathcal F$. 
(Note that the above four cases can be grouped as in point $(i)$ of the statement: indeed, the case $w\leq \beta$ is covered by the union of case
$\iota=0$ and case $0<\iota<\min\{2\beta/w-1, 1\}$, that implies $\iota<\min\{\beta/w,1\}$.)
\item[(ii)] For $(1-\iota)w/2=\delta/2<\rho<\delta$, i.e. for  
$\max\{0,(2\beta/w-1)\}<\iota<\beta/w<1$  
(that obviously implies $\beta<w$ and it is possible only if $\beta>0$), we have
\begin{equation*}
 {t}^{\beta-w\iota}(t^{1-w(1-\iota)}B_{t,j}-K^*_{\infty,j})\stackrel{P}\longrightarrow -\tfrac{\iota\beta/\alpha}{\beta-w\iota}Z^*_\infty\,.
\end{equation*}
Moreover, since $\rho_{t+1}=\iota\overline{P}_t$, by the limit relation~\eqref{clt-Pmedio} in Corollary~\ref{cor-clt}, we have 
\begin{align*}
t^{\rho-\delta/2}(t^{1-\rho}\rho_{t+1}-\rho_\infty)& =
\iota \tfrac{t^{(w-\beta)-(1-\iota) w/2}}{t^{\beta/2}}
t^{\beta/2}(t^{1-(w-\beta)}\overline{P}_t+\tfrac{\beta w}{(w-\beta)}t^{-(w-\beta)}-\tfrac{\beta}{\alpha}Z^*_\infty)
\\
& \qquad -\tfrac{\beta\iota w}{(w-\beta)}t^{-(1-\iota)w/2}
\stackrel{P}\longrightarrow 0\,,
\end{align*}
provided $\rho-\delta/2=(w-\beta)-(1-\iota)w/2< \beta/2$, 
and so, by Theorem~\ref{th-CLT-general}, in this case we have 
\begin{equation*}
t^{(1-\iota)w/2}\left(t^{1-(1-\iota)w}B_{t,j} +
\tfrac{\iota(\beta/\alpha)}{(\beta-w\iota)}\,Z^*_\infty \,t^{-(\beta-w\iota)}-K^*_{\infty,j}\right)\stackrel{stably}\longrightarrow 
{\mathcal N}(0,\, K^*_{\infty,j})\,.
\end{equation*}
\end{itemize}
Finally, the thesis follows because $B_{t,j}=\frac{K_{t,j}}{(\theta+t)}=\frac{K_{t,j}}{t}(1+O(\frac{1}{t}))$.
\end{proof}

\section{Final remarks and possible extensions for future research}\label{sec:future_research}

We have proposed an Indian-buffet-type model for multi-factorial innovation, in which the inclusion probability of each feature $P_{t,j}$ is subjected to a forcing input toward zero, controlled by the parameter $w$, and a mean-field interaction among the observed features, regulated by the parameter $\iota$. 
The mean-field term represents the global level
of activity of the observed feature system: an old feature may be selected not only because it is itself popular,
but also because the overall feature environment is highly active.
Our analysis has focused on rigorously establishing the asymptotic behavior of the model, with all results depending on these two novel 
parameters, $w$ and $\iota$,  alongside the standard innovation parameters $\alpha,\,\beta$ and $\theta$. 
Given the length and the theoretical richness of the paper, we have not addressed the statistical challenges associated with estimating these parameters. 
Nevertheless, the bases for constructing effective estimators is already provided by the theoretical results presented here:   
indeed, from the observation of the number $D_t$ of the distinct appeared features, we can estimate the parameters $\alpha$ and $\beta$; while 
the power-law behavior proven for observable processes, such as the average number $\overline{T}_t$ of features per agent  and feature popularity $K_{t,j}$, suggests that  
the new parameters, $w$ and $\iota$, can be estimated as slopes or intercepts via linear regressions on a log-log scale.

The numerous phase transitions and the proliferation of different regimes (see Table~\ref{table-Z-T-P-K-medio},~\ref{table-K-j} and~\ref{table-P-j}) arising from comparisons between parameters might suggest that the model is over-parameterized. Nevertheless, we do not believe this to be the case. Indeed, the three parameters $\alpha$, $\beta$, and $\theta$ were already present in the standard IBP~\cite{TG}, while the new parameters $\iota$ and $w$ serve distinct and specific roles, as explained above. In fact, these two parameters enable the model to capture a wide range of behaviors and reproduce diverse asymptotic regimes, which would not be possible using only the standard parameters. The phenomenon under study is inherently complex, and it is therefore natural that several parameters are required to adequately characterize it. 
We also emphasize that, in order to establish all these different regimes, we had to employ a range of distinct proof techniques, as the dynamics of the stochastic processes of interest behave differently across the various phases.  In particular, some of these proofs required the development of general theorems that may be of independent interest beyond the context considered in this paper, 
and non immediate technical lemmas; 
specifically, we refer to the results provided in Section~\ref{app-tech-res} (for general results) and in Section~\ref{sec:technical_results_model} (for technical lemmas) 
of the supplementary.

Finally, we highlight that, although the interaction structure considered in this article is a simple mean-field average over all features, it is already sufficient to generate a rich variety of asymptotic behaviors. This interaction can be described using just two additional parameters, $w$ and $\iota$, which can be analytically related to the convergence rates and to the almost-sure limits of certain observable processes, such as $(\overline{T}_t)_t$ and $(K_{t,j})_t$. 
Consequently, as already observed above, the explicit dependence of the asymptotic rates on the parameters suggests possible routes toward statistical estimation. While one could consider more general interaction matrices, including sparse or heterogeneous interactions, this would substantially increase the number of parameters and make their estimation considerably more challenging. Therefore, the mean-field formulation represents a natural balance between complexity and tractability, enabling the model to capture diverse behaviors without compromising statistical feasibility.

In addition, based on our experience with general interacting reinforced models on finite networks, the asymptotic behavior in this context is typically governed by the eigenstructure of the interaction matrix. In particular, first-order convergence is determined by the leading eigenvalue, with the a.s.\ limit identified by its corresponding eigenvector, while second-order asymptotics are influenced by the eigenvalue with the second largest real part.
However, in the context of this work, the number of vertices of the network is not fixed, as in the previous literature, since the number of observed features $D_t$ diverges to infinity. As a consequence, the increasing dimension of the interaction matrix implies that its eigenstructure may evolve over time. The mean-field interaction has the desirable property that the eigenvalues remain constant over time, and are given by $w$ and $w(1-\iota)$. Analogously, the leading eigenvector is always proportional to the uniform vector with all the entries equal to one. This stability property does not hold for a generic type of interaction.

An interesting direction for future research would be to identify an interaction matrix that grows in dimension while maintaining a fixed leading eigenvalue/eigenvector and a fixed relative second eigenvalue. This is not straightforward and would require a separate dedicated work. Nevertheless, we recognize that this is an important future development that could further enhance the significance of this topic. We believe, however, that the present work, based on the mean-field interaction, provides a first mathematically tractable step in the literature on
stochastic multi-factorial innovation models with feature-interaction.

Finally, a newsworthy direction for future work is to develop inferential procedures for the proposed interacting dynamics. In particular, it would be interesting to compare or combine the present recursive approach with Bayesian nonparametric methods for unseen-feature prediction, such as scaled-process priors~\cite{CFMB2024}, which provide tractable posterior distributions for the number of unseen genetic variants. Such a comparison would clarify the respective roles of exchangeable feature-allocation priors and non-exchangeable interacting innovation dynamics.


\subsection*{Authors' contributions}
All the authors contributed equally to the present work.

\appendix

Sections, theorems and remarks reported in this supplementary document are referenced with capital letters,
while those presented in the main document are 
referenced with numbers.

\section{Technical lemmas on deterministic sequences}\label{sec:Technical_lemmas_deterministic}

For $x\geq 0$, set  
\begin{equation}\label{eq-def-zeta}
\zeta_t(x)=\tfrac{\Gamma(x+\theta+t+1)}{\Gamma(\theta+t+1)}\quad\mbox{for } t\geq 0\,,
\end{equation}
so that we have (see \cite[Lemma~4.1]{Gouet93})
\begin{equation}\label{eq-zeta-01}
\zeta_t(x)=(\theta+t+1)^x+O((\theta+t+1)^{x-1})
\sim_{t\to +\infty} t^x
\end{equation}
and
\begin{equation}\label{eq-zeta-02}
\tfrac{1}{\zeta_t(x)}=\tfrac{1}{(\theta+t+1)^x}+O\big(\tfrac{1}{(\theta+t+1)^{x+1}}\big)\,.
\end{equation}
Moreover, we have 
\begin{equation}\label{eq-zeta-1}
\tfrac{\zeta_{t+1}(x)}{\zeta_t(x)}=
\tfrac{\Gamma(x+\theta+t+2)}{\Gamma(x+\theta+t+1)}\tfrac{\Gamma(\theta+t+1)}{\Gamma(\theta+t+2)}
=1+\tfrac{x}{\theta+t+1}
\end{equation}
and since for $y\geq 0$ we have $1-y \leq (1+y)^{-1} \leq 1-y+y^2$,
\begin{equation}\label{eq-zeta-2}
\begin{split}
1-\tfrac{x}{\theta+t+1} \leq \tfrac{\zeta_{t}(x)}{\zeta_{t+1}(x)}\leq 1-\tfrac{x}{\theta+t+1}+\big(\tfrac{x}{\theta+t+1}\big)^2\,.
\end{split}
\end{equation}
Finally, by \eqref{eq-zeta-1}, we have for $t\geq 1$ 
\begin{equation}\label{eq-zeta-3}
\begin{split}
\tfrac{\tfrac{\zeta_{t+1}(x)}{\ln(\theta+t+1)}}{
\tfrac{\zeta_t(x)}{\ln(\theta+t)}} 
& =
\left(1+\tfrac{x}{\theta+t+1}\right)
\Big(1+\tfrac{\ln(1-\tfrac{1}{\theta+t+1})}{\ln(\theta+t+1)}\Big)
\\
& = \big(1+\tfrac{x}{\theta+t+1}\big)
\Big(1-\tfrac{1}{(\theta+t+1)\ln(\theta+t+1)}+ O\big(\tfrac{1}{t^2\ln(t)}\big)\Big)
\\
& = 1+\tfrac{x}{\theta+t+1}-
\tfrac{1}{(\theta+t+1)\ln(\theta+t+1)}+O\big(\tfrac{1}{t^2\ln(t)}\big)\,.
\end{split}
\end{equation}
\begin{lem}\label{lem:Un_asymp}
Let $\zeta_t(x) \sim t^x$ defined as in~\eqref{eq-def-zeta}, $0<\delta\leq 1$, $0\leq \rho\leq 1$, and set 
\begin{equation}\label{eq:delta0_b_t}
    \delta_0 = \max\{\delta,\rho\}\in (0,1],
    \quad
    b_k = \tfrac{\theta+k}{r_k} 
    \quad\mbox{and}\quad
    a_k= \tfrac{(\theta+k+1)^2}{b_k^2\, r_k}\,,
\end{equation}
where 
$r_0=\zeta_{0}(1-\delta_0)$ and for $k\geq 1$
\begin{equation}\label{eq:r_t_appB}
    r_k =
\begin{cases}
\zeta_k(1-\delta_0)
& \text{if } \delta \neq \rho\,;\\
\zeta_k(1-\delta_0)/\ln(\theta+k)
& \text{if } \delta = \rho\,.
\end{cases}
\end{equation}
Then for $n\to +\infty$
    \[ b_n \sum_{k>n} \tfrac{r_k}{(\theta+k+1)^2}=
    b_n \sum_{k>n} \tfrac{1}{a_k\,b_k^2} \to \tfrac{1}{\delta_0},
    \quad \mbox{and}\quad
    \tfrac{1}{b_n} \sum_{k\leq n} \tfrac{1}{r_k} \to \tfrac{1}{\delta_0}\,. 
    \]
\end{lem}

\begin{proof}[Proof of Lemma~\ref{lem:Un_asymp}] 
The proof is based on the continuous integral approximation of the series with proper integrals, namely
\begin{multline}\label{eq:approx_series_variance_0}
\Big| 
b_n \sum_{k>n} \tfrac{1}{a_k\,b_k^2} - \tfrac{1}{\delta_0}
\Big| 
 \leq 
\Big| 
b_n \Big( 
\sum_{k>n} \tfrac{1}{a_k\,b_k^2} - \textstyle{\int}_n^{\infty} \overline{c}(t) dt
\Big)
\Big| 
\\
 +
b_n \textstyle{\int}_n^{\infty} |\overline{c}(t) - c(t) | dt
+
\Big|
b_n  
\textstyle{\int}_n^{\infty} c(t) dt - \tfrac{1}{\delta_0}
\Big|\,,
\end{multline}
where $\overline{c}(t) = r_t/(\theta+t+1)^2$ is the convex function that interpolates $c_k = (b_k^2a_k)^{-1} = r_k/(\theta+k+1)^2$ and $c(t) = (t^{1+\delta_0})^{-1}$
if $\delta\neq\rho$; $c(t) = (t^{1+\delta_0}\ln t)^{-1}$ if $\delta=\rho=\delta_0$.

For the first term in \eqref{eq:approx_series_variance_0}, the trapezoidal rule overestimates $\int_{n}^{n+M} \overline{c}(t) dt < \tfrac{c_n}{2} + \sum_{n+1}^{N-1} c_k + \tfrac{c_M}{2}$ and hence
    \[
    0 < b_n \Big( \textstyle{\int}_n^{\infty} \overline{c}(t) dt - \sum_{k>n} \tfrac{1}{b_k^2a_k}
    \Big) \leq b_n\tfrac{c_n}{2} = \tfrac{\theta+n}{2(\theta+n+1)^2} = O(n^{-1})\,. 
    \]
For the second term in \eqref{eq:approx_series_variance_0}, by \eqref{eq-zeta-01}, $|\overline{c}(t) -c(t)|\leq C t^{-2-w_0}$, and hence
\[
b_n\textstyle{\int}_n^{\infty} |\overline{c}(t) -c(t)| dt = O(n^{-1})\,.
\]

The second term in \eqref{eq:approx_series_variance_0} is immediate when $\delta\neq \rho$, since $\int_n^{\infty} c(t) dt = n^{-\delta_0}/\delta_0$.
When $\delta=\rho=\delta_0$, 
\[
b_n \textstyle{\int}_n^{\infty} c(t)dt = n^{\delta_0}\ln(n) \textstyle{\int}_{\delta_0\ln n}^\infty \tfrac{e^{-y}}{y} dy \to \tfrac{1}{\delta_0}\,,
\]
since, for any $x>0$ \cite[eq.5.1.20, pag.229]{AS}
\[
\tfrac{1}{2} \ln\Big( 1+\tfrac{2}{x} \Big)
< e^{x}\textstyle{\int}_x^\infty \tfrac{e^{-y}}{y}dy < \ln\Big( 1+\tfrac{1}{x} \Big)\,.
\]

\bigskip

The proof of $\tfrac{1}{b_n} \sum_{k\leq n} \tfrac{1}{r_k} \to \tfrac{1}{\delta_0}$ is trivial for $\delta_0=1$, since $r_k \equiv 1$ when $\delta\neq\rho$ and $\sum_{k\leq n}r_k^{-1}\sim n\ln(n)$ for $\delta=\rho=1$.
For $\delta_0<1$, the proof is again made on continuous integral approximations, here
\begin{multline}\label{eq:approx_series_variance_1}
\Big| 
\tfrac{1}{b_n} \sum_{k\leq n} \tfrac{1}{r_k} - \tfrac{1}{\delta_0}
\Big| 
 \leq 
\Big| 
\tfrac{1}{b_n} \Big( 
\sum_{k\leq n} \tfrac{1}{r_k} - \textstyle{\int}_0^n \tfrac{1}{r_t}  dt
\Big)
\Big| 
\\
 +
\tfrac{1}{b_n} \textstyle{\int}_0^n |\tfrac{1}{r_t}  - c(t) | dt
+
\Big|
\tfrac{1}{b_n}
\textstyle{\int}_0^n c(t) dt - \tfrac{1}{\delta_0}
\Big|\,,
\end{multline}
where $c(t) = (t^{-1+w_0})$ if $w\neq\beta$; $c(t) = (t^{-1+\delta_0}\ln t)$ if $\delta=\rho=\delta_0$.
Since $\tfrac{1}{r_t}$ is eventually decreasing by \eqref{eq-zeta-1} and \eqref{eq-zeta-3}, the first term in \eqref{eq:approx_series_variance_1} is vanishing, since
    \[
    0 < \tfrac{1}{b_n} \Big(\sum_{k\leq n} \tfrac{1}{r_k} -\textstyle{ \int}_{0}^{n} \tfrac{1}{r_t} dt
    \Big) \leq \tfrac{1}{b_n}\tfrac{2}{r_{\min}} \to 0\,. 
    \]
Moreover, by \eqref{eq-zeta-02}, $|\tfrac{1}{r_t} -c(t)|\leq C t^{-2+\delta_0}$, and hence
also the second term in \eqref{eq:approx_series_variance_1} is vanishing:
\[
\tfrac{1}{b_n} \textstyle{\int}_{0}^{n} |\tfrac{1}{r_t} -c(t)| dt \to 0\,.
\]
When $\delta\neq \rho$, $\int_{0}^{n} c(t) dt = n^{\delta_0}/\delta_0$ and the proof is done.
When $\delta=\rho=\delta_0$, with $e^y = t^{\delta_0}$, we obtain
\[
\textstyle{\int}_{0}^{n} c(t) dt = 
\textstyle{\int}_{0}^{n} \tfrac{t^{\delta_0}\ln t}{t} dt = \textstyle{\int}_{-\infty}^{\delta_0\ln n} \tfrac{e^yy}{\delta_0^2} dy 
= \tfrac{n^{\delta_0}(\delta_0\ln(n)-1)}{\delta_0^2}\,. \qedhere
\]
\end{proof}

\begin{lem}[{\cite[Supplementary material]{ale-cri-ghi-complete}}]
\label{lem-to-zero}  
If $a_t \geq 0$, $a_t \leq 1$ for $t$ large enough, $\sum_t a_t = +\infty$, $\delta_t \geq 0$, $\sum_t \delta_t < +\infty$, $b > 0$,
$y_t \geq 0$ and $y_{t+1} \leq (1 - a_t )^b y_t + \delta_t$, then
$\lim_{t\to+\infty} y_t = 0$.
\end{lem}

\section{General results for some recursive dynamics}\label{app-tech-res}

We start by providing in Subsection~\ref{subsec:generalized-SA} two results, probably known, but for which we were not able to find a reference in the literature and so, for the reader's convenience,  we state and prove them here. In Subsection~\ref{subsec:strict-pos} and Subsection~\ref{subsec:CLT} we provide two new general results.

\subsection{Results of the first-order for a generalized stochastic approximation dynamics}
\label{subsec:generalized-SA}

\begin{thm}\label{thm:app-cor} Let $(X_t)_t$ and $(A_t)_t$ be two $\mathcal{F}$-adapted real stochastic processes, where $(X_t)_t$ follows the dynamics 
$$
X_{t+1} = X_t + \eta_t (A_t - b_t X_t) + \eta_t \Delta M_{t+1} + \rho_{t+1}\,,
$$
where $(b_t)_t$ and $(\eta_t)_t$ are two sequences of real numbers with   $b_t\rightarrow b>0$, $\eta_t \geq 0$ for $t$ large enough, $\sum_t \eta_t = +\infty$, $\sum_t \eta_t^2 < +\infty$, $(\Delta M_{t+1})_t$ is an $\mathcal{F}$-martingale difference such that 
\begin{equation}\label{app-eq-delta-M-quad-cond}
\sum_t \eta_t^2 E[(\Delta M_{t+1})^2\mid \mathcal{F}_t] < + \infty \quad \hbox{a.s.}
\end{equation}
and $\rho_{t+1}$ is  a ${\mathcal F}_{t+1}$-measurable remainder term such that $\sum_t |\rho_{t+1}|<+\infty$ a.s. 
Then, we have
$$\tfrac{1}{b}\liminf_{t\rightarrow\infty} A_t\leq \liminf_{t\rightarrow\infty} X_t \leq \limsup_{t\rightarrow\infty} X_t \leq \tfrac{1}{b}\limsup_{t\rightarrow\infty} A_t \quad \hbox{a.s.}\,.$$
As a consequence, if $(A_t)_t$ converges a.s. to a random variable $A_\infty$, then we have 
$X_t \xrightarrow{a.s.} A_\infty/b$.
\end{thm}

\begin{proof}[Proof of Theorem~\ref{thm:app-cor}]
Without loss of generality, since $b_t\to b>0$, we can assume $b_t>0$ for all $t$. Moreover, we can focus on proving the thesis with $b_t$ replaced by its limit $b>0$. Indeed, since $(b_t/b)\to 1$, the thesis with a time-dependent sequence $(b_t)_t$
follows directly with the one with constant $b$ by setting 
    $\widetilde{\eta}_t=(b_t/b)\eta_t\geq 0$ for $t$ large enough, $\widetilde{A}_t= (b/b_t)A_t$ and $\Delta \widetilde{M}_{t+1}=(b/b_t)\Delta M_{t+1}$.
    \\
\indent Assume $b_t=b>0$ for all $t$. Without loss of generality, since 
$\eta_t\geq 0$ for $t$ large enough and $\eta_t \to 0$, we can assume $\eta_t$ such that $\eta_t\geq 0$ and $(1-b\eta_t)>0$ for all $t$. The rest of the proof is divided in two steps.\\
\indent { \em Step (1):}
For any $a\in\mathbb{R}$ and $t_0\in\mathbb{N}$ we can define the process 
$(X^{a,t_0}_{t})_{t\geq 0}$ as follows: $X^{a,t_0}_{0}=X_{0}$ and 
$X^{a,t_0}_{t}=\sum_{n=0}^{t-1} \Delta X^{a,t_0}_{n+1}$ for $t\geq1$, with
    \[
    \Delta X^{a,t_0}_{n+1} = 
    \begin{cases}
        \eta_n (A_n - b X_{n}^{a,t_0}) + \eta_n \Delta M_{n+1} + \rho_{n+1} & \text{if $n\leq t_0$}\,;
        \\
        \eta_n (a - b X_{n}^{a,t_0}) + \eta_n \Delta M_{n+1} + \rho_{n+1}
        & \text{if $n> t_0$}\,;
    \end{cases}
    \]
so that $X^{a,t_0}_{t}\equiv X_{t}$ for any $t\leq t_0$. Moreover notice that, for any $n> t_0$, 
\[
\Delta X^{a,t_0}_{n+1}-\Delta X_{n+1} = \eta_n ((a-A_n) - b (X^{a,t_0}_{n}-X_n)) 
\]
and then
\[
(X^{a,n_0}_{n+1} - X_{n+1}) = \eta_n (a-A_n) + ( 1- b\eta_n ) (X^{a,t_0}_{n}-X_n)\,.
\]
This means that
\begin{itemize}
    \item[(i)] \(\cap_{t\geq t_0} \{X^{a,t_0}_{t} \leq X_{t}\}\) with probability one on $\cap_{t\geq t_0} \{a < A_t\}$,
    \item[(ii)] \(\cap_{t\geq t_0} \{X^{a,t_0}_{t} \geq X_{t}\}\) with probability one on $\cap_{t\geq t_0} \{a > A_t\}$.
\end{itemize}
Therefore, on the set $\{a > \limsup_{t\rightarrow\infty} A_t\}$, with probability one we have that
\[
\limsup_{t\rightarrow\infty} X_t = \inf_{t_0} \sup_{t\geq t_0} X_t \leq \inf_{t_0} \sup_{t\geq n_0} X^{a,t_0}_{t} = \limsup_{t\rightarrow\infty} X^{a,t_0}_{t}\,.
\]
Analogously, on the set $\{a < \liminf_{t\rightarrow\infty} A_t\}$, with probability one we have that
\[
\liminf_{t\rightarrow\infty} X_t = \sup_{t_0} \inf_{t\geq t_0} X_t \geq \sup_{t_0} \inf_{t\geq t_0} X^{a,t_0}_{t} = \liminf_{t\rightarrow\infty} X^{a,t_0}_{t}\,.
\]
\noindent { \em Step (2):} We will now show that, for any fixed $a\in\mathbb{R}$ and $t_0\in\mathbb{N}$,
$$\limsup_{t\rightarrow\infty} X^{a,t_0}_{t}=\liminf_{t\rightarrow\infty} X^{a,t_0}_{t}=\tfrac{a}{b}  \quad \hbox{a.s.}\,.$$
We 
define $\widetilde{\rho}_{t_0}=0$ and $\widetilde{\rho}_{t+1}=(1-b\eta_t)\widetilde{\rho}_t+\rho_{t+1}$ for each $t\geq t_0$, so that we have 
$|\widetilde{\rho}_{t+1}|\leq (1-b\eta_t)|\widetilde{\rho}_t|+|\rho_{t+1}|$ and hence, by Lemma~\ref{lem-to-zero}, we get $\widetilde{\rho}_t\stackrel{a.s.}\to  0$.   
Set $Y_t=X^{a,t_0}_t-a/b -\widetilde{\rho}_t$ for any $t\geq t_0$. After some computations, we arrive to 
$$
Y_{t+1}=(1-b\eta_t)Y_t+\eta_t\Delta M_{t+1}
$$
and so
\begin{equation*}
\begin{split}
E[Y_{t+1}^2\mid{\mathcal F}_t]&= (1 +b^2\,\eta_t^2) Y_t^2 
+ \eta_t^2 E[(\Delta M_{t+1})^2\mid{\mathcal F}_t]
-2b\, \eta_t\,Y_t^2 \,.
\end{split}
\end{equation*}
Hence, $(Y_t^2)_t$ is a non-negative almost super-martingale, which converges a.s. since 
$\sum_t \eta^2_t<+\infty$ and \eqref{app-eq-delta-M-quad-cond} (see Theorem~\ref{th-almost-supermart}). We also get $\sum_t \eta_t Y_t^2<+\infty$ a.s.\ and so, since $\sum_t\eta_t=+\infty$, we necessarily have $Y_t^2\stackrel{a.s.}\to 0$, that is $X^{a,t_0}_t\stackrel{a.s.}\to a/b$. \\
\indent{ \em Conclusion:} We have shown that, for any pair $a_1<a_2$, on the set 
$$\{a_1 < \liminf_{t\rightarrow\infty} A_t<\limsup_{t\rightarrow\infty} A_t < a_2\}\,,$$ 
we have that 
$$\tfrac{a_1}{b}\leq \liminf_{t\rightarrow\infty} X_t < \limsup_{n\rightarrow\infty} X_t \leq \tfrac{a_2}{b}  \quad \hbox{a.s.}\,.$$
Finally, since $a_1<a_2$ are arbitrary, we get the thesis.
\end{proof}

\begin{lem}\label{lem:alternative_to_epsilon}
Let $(X_t)$ be a non-negative stochastic process such that $\sum_t X_t/t^2<+\infty$ and with dynamics
$$
X_{t+1} = X_t + \tfrac{1}{\theta+t+1} (A_t - b_t X_t) + \tfrac{ \Delta M_{t+1} }{\theta+t+1}+ \rho_{t}\,,
$$
where $b_t\geq 0$, $(A_t)_t$ is a.s. convergent, $(\Delta M_{t+1})_t$ is an $\mathcal{F}$-martingale difference and $(\rho_t)$ is an $\mathcal F$-adapted process with 
$\sum_t |\rho_t|<+\infty$ a.s. 
Then we have $X_t=o(t^\epsilon)$ for any $\epsilon>0$.
\end{lem}

\begin{proof}
Set $X_t^*=X_t/\zeta_t(\epsilon)$. 
\begin{equation*}
  X^*_{t+1}
  =\big(1-\tfrac{b_t}{\theta+t+1}\big)\tfrac{\zeta_{t}(\epsilon)}{\zeta_{t+1}(\epsilon)}X^*_t
  + \tfrac{\Delta M_{t+1}}{(\theta+t+1)\zeta_{t+1}(\epsilon)}
  +\tfrac{A_t/(\theta+t+1)+\rho_t}{\zeta_{t+1}(\epsilon)}\,.
  \end{equation*}
  Since \eqref{eq-zeta-2},  we get 
  \begin{equation*}
  \begin{split}
  X^*_{t+1}
  &=\big(1-\tfrac{b_t}{\theta+t+1}\big)
  \big(1-\tfrac{\epsilon}{\theta+t+1}+O(\tfrac{1}{t^2})\big)
  X^*_t + \tfrac{\Delta M_{t+1}}{(\theta+t+1)\zeta_{t+1}(\epsilon)}  
  +\tfrac{\tfrac{A_t}{(\theta+t+1)}+\rho_t}{\zeta_{t+1}(\epsilon)}\\
&=X^*_t - \tfrac{b_t+\epsilon}{\theta+t+1}X_t^* + O(1/t^2)X^*_t
 +\tfrac{\Delta M_{t+1}}{(\theta+t+1)\zeta_{t+1}(\epsilon)}
   +\tfrac{\tfrac{A_t}{(\theta+t+1)}+\rho_t}{\zeta_{t+1}(\epsilon)}\\
&\leq X^*_t +\tfrac{\Delta M_{t+1}}{(\theta+t+1)\zeta_{t+1}(\epsilon)}
    +\tfrac{\tfrac{A_t}{(\theta+t+1)}+\rho_t}{\zeta_{t+1}(\epsilon)} + O(1/t^2)X^*_t\,.
  \end{split}
  \end{equation*}
Hence, we have 
$$
E[X^*_{t+1}\mid {\mathcal F}_t]\leq X^*_t + 
\tfrac{|A_t|/(\theta+t+1)+|\rho_t|}{\zeta_{t+1}(\epsilon)}+ O(1/t^2)X^*_t
$$
and so $(X^*_t)$ is a non-negative almost super-martingale (see Theorem~\ref{th-almost-supermart}) and, since we have 
$$
\sum_t \Big[ \tfrac{|A_t|/(\theta+t+1)+|\rho_t|}{\zeta_{t+1}(\epsilon)}+ O(1/t^2)X^*_t \Big]<+\infty
\quad\mbox{a.s.}\,,$$ 
we can conclude that $X^*_t$, and so $X_t/t^\epsilon$, converges a.s. to a finite 
random variable. Since $\epsilon>0$ is arbitrary, the limit random variable has to be equal to zero, and so we can conclude that $X_t=o(t^{\epsilon})$.
\end{proof}

\subsection{Strict positiveness of the a.s. limit for a particular class of non-negative almost super-martingales}\label{subsec:strict-pos}

Let $X=(X_t)_t$ be a {\em non-negative} real stochastic process with discrete time, adapted to a filtration ${\mathcal F}=({\mathcal F}_t)_t$ and such that 
\begin{equation}\label{eq-general-dyn}
X_{t+1}=\Big(1-\frac{1}{\theta+t+1}\Big) X_t+ \frac{1}{(\theta+t+1)} Y_{t+1}\,,
\end{equation}
where $\theta>0$ and $(Y_{t})_t$ is a {\em non-negative} $\mathcal F$-adapted real stochastic process.  
Now, let us set $H_t=(\theta + t) X_t$. From \eqref{eq-general-dyn}, we get 
$$
H_{t+1}=H_t+Y_{t+1}
$$
and so $H_{t}=H_0+\sum_{n=1}^t (H_n-H_{n-1})=H_0+\sum_{n=1}^t Y_{n}$. Hence, $(H_t)_t$ is a non-decreasing sequence of random variables with $\sup_t H_t=H_0+\sum_t Y_t$. Therefore, if the series $\sum_t Y_t$ is a.s.\ convergent, then $(H_t)$, and so $(t\, X_t)$, is a.s. convergent. This means that $X_t=O(1/t)$ and so it a.s. converges to zero. The following result deals with the case when $\sum_t Y_t=+\infty$.\\
\indent (Note we use the notation $\rho_t^+=\max\{\rho_t,0\}$ and $\rho_t^-=\max\{-\rho,0\}$ so that $\rho_t=\rho_t^+-\rho_t^-$ and $|\rho_t|=\rho_t^++\rho_t^-$.)

\begin{thm}[Strict positiveness of the a.s.\ limit]\label{th-general}
In the framework described above, if $\sum_t Y_t=+\infty$ a.s., then  $(t\, X_t)_t$ converges a.s.\ to $+\infty$. \\ Moreover,  if we also assume 
\begin{equation} \label{eq-cond-sp}
E[Y_{t+1} \mid {\mathcal F}_t]= \delta X_t + \rho_t\,,
\end{equation}
where $0<\delta\leq 1$ and $(\rho_t)_t$ is an $\mathcal F$-adapted stochastic process such that $\rho_t\stackrel{a.s.}\longrightarrow 0$ with  
$(*)\, \sum_t \rho_t^+/(\theta+t)^{\delta}<+\infty$ a.s., then 
$\sum_t|\rho_t|/(\theta+t)^\delta<+\infty$ a.s.\ and 
$(t^{(1-\delta)}\, X_t)_t$ converges a.s. to a finite  random variable $X^*_\infty$, that is $X_t=O(1/t^{1-\delta})$.  
\\
Finally, $P(X^*_\infty > 0)=1$, that is $t^{-(1-\delta)}$ is the exact convergence rate of $(X_t)_t$ toward zero, if \\
a) we replace condition~$(*)$ by the more restrictive one: 
\begin{equation}\label{eq-cond-rho-2}
\sum_t\tfrac{|\rho_t|}{(\theta+t)^{\delta-\epsilon}}<+\infty\;\mbox{a.s.}\quad\mbox{for some } 0<\epsilon<\delta\,,
\end{equation}
b) we also add the assumption 
\begin{equation}\label{eq-cond-second}
E[Y_{t+1}^2 \mid {\mathcal F}_t ]\leq O(E[Y_{t+1}\mid {\mathcal F}_t])\,,
\end{equation}
and c) we require that $\rho_t$ is a.s. in one of the following conditions:
\begin{itemize}
\item[i)] $\rho_t\geq 0$ eventually;
\item[ii)] $t |\rho_t|=O(1)$;
\item[iii)] $\rho_t$ is a function of $X_t$ such that $|\rho_t|=o(X_t)$.
\end{itemize}
\end{thm}

\begin{rem}\label{rem-general}
When the random variables $Y_t$ are uniformly bounded and $\rho_t\geq 0$, 
in order to have $\sum_t Y_t=+\infty$ a.s., it is enough that exists a finite time $\tau$, possibly random, such that $X_{\tau}>0$ a.s.. Indeed we a.s. have 
\begin{equation*}
\begin{split}
\sum_t E[Y_{t+1}\mid {\mathcal F}_t]&=\delta \sum_t X_t +\sum_t \rho_t 
\geq \delta \sum_t \tfrac{H_t}{(\theta+t)}
\geq \delta H_{\tau}\sum_{t\geq \tau} \tfrac{1}{\theta+t} 
=+\infty\,,
\end{split}
\end{equation*}
where we used that 
$H_t=(\theta+t)X_t$ is non-decreasing (see above) and 
$H_{\tau}=(\theta+\tau)X_{\tau}>0$ a.s.. This is sufficient in order to conclude because of 
Lemma~\ref{del-mey-result-app}. 
\end{rem}

\begin{proof}[Proof of Theorem~\ref{th-general}]
Let us set $H_t=(\theta + t) X_t$ as we have done above so that, from \eqref{eq-general-dyn}, we have 
$$
H_{t+1}=H_t+Y_{t+1}
$$
and so $H_{t}=H_0+\sum_{n=1}^t (H_n-H_{n-1})=H_0+\sum_{n=1}^t Y_{n}$. Hence, $(H_t)_t$ is a non-decreasing sequence of random variables with $\sup_t H_t=H_0+\sum_t Y_t$. Therefore, if  $\sum_t Y_t=+\infty$ a.s., then $(H_t)$, and so $(t\, X_t)$, diverges a.s.\ to $+\infty$ (and so $X_t>0$ for $t$ large enough). \\
Now, we set 
  $X^*_t=\zeta_t(1-\delta)X_t$ and $\Delta M_{t+1}=Y_{t+1}-E[Y_{t+1}\mid {\mathcal F}_t]$ so that  we have 
  \begin{equation*}
  X^*_{t+1}
  =\big(1-\tfrac{(1-\delta)}{\theta+t+1}\big)\tfrac{\zeta_{t+1}(1-\delta)}{\zeta_t(1-\delta)}X^*_t+ \tfrac{\zeta_{t+1}(1-\delta)\Delta M_{t+1}}{\theta+t+1}
+\tfrac{\zeta_{t+1}(1-\delta)}{\theta+t+1}\rho_t\,.
  \end{equation*}
  Since \eqref{eq-zeta-1}, 
  we get
  \begin{multline}\label{eq:X_star}
  X^*_{t+1}
  =\big(1-\tfrac{(1-\delta)}{\theta+t+1}\big)\big(1+\tfrac{(1-\delta)}{\theta+t+1}\big)
  X^*_t+\tfrac{\zeta_{t+1}(1-\delta)\Delta M_{t+1}}{\theta+t+1}
+\tfrac{\zeta_{t+1}(1-\delta)}{\theta+t+1}\rho_t\\
  \begin{aligned}
&=X^*_t - \tfrac{(1-\delta)^2}{(\theta+t+1)^2}X^*_t
+\tfrac{\zeta_{t+1}(1-\delta)\Delta M_{t+1}}{\theta+t+1} +\tfrac{\zeta_{t+1}(1-\delta)}{\theta+t+1}\rho_t\\
&\leq X^*_t + \tfrac{\zeta_{t+1}(1-\delta)\Delta M_{t+1}}{\theta+t+1}+
\tfrac{\zeta_{t+1}(1-\delta)}{\theta+t+1}\rho_t\,.
  \end{aligned}
  \end{multline}
Hence, we have 
$$
E[X^*_{t+1}\mid {\mathcal F}_t]\leq X^*_t + 
\tfrac{\zeta_{t+1}(1-\delta)}{\theta+t+1}\rho^+_t
-\tfrac{\zeta_{t+1}(1-\delta)}{\theta+t+1}\rho^-_t
$$
and so $(X^*_t)$ is a non-negative almost super-martingale (see~Theorem~\ref{th-almost-supermart}) and, since we have 
$$
\sum_t \tfrac{\zeta_{t+1}(1-\delta)}{(\theta+t+1)}\rho^+_t
\leq 
\sum_t \tfrac{\rho^+_t}{(\theta+t)^\delta}+ \sum_t \tfrac{O((\theta+t)^{-\delta})}{(\theta+t)}\rho_t^+<+\infty
\quad\mbox{a.s.}\,,$$ 
we can conclude that $X^*_t$, and so $t^{1-\delta}X_t$, converges a.s. to a finite 
random variable $X^*_\infty$ (and so $X_t\stackrel{a.s.}\longrightarrow 0$) and also 
$\sum_t \tfrac{\zeta_{t+1}(1-\delta)}{(\theta+t)}\rho^-_t<+\infty$ a.s.\ 
so that we get $\sum_t |\rho_t|/(\theta+t)^\delta<+\infty$ a.s..\\
\indent Finally, we assume conditions \eqref{eq-cond-second}, \eqref{eq-cond-rho-2} and one of the conditions i)-iii) and we are going to prove that $X_\infty^*>0$ a.s. To this purpose we proceed in some steps.
\\
\indent { \em Step 1 (Proof of the fact that, for each $0<\eta< \delta\leq 1$, we have  
$1/X_t=o(t^{1-\eta})$ or, equivalently, $1/H_t=o(t^{-\eta})$):}\\
\indent We set $Q_t=(\theta+t)/H_t^\gamma$ with $\gamma>0$ and we observe that we have 
\begin{multline*}
\tfrac{Q_{t+1}}{Q_t}- 1=\tfrac{(\theta+t+1)}{(\theta+t)}\tfrac{H_t^\gamma}{H_{t+1}^\gamma}-1
=\left(1+\tfrac{1}{\theta+t}\right)\tfrac{H_t^\gamma}{H_{t+1}^\gamma}-1\\
\begin{aligned}
&=\big(\tfrac{H_t^\gamma}{H_{t+1}^\gamma}\big)^\gamma-1+
\tfrac{1}{(\theta+t)}\big(\tfrac{H_t}{H_{t+1}}\big)^\gamma
=\big(1-\tfrac{Y_{t+1}}{H_{t+1}}\big)^\gamma-1+
\tfrac{1}{(\theta+t)}\big(\tfrac{H_t}{H_{t+1}}\big)^\gamma\\
&\leq -\gamma \tfrac{Y_{t+1}}{H_{t+1}}+\tfrac{1}{(\theta+t)}
= -\gamma \tfrac{Y_{t+1}}{H_t}+\gamma Y_{t+1}\big(\tfrac{1}{H_t}-\tfrac{1}{H_{t+1}}\big)
+\tfrac{1}{(\theta+t)}\\
&=-\gamma \tfrac{Y_{t+1}}{H_t}+\gamma \tfrac{Y_{t+1}^2}{H_tH_{t+1}}
+\tfrac{1}{(\theta+t)}
\leq -\gamma \tfrac{Y_{t+1}}{H_t}+\gamma \tfrac{Y_{t+1}^2}{H_t^2}
+\tfrac{1}{(\theta+t)}\,,
\end{aligned}
\end{multline*}
where the first inequality follows since $(1-x)^\gamma\leq 1-\gamma x$ for $0\leq x\leq 1$ and $H_t\leq H_{t+1}$ and this last relation implies also the second inequality. 
If we take the conditional expectation given the past ${\mathcal F}_t$ and use the equality 
$H_t=(\theta+t)X_t$ and $E[Y_{t+1}\mid {\mathcal F}_t]=\delta X_t+\rho_t$, we obtain 
\begin{multline}\label{eq-key}
E\bigl[\tfrac{Q_{t+1}}{Q_t}- 1\mid {\mathcal F}_t\bigr]\leq
-\gamma \tfrac{(\delta X_t +\rho_t)}{H_t}+\gamma \tfrac{E[Y_{t+1}^2\mid{\mathcal F}_t]}{H_t^2}
+\tfrac{1}{(\theta+t)}\\
\begin{aligned}
&\leq -\gamma \tfrac{(\delta X_t +\rho_t)}{H_t}+\gamma \tfrac{O\left(E[Y_{t+1}\mid{\mathcal F}_t]\right)}{H_t^2}+\tfrac{1}{(\theta+t)}\\
&=\tfrac{(1-\gamma \delta)}{(\theta+t)}
-\tfrac{\gamma}{(\theta+t)} \tfrac{(\theta+t)\rho_t}{H_t}+
\tfrac{1}{(\theta+t)}O\big(\tfrac{1}{H_t}\big)
+ 
\tfrac{1}{(\theta+t)}O\big(\tfrac{(\theta+t)|\rho_t|}{H_t^2}\big)\\
&=\tfrac{1}{(\theta+t)}\Big[(1-\gamma \delta)
+O\big(\tfrac{1}{H_t}\big)
+\tfrac{(\theta+t)|\rho_t|}{H_t}\big( -\gamma sign(\rho_t)
+ O\big(\tfrac{1}{H_t}\big)\big) \Big]\,.
\end{aligned}
\end{multline}
If we prove that the above last quantity into the square brackets is a.s.\ eventually non-positive for $\gamma \delta>1$, then we get that for $\gamma\delta>1$ we a.s. have 
$$
E \big[Q_{t+1}- Q_t\mid {\mathcal F}_t \big]=
Q_t E\big[\tfrac{Q_{t+1}}{Q_t}- 1\mid {\mathcal F}_t\big]\leq 0
\quad\mbox{eventually}.
$$
This proves that, for each $\gamma>1/\delta$, the stochastic process $(Q_t)_t$, i.e. 
 $((\theta+t)/H_t^{\gamma})_t$ is eventually a (non-negative) super-martingales and
 so, for each $\gamma>1/\delta$, it converges a.s.\ to a real random variable. 
 Since $\gamma>1/\delta$ is arbitrary,
 we necessarily have that $(\theta+t)/H_t^{\gamma}$ converges a.s. to zero. 
 This fact concludes the proof of Step 1 because it implies that $(\theta+t)^{1/\gamma}/H_t$ converges a.s. to zero and we can set $\eta=1/\gamma<\delta\leq 1$. Therefore the point is to have additional assumptions that entail the last term in the square brackets of \eqref{eq-key} to be non-positive for $\gamma\delta>1$. When 
 $\gamma\delta>1$, we have $(1-\gamma\delta)<0$ and so, since $H_t\to +\infty$, the first  quantity $O(1/H_t)$ can be neglected. Hence the crucial term is 
 $$
 \tfrac{(\theta+t)|\rho_t|}{H_t}\Big( -\gamma sign(\rho_t)
+ O\big(\tfrac{1}{H_t}\big)\Big)\,.
 $$
In the case i), that is when $\rho_t\geq 0$ eventually with probability one, we are done because the above quantity is equal to zero when $\rho_t=0$ and, when $\rho_t>0$, the quantity $-\gamma sign(\rho_t)+O(1/H_t)=-\gamma+O(1/H_t)$ is eventually non-positive with probability one, because $-\gamma<0$ and $O(1/H_t)\to 0$ a.s.. 
Another simple cases are case ii), that is when $\limsup_t (t |\rho_t|)<+\infty$ a.s., and case iii), that is when $|\rho_t|=o(X_t)$. Indeed, in these cases we a.s. have
$$
\lim_t \tfrac{(\theta+t) |\rho_t|}{H_t}=\lim_t \tfrac{|\rho_t|}{X_t}=0\,.
$$
 \indent {\em Step 2 (Convergence of  fundamental series):}\\
From Step 1 and assumption \eqref{eq-cond-rho-2}, we a.s. get 
$$
\sum_t \tfrac{|\rho_t|}{H_t}=
\sum_t 
\tfrac{(\theta+t)^{\eta}}{H_t}\,\tfrac{|\rho_t|}{(\theta+t)^{\eta}}\\
=\sum_t o(1)\tfrac{|\rho_t|}{(\theta+t)^{\eta}}<+\infty\,,
$$
where $0<\eta=\delta-\epsilon<\delta$. Moreover, we a.s. have 
$$
\sum_t \tfrac{1}{t H_t}=\sum_t\tfrac{t^\eta}{H_t} \tfrac{1}{t^{1+\eta}}=
\sum_t o(1) \tfrac{1}{t^{1+\eta}}<+\infty\,,
$$
where $0<\eta<\delta$.
\\
\indent {\em Step 3 (Conclusion):} 
 Set $L_t=\ln(H_t/(\theta+t)^{\delta})$ and $U_t = E[L_{t+1} - L_t \mid \mathcal{F}_t ] $ 
and $V_t = E[(L_{t+1} - L_t )^2 \mid {\mathcal F}_t]$. If we prove that $\sum_t U_t$ and $\sum_t V_t$ are a.s.
convergent, then $L_t$ converges a.s.\ to a real random variable (see Lemma~\ref{lemma-pemantle}). This fact implies that $(\theta+t)^{1-\delta}X_t=(\theta+t)X_t/(\theta+t)^{\delta}=H_t/(\theta+t)^{\delta}$
converges to a random variable with values in $(0, +\infty)$ and so $P(X^*_\infty>0)=1$. The rest of the proof is devoted to verify that $\sum_t |U_t | < +\infty$ and
$\sum_t V_t < +\infty$ a.s..
To this regard, we note that
\begin{equation*}
  \begin{split}
    U_t
    &=E[\ln(H_{t+1})-\ln(H_t)\mid \mathcal{F}_t]-\delta\left(\ln(\theta+t+1)-\ln(\theta+t)\right)
    \\
    &=E[\ln(H_{t}+Y_{t+1})-\ln(H_t)\mid \mathcal{F}_t]-\delta\ln(1+1/(\theta+t))
    \\
   &=E\Big[\textstyle{\int}_0^{Y_{t+1}}\tfrac{1}{H_t+x}\,dx\mid {\mathcal F}_t\Big]-\delta\ln(1+1/(\theta+t))\,.
  \end{split}
  \end{equation*}
Since $1/(H_t+x)\leq 1/H_t$ and $\ln(1+1/(\theta+t))\geq 1/(\theta+t)-1/(2(\theta+t)^2)$
  for each $x\geq 0$ and each $t$,    the last term of the above equalities
  is smaller than or equal to 
\begin{multline*}
\tfrac{E[Y_{t+1}|\mathcal{F}_t]}{H_t}-\tfrac{\delta}{\theta+t}+\tfrac{\delta}{2 (\theta+t)^2}
=\tfrac{\delta X_t+\rho_t}{H_t}-\tfrac{\delta}{\theta+t}+\tfrac{\delta}{2(\theta+t)^2}\\
=
\tfrac{\delta}{\theta+t}+\tfrac{\rho_t}{H_t}-\tfrac{\delta}{\theta+t}+\tfrac{\delta}{2 (\theta+t)^2}
=\tfrac{\rho_t}{H_t}+O(1/t^2)\,.
\end{multline*}
Now, we
note that $-U_t=\delta\ln(1+1/(\theta+t))-\ln(H_{t+1})+\ln(H_t)$. 
Using $\ln(1+1/(\theta+t))\leq 1/(\theta+t)$  
 and $1/(H_t + x) \geq 1/H_t - x/H_t^2$ for each $x\geq 0$ and each $t$, 
 we find that
$-U_t$  is smaller than or equal to 
 \begin{multline*}
 \tfrac{\delta}{(\theta+t)}-\tfrac{1}{H_t}E[Y_{t+1}\mid \mathcal{F}_t]+
 \tfrac{1}{2H_t^2}E[Y_{t+1}^2\mid \mathcal{F}_t]\\
 \begin{aligned}
 &\leq \tfrac{\delta}{(\theta+t)}-\tfrac{\delta X_t+\rho_t}{H_t}+
 \tfrac{O(\delta X_t +\rho_t)}{2H_t^2}
 =\tfrac{\delta}{\theta+t}-\tfrac{\delta}{\theta+t}-\tfrac{\rho_t}{H_t}+
 \tfrac{O(\delta X_t +\rho_t)}{2H_t^2}\\
 &=-\tfrac{\rho_t}{H_t}+O(1/(tH_t))+o(\rho_t/H_t)\,.
 \end{aligned}
 \end{multline*}
 Thus, by Step 2, we can derive that $\sum_t|U_t|<+\infty$ a.s.. 
Finally, we observe we have
 \begin{multline*}
E[( \ln(H_{t+1})-\ln(H_t)-\delta\ln(\theta+t+1)+\delta\ln(\theta+t) )^2
\mid \mathcal{F}_t]\\  
\begin{aligned}
&\leq
2\left\{
E[(\ln(H_{t+1})-\ln(H_t))^2\mid \mathcal{F}_t]+
\delta^2(\ln(1+1/(\theta+t))^2
\right\}
\\
&\leq
2 E\bigg[\Big(\textstyle{\int}_0^{Y_{t+1}}\tfrac{1}{H_t+x}\,dx\Big)^2\,\Big\vert \mathcal{F}_t\bigg]
+2\delta^2/(\theta+t)^2
\\
& \leq 2
E[(Y_{t+1}/H_t)^2\mid \mathcal{F}_t]+
O(1/t^2)
\leq \tfrac{2}{H_t^2} E[Y_{t+1}^2\mid\mathcal{F}_t]+O(1/t^2)\\ 
&= \tfrac{1}{H_t^2} O(\delta X_t+\rho_t) +O(1/t^2)\\
&= O(1/(tH_t)) + o(\rho_t/H_t) + O(1/t^2)\,.
  \end{aligned}
 \end{multline*}
Therefore we have also 
$\sum_t V_t<+\infty$ a.s.\ and we can conclude.
\end{proof}

\subsection{Central limit theorem for a recursive dynamics}
\label{subsec:CLT}

Let $X=(X_t)_t$ be a real stochastic process with discrete time, 
adapted to a complete filtration ${\mathcal F}=({\mathcal F}_t)_t$ with dynamics 
\begin{equation}\label{eq-general-dyn-real}
X_{t+1}=\Big(1-\frac{(1-\delta)}{\theta+t+1}\Big) X_t+ \frac{1}{(\theta+t+1)} \Delta M_{t+1}
+\frac{1}{\theta+t+1}\rho_{t+1}\,,
\end{equation}
where $\theta>0$, $0<\delta\leq 1$, $\Delta M_{t+1}$ is a martingale difference and $(\rho_t)_t$ is an $\mathcal F$-adapted stochastic process such that $t^{1-\rho}\rho_{t+1}\stackrel{a.s.}\longrightarrow \rho_\infty$ with $0\leq \rho\leq 1$ and $\rho_\infty$ being a real (${\mathcal F}_\infty$-measurable) random variable.
\\

\indent Note that a non-negative stochastic process $X=(X_t)_t$ with dynamics~\eqref{eq-general-dyn}, where $Y=(Y_t)_t$ is a sequence of integrable random variable that verify condition~\eqref{eq-cond-sp} with $t^{1-\rho}\rho_t\stackrel{a.s.}\longrightarrow \rho_\infty$ for a suitable $\rho\in [0,1)$, 
satisfy also dynamics~\eqref{eq-general-dyn-real} with $\Delta M_{t+1}=Y_{t+1}-E[Y_{t+1}|\mathcal{F}_t]$ (and in this case the reminder term $\rho_{t+1}$ in~\eqref{eq-general-dyn-real} coincides with $\rho_t$ of condition~\eqref{eq-cond-sp} and so it is even ${\mathcal F}_t$-measurable).
\\

\indent When $t^{1-\delta} X_t$ a.s. converges toward a real random variable $X^*_\infty$ (to this regards, see Remark~\ref{rem-conv} after the following proof), 
the result below provides a central limit theorem for the difference 
$(t^{1-\delta} X_t-X^*_\infty)$.

\begin{thm}[CLT]\label{th-CLT-general}
Given the dynamics~\eqref{eq-general-dyn-real} and the existence of $X^*_\infty=a.s.-\lim_t t^{1-\delta} X_t$, assume $0\leq \rho< \delta$,
\begin{equation} \label{AA:eq:momento-M1}
E\Big[ \sum_{t=1}^{\infty} 
\frac{|\Delta {M}_{t+1}|^p}{t^{\delta+\epsilon}}
\Big] 
<+\infty
\end{equation}
for $p=2$ and $p=4$ and some $0<\epsilon<\delta$,
and
\begin{equation} \label{AA:eq:momento-convergence}
t^{1-\delta}  E\Big[ (\Delta {M}_{t+1})^2|\mathcal{F}_{t}\Big] 
\stackrel{a.s.}\longrightarrow V^*_\infty
\end{equation}
for some real (${\mathcal F}_\infty$-measurable) random variable $V^*_\infty$.
Then:
\begin{itemize}
\item[(i)] For $0\leq \rho\leq \delta/2$, we have 
\begin{equation}\label{eq-clt-1}
t^{\delta/2}(t^{1-\delta} X_t-X^*_\infty)\stackrel{stably}\rightarrow 
\begin{cases}
    {\mathcal N}(0,V^*_\infty/\delta) & \text{if $0\leq \rho <\delta/2$}\,;\\
    {\mathcal N}(-2\rho_\infty/\delta , V^*_\infty/\delta) & \text{if $\rho =\delta/2$}\,;\\
\end{cases}
\end{equation}
or, equivalently, in a compact form 
\begin{equation}\label{eq-clt-compact-general}
t^{\delta/2}(t^{1-\delta} X_t+\tfrac{ \rho_\infty}{(\delta-\rho)}t^{-(\delta-\rho)} -X^*_\infty)\stackrel{stably}\rightarrow {\mathcal N}(0,V^*_\infty/\delta)\,,
\end{equation}
and the all the limit relations are also in the sense of the a.s. conditional convergence with respect to the 
filtration~$\mathcal F$.\\
\item[(ii)] For $\delta/2<\rho <\delta$, we have 
$$
t^{\delta-\rho}(t^{1-\delta} X_t-X^*_\infty)\stackrel{P}\longrightarrow -\frac{ \rho_\infty}{(\delta-\rho)}\,.
$$
Moreover, adding the assumption $(*)\; t^{\rho-\delta/2}(t^{1-\rho}\rho_{t+1}-\rho_\infty)\stackrel{P}
\rightarrow 0$, the stable convergence~\eqref{eq-clt-compact-general} holds true. 
Finally, if the additional assumption $(*)$ holds true in the sense of the a.s.\ convergence, then 
the limit relation~\eqref{eq-clt-compact-general} also holds true in the sense of the a.s.\ conditional convergence. 
\end{itemize}
\end{thm}

\begin{rem}[Comparison with the CLTs by L.-X. Zhang~\cite{Zhang2016}]
\rm 
The dynamics of the process $X^*_t=\zeta_t(1-\delta)X_t$, where $\zeta_t(x) \sim t^x$ is defined in~\eqref{eq-def-zeta}, 
is 
$$
X^*_{t+1}=X^*_t - \tfrac{(1-\delta)^2}{(\theta+t+1)^2}X^*_t
+\tfrac{\zeta_{t+1}(1-\delta)\Delta M_{t+1}}{\theta+t+1} +\tfrac{\zeta_{t+1}(1-\delta)}{\theta+t+1}\rho_t
$$
and, if we compare it with the dynamics (1.1) in~\cite{Zhang2016}, then, using the notation in that paper, we get $h(\cdot)=0$, 
a case which is excluded in~\cite{Zhang2016}. 
In addition, there is also another important difference. For $\delta<1$, we have by~\eqref{AA:eq:momento-convergence}
that $E[(\zeta_{t+1}(1-\delta)\Delta M_{t+1})^2|\mathcal{F}_t]$ tends a.s.\ to infinity and so assumption (2.4) in~\cite{Zhang2016} does not hold. 
Therefore Theorem~\ref{th-CLT-general} is not a consequence of the results in \cite{Zhang2016}.
\end{rem}

Before presenting the proof of Theorem~\ref{th-CLT-general}, we will report a  preliminary  result used in that proof. Indeed, a crucial point will be played by the following martingale: $\widehat{M}_0 = 0$ and, for any $t\geq0$,
\begin{equation}\label{eq:wideHat_M}
\widehat{M}_{t+1} = 
\sum_{k=0}^{t} \Delta \widehat{M}_{k+1},
\quad\hbox{with}\quad
\Delta \widehat{M}_{k+1}=\tfrac{r_k}{\theta + k + 1} \Delta M_{k+1}\,,
\end{equation}
where the sequence $(r_k)$ is defined as in \eqref{eq:r_t_appB}. The proof of the following  auxiliary result will be presented after the proof of the CLT (Theorem~\ref{th-CLT-general}).

\begin{thm}\label{AA:lem:Delta_nToGauss}
Given the dynamics~\eqref{eq-general-dyn-real}, let 
$\delta_0$ and $(b_t)_t$ defined as in~\eqref{eq:delta0_b_t} and 
$(\widehat{M}_t)_t$ as in \eqref{eq:wideHat_M}. Assume that,  
 for $p=2$ and $p=4$ and  some $0<\epsilon<\delta_0$, we have 
\begin{equation}\label{AA:eq:momento-p}
E\Big[ \sum_t t^{p\delta_0} 
\frac{|\Delta \widehat{M}_{t+1}|^p}{t^{\delta_0+\epsilon}}
\Big] 
<+\infty
\end{equation}
and 
\begin{equation} \label{AA:eq:momento-convergence-r}
\frac{(\theta+t+1)^2}{r_t}E\Big[ (\Delta \widehat{M}_{t+1})^2|\mathcal{F}_{t}\Big]= 
r_t E\Big[ (\Delta {M}_{t+1})^2|\mathcal{F}_{t}\Big] 
\stackrel{a.s.}\longrightarrow V^*_\infty
\end{equation}
for some real (${\mathcal F}_\infty$-measurable) random variable $V^*_\infty$. 
Then, $(\widehat{M}_t)_t$ is a martingale bounded in $L^4$ and so it converges a.s.\ and in 
$L^4$ to a real random variable $\widehat{M}_\infty$ with $E[\widehat{M}_\infty^4]<+\infty$.  Moreover, we have 
$$\sqrt{b_t}(\widehat{M}_t-\widehat{M}_\infty)\rightarrow {\mathcal N}(0,V^*_\infty/\delta_0)\,,
$$  where 
the convergence is in the sense of the a.s. conditional convergence with respect to the filtration~$\mathcal F$ (and so also in the sense of the stable convergence). 
\end{thm}

We can now present the proof of the CLT.

\begin{proof}[Proof of Theorem~\ref{th-CLT-general}]
Let  us consider the process $X^*=(X^*_t)_t$ defined as 
$X^*_t = \zeta_t(1-\delta)X_t$, where $\zeta_t(x) \sim t^x$ is defined in~\eqref{eq-def-zeta}, 
so that $X^*_t\stackrel{a.s.}\sim t^{1-\delta}X_t\stackrel{a.s.}\to X^*_\infty$ by assumption. 
Starting from dynamics~\eqref{eq-general-dyn-real} and 
recalling that $r_k=\zeta_k(1-\delta)$ (by assumption $\delta>\rho$), we can follow the analogous calculations as in~\eqref{eq:X_star} to get 
\begin{equation}\label{eq-sum}
X^*_{k+1}-X^*_k = A_k+B_{k+1}+C_{k+1}+D_{k+1}\,,
\end{equation}
where
$$\left\{\begin{aligned}
A_k&=-\Big(\tfrac{1-\delta}{\theta+k+1}\Big)^2 X^*_k\quad\mbox{(which is $=0$ if $\delta=1$)}\,;\\
B_{k+1}&=\tfrac{\zeta_{k+1}(1-\delta)\rho_{k+1}}{\theta+k+1}\,;\\
C_{k+1}&=\Big(\tfrac{\zeta_{k+1}(1-\delta)}{\zeta_k(1-\delta)}-1\Big)\Delta \widehat{M}_{k+1}\quad\hbox{(which is $=0$ if $\delta=1$)}\,;\\
D_{k+1}&=\Delta \widehat{M}_{k+1}\,.
\end{aligned}\right.$$
Note that $\sum_{k\geq t}(X_{k+1}^*-X^*_k)=(X^*_\infty-X^*_t)$. Therefore, it is enough to study the asymptotic behavior of the four terms $\sum_{k\geq t}A_k,\,\sum_{k\geq t}B_{k+1}$, $\sum_{k\geq t}C_{k+1}$ and $\sum_{k\geq t} D_{k+1}$.

Let us start with the last term $\sum_{k\geq t}D_{k+1}$.
First of all we observe that, since $\delta>\rho$ by assumption, we have $r_t\sim t^{1-\delta}$ and so  
$$
E[\,|\Delta \widehat{M}_{t+1}|^p\,]  = 
E[\,|\Delta M_{t+1}|^p\,] \tfrac{r_t^p}{(\theta+t+1)^p} = 
E[\,|\Delta M_{t+1}|^p\,]O(t^{-p\delta})\,.
$$
Hence, relation \eqref{AA:eq:momento-p} with $\delta_0=\delta$ 
follows from~\eqref{AA:eq:momento-M1}. Moreover, condition~\eqref{AA:eq:momento-convergence} equals condition~\eqref{AA:eq:momento-convergence-r}.  
It follows by Theorem~\ref{AA:lem:Delta_nToGauss} and the fact that $b_t\sim t^\delta$, that we have $\widehat{M}_t\stackrel{a.s.}\longrightarrow \widehat{M}_\infty$ for a suitable real random variable $\widehat{M}_\infty$ and 
\[
t^{\delta/2} \sum_{k\geq t} D_{k+1}=t^{\delta/2}\sum_{k\geq t} \Delta \widehat{M}_{k+1} =
t^{\delta/2}(\widehat{M}_\infty-\widehat{M}_t)
\stackrel{stably}\longrightarrow {\mathcal N}\Big(0,V^*_\infty/\delta\Big)\,,
\]
where the convergence is also in the sense of the a.s. conditional convergence with respect to the filtration~$\mathcal F$.\\
\indent 
Let us now focus on the term $\sum_{k\geq t}C_{k+1}$ when $\delta<1$.
Let us define $Y_{k}=\zeta_{k}(1-\delta)\Delta M_{k+1}$ and $\mathcal{G}_k=\mathcal{F}_{k+1}$, then $E[Y_{k}\mid\mathcal{G}_{k-1}] = 0$.
Now, let $\tilde{a}_k=k^{-\delta}(\theta+k+1)^2$ and $\tilde{b}_k=k^{\delta/2}$. Then 
$\tilde{b}_k\uparrow +\infty$ and
\begin{equation*}
\begin{split}
\sum_{k=1}^{\infty}\tfrac{E[Y_{k}^2]}{\tilde{a}_k^2\tilde{b}_k^2}
&=
\sum_{k=1}^{\infty}\tfrac{\zeta_k(1-\delta)^2}{\tilde{a}_k^2\,\tilde{b}_k^2} E[(\Delta M_{k+1})^2]
=
\sum_{k=1}^{\infty}\tfrac{\zeta_k(1-\delta)^2(\theta+k+1)^2}{\tilde{a}_k^2\,\tilde{b}_k^2\,r_k^2}E[(\Delta \widehat{M}_{k+1})^2] \\
&=\sum_{k=1}^{\infty}\tfrac{\zeta_k(1-\delta)^2(\theta+k+1)^2}{\tilde{a}_k^2\,\tilde{b}_k^2\,\zeta_k(1-\delta)^2}E[(\Delta \widehat{M}_{k+1})^2] 
=\sum_{k=1}^{\infty}
\tfrac{k^{\delta}}{(\theta+k+1)^2}E[(\Delta \widehat{M}_{k+1})^2]\\
&\leq 
\sum_{k=1}^{\infty} k^{2\delta}\tfrac{E[(\Delta \widehat{M}_{k+1})^2]}{k^{\delta+\epsilon}}
\,,
\end{split}
\end{equation*}
which is finite, by \eqref{AA:eq:momento-p} with $p=2$, $\delta_0=\delta$  and some $0<\epsilon< \delta$.
Since $\tilde{b}_t\sum_{k\geq t} \tfrac{1}{\tilde{a}_k\tilde{b}_k^2} \to 0$, then by Lemma~\ref{lemma-kro}, we obtain 
\begin{multline*}
t^{\delta/2}\sum_{k\geq t} C_{k+1}=
t^{\delta/2} \sum_{k\geq t} \Big(\tfrac{\zeta_{k+1}(1-\delta)}{\zeta_{k}(1-\delta)}-1\Big) \Delta \widehat{M}_{k+1} =
t^{\delta/2} \sum_{k\geq t} \tfrac{1-\delta}{\theta+k+1} \Delta \widehat{M}_{k+1} \\
= (1-\delta) t^{\delta/2} \sum_{k\geq t} \tfrac{\zeta_{k}(1-\delta)\Delta M_{k+1}}{(\theta+k+1)^2}  
= (1-\delta) \tilde{b}_t\sum_{k\geq t} \tfrac{Y_{k}}{\tilde{a}_k\,\tilde{b}_k^2}\stackrel{a.s.}
\longrightarrow 0\,.
\end{multline*}
Let us now focus on the term $\sum_{k\geq t} A_k$ when $\delta<1$. 
Since $X^*_t$ converges a.s. to a real random variable $X^*_\infty$, we have 
\[
t^{\delta/2} \sum_{k\geq t} A_k=
t^{\delta/2} \sum_{k\geq t} \Big(\tfrac{1-\delta}{\theta+k+1}\Big)^2 X^*_k = 
 t^{\delta/2} \sum_{k\geq t} \tfrac{O(1)}{k^{1+\delta}}\stackrel{a.s.}\longrightarrow 0\,.
\]
Finally, regarding the term $\sum_{k\geq t}B_{k+1}$, since~\eqref{eq-zeta-01}, we have
\begin{equation}\label{eq:rho_t}
    {t}^{\delta-\rho}\zeta_{k+1}(1-\delta)\rho_{k+1} = \rho_\infty + o(1)\,,
\end{equation}
and hence
\begin{equation*}
\sum_{k\geq t} B_{k+1}=
\sum_{k\geq t} \tfrac{\zeta_{k+1}(1-\delta)\rho_{k+1}}{\theta+k+1} = 
\sum_{k\geq t} \tfrac{\rho_\infty}{k^{1+\delta-\rho}} + o\Big(\tfrac{1}{t^{\delta-\rho}}\Big) \sim \tfrac{\rho_{\infty}}{(\delta-\rho)} \tfrac{1}{t^{\delta-\rho}}\,,
\end{equation*}
so that
\[
\begin{aligned}
t^{\delta/2} \sum_{k\geq t}B_{k+1}
=
\tfrac{t^{\delta/2}}{t^{\delta-\rho}}\, t^{\delta-\rho} \sum_{k\geq t} B_{k+1}
& \stackrel{a.s.}\longrightarrow
\begin{cases}
    0 & \text{if $\rho<\delta/2$}\,;
    \\
    \rho_\infty/\rho=2\rho_\infty/\delta  & \text{if $\rho = \delta/2$}\,;
\end{cases}
\\
 t^{\delta-\rho} \sum_{k\geq t} B_{k+1} 
& \stackrel{a.s.}\longrightarrow
\tfrac{\rho_\infty}{\delta-\rho}\quad \text{if $\rho > \delta/2$}\,.
\end{aligned}
\]
Recalling that the stable convergence combines well with convergence in probability (e.g. Theorem 2.4.2 in \cite{crimaldi-libro}), and so with the a.s.\ convergence,  
and the a.s. conditional convergence combines well with the a.s. convergence (see Lemma~\ref{lem-new-as-cond-conv}), 
we are now ready to conclude:
\begin{itemize}
\item[(i)] if $\rho\leq \delta/2$, then, we have to combine the stable convergence  and the a.s.\ conditional convergence of the suitable rescaled sequence 
$\sum_{k\geq t} D_{k+1}=(\widehat{M}_\infty-\widehat{M}_t)$ with the a.s.\ convergences of the suitably rescaled sequences 
$\sum_{k\geq t} A_k$ (only when $\delta<1$), $\sum_{k\geq t} B_{k+1}$  and $\sum_{k\geq t} C_{k+1}$ (only when $\delta<1$) so that we obtain~\eqref{eq-clt-1} 
(with $X^*_t$ instead of $t^{1-\delta}X_t$), in the sense of the stable convergence and in the sense of  the a.s. conditional convergence with respect to $\mathcal F$. 
\item[(ii)] if $\delta/2 < \rho < \delta$, then the leading term is $\sum_{k\geq t} B_{k+1}$ and the a.s.\ convergence of $t^{\delta-\rho} \sum_{k\geq t} B_{k+1}$ 
has to be combined with the a.s.\ limit results for the other terms involving $A_k$ and $C_{k+1}$ (only when $\delta<1$) and, above all, with 
the convergence in probability to zero of 
$\tfrac{t^{\delta-\rho}}{t^{\delta/2}}t^{\delta/2}(\widehat{M}_\infty-\widehat{M}_t)$, where $\tfrac{t^{\delta-\rho}}{t^{\delta/2}}=t^{-(\rho-\delta/2)}\to 0$.
\end{itemize}
Furthermore, the \eqref{eq-clt-compact-general} (with $X^*_t$ instead of $t^{1-\delta}X_t$) is obvious when $\rho\leq \delta/2$, 
because the stable convergence and the a.s.\ conditional convergence combine well with the a.s.\  convergence of 
$t^{-(\delta/2-\rho)}\rho_\infty/(\delta-\rho)$ to zero when $\rho<\delta/2$ and to $2\rho_\infty/\delta$ when $\rho=\delta/2$. 
In the case $\delta/2<\rho<\delta$, if we assume $(*)\; t^{\rho-\delta/2}(t^{1-\rho}\rho_{t+1}-\rho_\infty)\stackrel{P}\longrightarrow 0$, then,
equation~\eqref{eq:rho_t} becomes
$${k}^{\delta-\rho}\zeta_{k+1}(1-\delta)\rho_{k+1} = \rho_\infty + o_P(k^{-(\rho-\delta/2)})\,,$$
so that 
$$
t^{\delta/2}
\Big(
 \sum_{k\geq t} B_{k+1}
-  \tfrac{\rho_\infty}{(\delta-\rho)}\tfrac{1}{t^{\delta-\rho}}
\Big)=o_P(1) \stackrel{P}\longrightarrow 0
$$
and this convergence in probability combines well with the stable convergence of $t^{\delta/2}\sum_{k\geq t} D_{k+1}=t^{\delta/2}(\widehat{M}_\infty-\widehat{M}_t)$. 
If the additional assumed convergence $(*)$ is a.s., then we can replace $o_P()$ with $o()$ and the above limit relation is in the sense of a.s. convergence and it combines well with the a.s. conditional converge of $t^{\delta/2}\sum_{k\geq t} D_{k+1}$.\\

Finally, we can obtain the results for $t^{1-\delta} X_t$ from those proven for $X^*_t$, because of~\eqref{eq-zeta-01} with $x=1-\delta$.
\end{proof}

\begin{rem}[Sufficient conditions for the a.s.\ convergence of $t^{1-\delta} X_t$] \label{rem-conv}
\rm Consider $X^*_t = \zeta_t(1-\delta)X_t$, where $\zeta_t(x) \sim t^x$ is defined in~\eqref{eq-def-zeta}. 
From \eqref{eq-sum}, if we can take the conditional expectation, we get
$$
E[X_{t+1}^*\mid{\mathcal F}_t] -X_t^*=A_t+\widetilde{B}_t\,,
$$
with $\widetilde{B}_t=E[B_{t+1}|{\mathcal F}_t]$. 
Then, if the process $X$ (and so $X^*$) is non-negative, we get that $A_t\leq 0$ for each $t$ and  $X^*$ results to be a non-negative almost martingale which converges a.s.\ when $\sum_t |\widetilde{B}_t|<+\infty$ a.s.\ (as it is, for instance, when $\rho<\delta$ and $\rho_{t+1}$ is ${\mathcal F}_t$-measurable so that $\widetilde{B}_t=B_{t+1}=O(1/t^{1+\delta-\rho})$). \\ 
\indent In general, if  
$\sum_t E[\, |A_t|\, ]+\sum_t E[\,|\widetilde{B}_t|\,]<+\infty$, the process $X^*$ is a quasi-martingale and a sufficient condition for its a.s. convergence is $\sup_t E[\,|X^*_t|\,]<+\infty$ (which implies $\sum_t E[\,|A_t|\,]<+\infty$). Summing up, $\sup_t E[\,|X^*_t|\,]<+\infty$ and $\sum_t E[\,|B_{t+1}|\,]<+\infty$ (which implies $\sum_t E[\,|\widetilde{B}_t|\,]<+\infty$) are sufficient conditions for the a.s.\ convergence of $X^*_t$. 
\end{rem}

\begin{proof}[Proof of Theorem~\ref{AA:lem:Delta_nToGauss}] 
First of all we observe that relation \eqref{AA:eq:momento-p}  with $p=4$ implies that 
$( \widehat{M}_t)_t$ is a martingale bounded in $L^4$ and so it converges
a.s.\ and in $L^4$ (in particular, in mean) to a real (${\mathcal F}_\infty$-measurable) random variable $\widehat{M}_\infty$ 
with $E[\widehat{M}_\infty^4]<+\infty$. 
The sequel of the proof is based on Theorem~\ref{fam_tri_vet_as_inf}. 
To this aim, 
for any $n\geq 1$, define 
$M_{n,0}=0$ and 
$M_{n,t} = \sqrt{b_n} (\widehat{M}_{n}-\widehat{M}_{n+t-1})$ for $t\geq 1$ 
(which is $\mathcal{F}_{n+t-1}$-measurable). We have $M_{n,t}\stackrel{L^1}\longrightarrow \sqrt{b_n}(\widehat{M}_n-\widehat{M}_\infty)$ as $t\to +\infty$.   
Moreover, we have  $X_{n,1}=M_{n,1}-M_{n,0}=0$ and, for each $t\geq 2$ 
\[
X_{n,t} = M_{n,t}-M_{n,t-1}=
\sqrt{b_n} (\widehat{M}_{n+t-2}-\widehat{M}_{n+t-1})\,, 
\]
so that 
\[
U_{n} = b_n \sum_{t\geq n} (\widehat{M}_{t+1}-\widehat{M}_t)^2 , \quad 
X_{n}^* = \sqrt{b_n} \sup_{t\geq n} |\widehat{M}_{t+1}-\widehat{M}_t|\,.
\]
\emph{First step: $ X_n^* \stackrel{a.s.}\longrightarrow 0$: } From \eqref{AA:eq:momento-p}, with 
$p=4$ and some $0<\epsilon < \delta_0$, we have 
\begin{equation}\label{AA:eq:inProof_mbound}
\begin{split}
E\Big[\sum_{t} (\sqrt{b_t})^4 |\Delta \widehat{M}_{t+1}|^{4} \Big]
&\leq E\Big[\sum_{t} t^{2\delta_0+(\delta_0-\epsilon)} |\Delta \widehat{M}_{t+1}|^{4} \Big]\\
&=
E\Big[\sum_{t} t^{p\delta_0} \tfrac{|\Delta \widehat{M}_{t+1}|^{p}}{t^{\delta_0+\epsilon}} \Big]
<+\infty\,,
\end{split}
\end{equation}
which implies $\sum_{t} (\sqrt{b_t})^4 (\Delta \widehat{M}_{t+1})^{4}<+\infty$ a.s.. Then we have 
$$
(X_n^*)^{4}\leq \sup_{t\geq n} (\sqrt{b_t})^4 |\Delta \widehat{M}_{t+1}|^{4}\leq 
\sum_{t\geq n} (\sqrt{b_t})^4 |\Delta \widehat{M}_{t+1}|^{4}\stackrel{a.s.}\to 0\quad\mbox{as } n\to +\infty\,.
$$
\emph{Second step: $(X_n^*)_n$ dominated in $L^1$:} 
We are going to show that $(X_n^*)_n$ is dominated in $L^4$. In fact, with $p=4$ and some $0<\epsilon<\delta_0$, we have 
\begin{multline*}
E\Big[(\sup_n X_n^*)^{4}\Big]\leq E\Big[\sup_n \sup_{t\geq n} (\sqrt{b_t})^4 |\Delta \widehat{M}_{t+1}|^{4}\Big]
\\
\leq E\Big[ \sum_t (\sqrt{b_t})^4 |\Delta \widehat{M}_{t+1}|^{4}\Big] < +\infty
\end{multline*}
as shown in \eqref{AA:eq:inProof_mbound} in the first step.
\\
\indent \emph{Third step: $U_n \stackrel{a.s.}\longrightarrow V^*_\infty/\delta_0$:} Let $a_k$ as defined in~\eqref{eq:delta0_b_t} (so that $a_kb_k^2=(\theta+k+1)^2/r_k$) and set 
$Y_{k} = a_kb_k^2 (\Delta \widehat{M}_{k+1})^2=r_k(\Delta M_{k+1})^2$. Then,  we have  $E[Y_{k}\mid\mathcal{F}_{k-1}]\to V^*_\infty$ by assumption~\eqref{AA:eq:momento-convergence-r}.
Moreover, we have $b_k\uparrow +\infty$ and, by \eqref{AA:eq:momento-p} with $p=4$ and some $0<\epsilon<\delta_0$, we get 
$$ 
\sum_{k=1}^{\infty}\tfrac{E[Y_k^2]}{a_k^2b_k^2}=
\sum_{k=1}^{\infty}b_k^2E[(\Delta\widehat{M}_{k+1})^4]
\leq C \sum_{k=1}^{\infty} k^{4\delta_0} \tfrac{E[(\Delta \widehat{M}_{k+1})^4]}{k^{\delta_0+\epsilon}}<+\infty\,.
$$
Since $b_t\sum_{k\geq t} \tfrac{1}{a_kb_k^2}\to \tfrac{1}{\delta_0}$ by the  technical Lemma~\ref{lem:Un_asymp}, then 
$
U_t = b_t\sum_{k\geq t} \tfrac{Y_k}{a_kb_k^2}\stackrel{a.s.}
\longrightarrow V^*_\infty/\delta_0
$
by Lemma~\ref{lemma-kro}. 
\end{proof}

\section{Proof of the results in Section~\ref{results-mean}}

We here collect the proofs for the results presented in Section~\ref{results-mean} regarding the predictive mean $Z_t$ and the averaged quantities $\overline{T}_t$, 
$\overline{P}_t$ and $\overline{K}_t$.

\subsection{First-order asymptotic results}\label{sec:proofs_first_order_average}

\begin{proof}[Proof of Theorem~\ref{th-Z-T-medio}] In the case $w > \beta$ the term involving $\lambda_t$ in the dynamics~\eqref{eq-diff-Z} of $(Z_t)_t$ behaves essentially as a remainder and hence we can focus on the process $S_t=\sum_{j\in {\mathcal O}_t} P_{t,j}=Z_t-\lambda_t$, so that
from \eqref{eq-diff-Z} we get  
  \begin{equation*}
  \begin{split}
      S_{t+1}-S_t&=Z_{t+1}-Z_t-(\lambda_{t+1}-\lambda_t)\\
  &=w\tfrac{T_{t+1}-Z_t}{\theta+t+1}-(1-w)\tfrac{Z_t}{\theta+t+1}
+  \tfrac{\lambda_t}{\theta+t+1}\\
  &=-\tfrac{Z_t}{\theta+t+1}+w\tfrac{T_{t+1}}{\theta+t+1}
  +\tfrac{\lambda_t}{\theta+t+1}
  =-\tfrac{S_t}{\theta+t+1}+\tfrac{wT_{t+1}}{\theta+t+1}\,.
  \end{split}
  \end{equation*}
  The above dynamics for $S_t$ can be written as
\begin{equation}\label{eq-dyn-S}
  S_{t+1}
  =\big(1-\tfrac{1}{\theta+t+1}\big)S_t+w\tfrac{T_{t+1}}{\theta+t+1}
  =\big(1-\tfrac{1-w}{\theta+t+1}\big)S_t+\tfrac{\Delta M_{t+1}}{\theta+t+1}+
  \tfrac{w\lambda_t}{\theta+t+1}\,,
\end{equation}
  with $\Delta M_{t+1}=w(T_{t+1}-Z_t)$.
  Denoting $S^{*}_{t}=\zeta_{t}(1-w)S_t$, where $\zeta_t(x)\sim t^x$ is defined in~\eqref{eq-def-zeta}, we can obtain
 \begin{equation*}
  S^{*}_{t+1}=\big(1-\tfrac{(1-w)^2}{(\theta+t+1)^2}\big)S^{*}_t
  +\tfrac{\zeta_{t+1}(1-w)\Delta M_{t+1}}{\theta+t+1} +\tfrac{w\zeta_{t+1}(1-w)\lambda_t}{\theta+t+1}\,,
  \end{equation*}
  and so, by the classical Theorem~\ref{th-almost-supermart}, the process $(S^{*}_t)_t$ is a non-negative almost super-martingale that converges a.s.\ to a finite r.v. $S^{*}_\infty$. However, 
  this is not enough for our scopes. Indeed, in that way, we do not know that  such random limit $S^{*}_\infty$ is a.s.\ strictly positive, and this is crucial for proving the right rate of convergence. 
  For this reason we are going to apply the more sophisticated Theorem~\ref{th-general} with $X_t=S_t$, $Y_{t+1}=wT_{t+1}$ and $\delta=w\in (0,1]$. 
  To this end, we observe that we have 
  $$
  \sum_t Y_{t}=w\sum_t T_t\geq w \sum_t N_t=w\sup_t \sum_{n=1}^t N_n=
  w\sup_t D_t=+\infty\quad\mbox{a.s.}\,
  $$
  (by~\eqref{eq-D-as}) and 
  $$
E[Y_{t+1}\mid {\mathcal F}_t] = E[w T_{t+1}\mid {\mathcal F}_t]=w Z_t=w S_t + 
w \lambda_t=\delta X_t +\rho_t\,,
  $$
  with $\rho_t=w\lambda_t>0$ and 
$$
  \sum_t \tfrac{\rho_t}{(\theta+t)^{\delta-\epsilon}}=
  O\big( \textstyle{\sum_t} \tfrac{1}{(\theta+t)^{1+w-(\beta+\epsilon)}}\big)<+\infty
  $$ 
  for some $\epsilon>0$ such that $w-(\beta+\epsilon)>0$ (note that this $\epsilon$ exists  because $w>\beta$ by assumption). 
  Therefore, by Theorem~\ref{th-general}, we firstly obtain that $S_t=O(1/t^{1-w})$ and so 
  $Z_t=S_t+\lambda_t=O(1/t^{1-w})+O(1/t^{1-\beta})=O(1/t^{1-w})$ because $\beta<w$. 
Hence,   by~\eqref{eq-sp-quad-T}, we get  
  $$
  E[Y_{t+1}^2\mid{\mathcal F}_t]=w^2(Z_t+Z_t^2 -R_t) 
  =O(Z_t)
  =O(E[Y_{t+1}\mid {\mathcal F}_t])\,.
  $$ 
Therefore, we can apply Theorem~\ref{th-general} (case $i)$) and we obtain that 
  $$
  t^{1-w}S_t\stackrel{a.s.}\longrightarrow S_\infty^*\in (0,+\infty)\,.
  $$
  As before, since $Z_t=S_t+\lambda_t$ and $\beta<w$, we get 
  $$
  t^{1-w} Z_t= t^{1-w} S_t +t^{1-w}\lambda_t \stackrel{a.s.}\longrightarrow S_\infty^*
  $$
  and so the first convergence  holds true with $Z_\infty^*=S^*_\infty$. 
\end{proof}

\begin{proof}[Proof of Theorem~\ref{th-Z-T-medio-altro-caso}] 
In the case $w \leq \beta$ the term involving $\lambda_t$ in the dynamics~\eqref{eq-diff-Z} of $(Z_t)_t$ 
cannot be considered as a remainder term as in the proof of Theorem~\ref{th-Z-T-medio}. Indeed, letting 
\begin{equation*}
Z^*_t = r_t Z_t
\qquad\mbox{with}\qquad
r_t=\begin{cases}
\zeta_t(1-\beta)\quad &\mbox{when } w<\beta\,,\\ 
\tfrac{\zeta_t(1-\beta)}{\ln(\theta+t)} \quad &\mbox{when } w=\beta\,,
\end{cases}
\end{equation*}
where $\zeta_t(x) \sim t^x$ is defined in~\eqref{eq-def-zeta}.
From \eqref{eq-diff-Z}, taking into account \eqref{eq-zeta-1} and~\eqref{eq-zeta-3},  we get 
\begin{equation}\label{eq:dynamics_Z_star}
    Z^*_{t+1} - Z^*_{t} = \eta_{t+1}(\alpha\beta-bZ^*_{t}) +
\eta_{t+1}\Delta M^*_{t+1}+ \rho_{t+1}\,,
\end{equation} 
where $\rho_{t+1}=Z^*_t O(\tfrac{1}{t^2}) + O(\tfrac{1}{t^{1+\beta}})$ (here $O(\tfrac{1}{t^2})$ and $O(\tfrac{1}{t^{1+\beta}})$ are deterministic terms and $\beta>0$), $\Delta M^*_{t+1}=w\,\zeta_{t+1}(1-\beta)(T_{t+1}-Z_t)$ and
\begin{equation*}
\begin{cases}
b=\beta-w\quad\mbox{and}\quad \eta_{t+1}=\tfrac{1}{(\theta+t+1)}    
    \quad&\mbox{if } w< \beta\,;
\\
b=1\quad\mbox{and}\quad \eta_{t+1}=\tfrac{1}{(\theta+t+1)\ln(\theta+t+1)}  
    \quad & \mbox{if } w= \beta\,.
\end{cases}
\end{equation*}
This suggests that $(Z^{*}_t)_t$ is not a non-negative almost super-martingale, as in the case $w>\beta$. Instead, it evolves according to the stochastic approximation dynamics in~\eqref{eq:dynamics_Z_star}, which suggests convergence toward a deterministic limit given by $\alpha\beta/b$. For this reason, we apply Theorem~\ref{thm:app-cor} to the process $(Z^{*}_t)_t$ governed by~\eqref{eq:dynamics_Z_star}.
To this end, note that from Lemma~\ref{lem-unif-int-new} we have
that $\sup_t E[r_t Z_t]<+\infty$ 
and hence, since by \eqref{eq-sp-quad-T},
\begin{equation*}
\begin{split}
E[(\Delta M^*_{t+1})^2\mid{\mathcal F}_t]&=
w^2\zeta_{t+1}(1-\beta)^2\left(E[T_{t+1}^2\mid\mathcal{F}_t]-Z_t^2\right)\\
&=w^2\zeta_{t+1}(1-\beta)^2(Z_t-R_t)\\
&\leq w^2\zeta_{t+1}(1-\beta)^2 Z_t\,,
\end{split}
\end{equation*}
we get, for any $w\leq \beta$, that
\begin{equation*}
\begin{split}
\sum_t \eta_{t+1}^2
E[(\Delta M_{t+1}^*)^2] \leq
\sum_t \tfrac{w^2 r_t}{(\theta+t+1)^2}
E[r_t Z_t]
<+\infty\,,
\end{split}
\end{equation*}
which implies 
$$
\sum_t \eta_{t+1}^2 E[(\Delta M^*_{t+1})^2\mid{\mathcal F}_t]<+\infty
\quad\hbox{a.s.}\,.
$$ 
In addition, since  $\sup_t E[Z_t^*]=\sup_t E[r_t Z_t]<+\infty$ (by Lemma~\ref{lem-unif-int-new}), the expression for $\rho_{t+1}$  implies  that  $\sum_t E[\,|\rho_{t+1}|\,]= 
\sum_t O(E[Z^*_t]/t^2) + O(1/t^{1+\beta})$ is convergent, and so 
$\sum_t |\rho_{t+1}|<+\infty$ with probability one. 
Then, we can apply Theorem~\ref{thm:app-cor} to the dynamics~\eqref{eq:dynamics_Z_star} so obtaining
\begin{equation*}
Z^*_{t} \stackrel{a.s.}\to \tfrac{\alpha\beta}{b} =
\begin{cases}
\tfrac{\alpha\beta}{\beta-w}
    \quad & \mbox{if } w< \beta\,;
\\
\alpha\beta
    \quad & \mbox{if } w= \beta\,.
\end{cases}
\end{equation*}
\end{proof}

\subsection{Second-order asymptotic results}\label{sec:proofs_second_order_average}

\begin{proof}[Proof of Theorem~\ref{thm:second_order-Z}]
We will first prove the CLT for the process $S_t=\sum_{j\in\mathcal{O}_t}P_{t,j}$ and then, since $S_t=Z_t-\lambda_t$, we will derive from it the one for $Z_t$. 
Indeed, since the dynamics of $S_t$ in~\eqref{eq-dyn-S}, i.e.
$$
  S_{t+1}=\left(1-\tfrac{1-w}{\theta+t+1}\right)S_t+\tfrac{w(T_{t+1}-Z_t)}{\theta+t+1}+
  \tfrac{w\lambda_t}{\theta+t+1},
$$
is equal to the general dynamics in~\eqref{eq-general-dyn-real} with 
$\delta=w\in (0,1]$, $\Delta M_{t+1}=w(T_{t+1}-Z_t)$, $\rho_{t+1}=w\lambda_t$, $\rho=\beta\in [0,1]$ 
and  $\rho_\infty=w\alpha$ (because $t^{1-\beta}\lambda_t$ trivially converges to $\alpha$), 
and since we are assuming $\beta<w$, that is~$\rho<\delta$, 
the CLT for $S_t$ can be obtained by applying Theorem~\ref{th-CLT-general} with $X_t=S_t$ and $X^*_\infty=Z^*_\infty$
(recall that, by Theorem~\ref{th-Z-T-medio}, $t^{1-w}S_t\stackrel{a.s.}\to Z_\infty^*$).  
Indeed, condition~\eqref{AA:eq:momento-M1} follows by~\eqref{eq:boundT_t} in Lemma~\ref{lem-unif-int-new} as for any $p\geq1 $ and $\epsilon>0$ we have 
$$
E\Big[ \sum_{t=1}^{\infty} 
\tfrac{|\Delta {M}_{t+1}|^p}{t^{w+\epsilon}}
\Big] =
\sum_{t=1}^{\infty} \tfrac{1}{t^{1+\epsilon}}
E[ t^{1-w}|\Delta {M}_{t+1}|^p] 
<+\infty\,.
$$
Moreover,
condition~\eqref{AA:eq:momento-convergence} follows by~\eqref{eq-sp-quad-T}, Theorem~\ref{th-Z-T-medio} and Lemma~\ref{lem-RsuS} (in the case $\beta<w$):
$$
 t^{1-w} E[(\Delta M_{t+1})^2|\mathcal{F}_t]=
w^2 t^{1-w}(Z_t-R_t)\stackrel{a.s.}\longrightarrow V_\infty^{*} = w^2 \Sigma_\infty\,.
$$ 
Then, since 
$(t^{1-\beta}\lambda_t-\alpha)=\alpha[1-1/(t+1)]^{1-\beta}-\alpha=O(t^{-1})=o(t^{-(\beta-w/2)})$,  
also condition $(*)$ of Theorem~\ref{th-CLT-general} is verified with a.s.\ convergence and so, from that theorem, we get
\[
t^{w/2}(t^{1-w}S_t+\tfrac{ w\alpha }{(w-\beta)}t^{-(w-\beta)} -Z^*_\infty)\stackrel{stably}\longrightarrow {\mathcal N}(0, w\Sigma_\infty)\,.
\] 
For what concerns the process $(Z_t)_t$ recall that $Z_t=S_t+\lambda_t$ and so 
$t^{1-w}Z_t=t^{1-w}S_t + t^{1-w}\lambda_t$, 
and hence the thesis follows by the fact that $t^{1-w}\lambda_t=\alpha t^{\beta-w}+O(t^{-(1+w-\beta)})= \alpha t^{\beta-w}+o(t^{-w/2})$.  
In addition, note that, from \eqref{eq:compact_tcl_Z_t}, for $w/2<\beta<w$ we have
\[\begin{aligned}
{t}^{w-\beta}(t^{1-w}Z_t -Z^*_\infty) &+\tfrac{\alpha\beta}{w-\beta} \\
&=
\tfrac{1}{t^{\beta-w/2}}\, t^{w/2}\, t^{\beta-w} \,[t^{w-\beta}(t^{1-w}Z_t-Z^*_\infty)+\tfrac{\alpha\beta}{w-\beta} ]\\
&= \tfrac{1}{t^{\beta-w/2}}\, t^{w/2}\,[(t^{1-w}Z_t-Z^*_\infty)+\tfrac{\alpha\beta}{w-\beta} t^{\beta-w} ]\stackrel{P}\rightarrow 0.
\end{aligned}
\]
\end{proof}

\begin{proof}[Proof of Corollary~\ref{cor-clt}] 
From from~\eqref{eq:compact_tcl_Z_t}, we can derive the corresponding result for $\overline{T}_t$. Indeed, by \eqref{eq-S}, we have
\(\overline{T}_t=(1+\tfrac{\theta}{t})S_t/w\) and, in addition, we have $\lambda_t=\alpha/(t+1)^{1-\beta}=\alpha/t^{1-\beta}+(1-\beta)O(1/t^{2-\beta})$, 
 so that we obtain
\begin{equation*}
\begin{split}
&t^{w/2}(t^{1-w}\overline{T}_t+\tfrac{\alpha}{(w-\beta)}t^{-(w-\beta)}-\tfrac{1}{w}Z^*_\infty)=\\
&t^{w/2}(t^{1-w}\overline{T}_t+\tfrac{\alpha\beta}{w(w-\beta)}t^{-(w-\beta)}+\tfrac{t^{1-w}}{w}\tfrac{\alpha}{t^{1-\beta}}-\tfrac{1}{w}Z^*_\infty)
=\\
&t^{w/2}(t^{1-w}\overline{T}_t+\tfrac{\alpha\beta}{w(w-\beta)}t^{-(w-\beta)}+\tfrac{t^{1-w}}{w}\lambda_t-\tfrac{1}{w}Z^*_\infty\Big)+O(1/t^{1+w/2-\beta})=\\
&\tfrac{t^{w/2}}{w}(t^{1-w}Z_t+\tfrac{\alpha\beta}{(w-\beta)}t^{\beta-w}-Z^*_\infty)+
\tfrac{\theta}{w} \tfrac{S_t}{t^{w/2}}+O(\tfrac{1}{t^{1+w/2-\beta}})\stackrel{stably}\longrightarrow 
{\mathcal N}(0, \Sigma_\infty/w)
\end{split}
\end{equation*}
and this convergence also holds true in the sense of the a.s. conditional convergence with respect to $\mathcal F$. 
Consequently, when $(w-\beta)<w/2$, the quantity  
\begin{equation*}
\begin{split}
t^{w-\beta}(t^{1-w}\overline{T}_t-\tfrac{1}{w}Z^*_\infty)+\tfrac{\alpha}{(w-\beta)}=
\tfrac{t^{w-\beta}}{t^{w/2}} t^{w/2}
(t^{1-w}\overline{T}_t+\tfrac{\alpha}{(w-\beta)}t^{-(w-\beta)}-\tfrac{1}{w}Z^*_\infty)
\end{split}
\end{equation*}
converges in probability toward zero. \\
\indent Now, remember that $\overline{P}_t=S_t/D_t=(Z_t-\lambda_t)/D_t$ so that, using the notation and the CLT for $(D_t)_t$ provided in Remark~\ref{rem-clt-D}, we can write
\begin{equation*}
\begin{split}
\overline{P}_t&=
\tfrac{1}{\lambda(\beta)a_t(\beta)} S_t
(
1 - \tfrac{D_t-\lambda(\beta)a_t(\beta)}{D_t}
)
\\
&=\tfrac{1}{\lambda(\beta)a_t(\beta)} S_t 
\Big[1-\tfrac{1}{\tfrac{D_t}{a_t(\beta)\lambda(\beta)}}\tfrac{1}{\sqrt{a_t(\beta)}} \cdot\tfrac{\sqrt{a_t(\beta)}}{\lambda(\beta)}\big(\tfrac{D_t}{a_t(\beta)}-\lambda(\beta)\big)\Big]\\
&=\tfrac{1}{\lambda(\beta)a_t(\beta)} S_t [1+\tfrac{1}{\sqrt{a_t(\beta)}} W_t]\,,
\end{split}
\end{equation*}
where $W_t\stackrel{stably}\longrightarrow {\mathcal N}(0,1/\lambda(\beta))$. 
Hence,  since $\lambda_t=\alpha/t^{1-\beta}+(1-\beta)O(1/t^{2-\beta})$, we get 
\begin{equation*}
\begin{split}
&t^{1-w}\lambda(\beta)a_t(\beta)\overline{P}_t+\tfrac{\alpha w}{w-\beta}t^{-(w-\beta)}-Z^*_\infty =\\
&t^{1-w}\lambda(\beta)a_t(\beta)\overline{P}_t + \alpha t^{-(w-\beta)} +\tfrac{\alpha\beta}{w-\beta}t^{-(w-\beta)}  -Z^*_\infty =\\
&\left(t^{1-w}Z_t+\tfrac{\alpha\beta}{w-\beta}t^{-(w-\beta)}-Z_\infty^*\right)+O(\tfrac{1}{t^{1+w-\beta}})+\tfrac{t^{1-w}S_t}{\sqrt{a_t(\beta)}}W_t\,.
\end{split}
\end{equation*}
Hence, taking into account~\eqref{eq:compact_tcl_Z_t}, since $\sqrt{a_t(\beta)}=o(t^{w/2})$, we obtain~\eqref{clt-Pmedio}. 
As a consequence, when $(w-\beta)<\beta/2$, we have $\sqrt{a_t(\beta)}=t^{\beta/2}$ and  
\begin{equation*}
\begin{split}
&t^{w-\beta}\left(t^{1-w}\lambda(\beta)a_t(\beta)\overline{P}_t-Z^*_\infty\right) +\tfrac{\alpha w}{w-\beta}=\\
&\tfrac{t^{w-\beta}}{\sqrt{a_t(\beta)}} \sqrt{a_t(\beta)}\Bigl(t^{1-w}\lambda(\beta)a_t(\beta)\overline{P}_t+\tfrac{\alpha w}{w-\beta}t^{-(w-\beta)}-Z^*_\infty\Bigr) 
\stackrel{P}\longrightarrow 0\,.
\end{split}
\end{equation*}
Finally, recalling  that  $\overline{K}_t=(\theta+t)\overline{P}_t/w$ (see~\eqref{eq-P-medio-2}), we have 
\begin{multline*}
w\lambda(\beta)\tfrac{a_t(\beta)}{t^w}\overline{K}_t+\tfrac{\alpha w}{w-\beta}t^{-(w-\beta)}-Z^*_\infty \\
= (t^{1-w}\lambda(\beta)a_t(\beta)\overline{P}_t+\tfrac{\alpha w}{w-\beta}t^{-(w-\beta)}-Z^*_\infty )+\theta\lambda(\beta)\tfrac{a_t(\beta)}{t^w}\overline{P}_t
\end{multline*}
so that, from~\eqref{clt-Pmedio}, we obtain~\eqref{clt-Kmedio}  
and, as a consequence, when $(w-\beta)<\beta/2$, we have (with the same arguments used above)
$$
t^{w-\beta}\left( \lambda(\beta)\tfrac{a_t(\beta)}{t^w}\overline{K}_t-\tfrac{Z^*_\infty}{w} \right)+\tfrac{\alpha}{w-\beta}
\stackrel{P}\longrightarrow 0\,.
$$
\end{proof}

\section{Technical results related to the model}\label{sec:technical_results_model}

In this section we present some technical lemmas used in the paper. To this end, let us take 
 the sequence $(r_t)$ as in \eqref{eq:r_t_appB} with $\delta=w$ e $\rho=\beta$.
\\
We prove a lemma regarding a condition of uniform integrability. 
\begin{lem}\label{lem-unif-int-new}
We have 
\begin{equation}\label{eq-tesi1-lemma}
\sup_{t\geq 1} E[e^{z\,r_t\,\overline{T}_t}]<+\infty,\quad 
\sup_{t\geq1} E\Big[\,e^{z\,r_t\,Z_t}\,\Big]<+\infty
\end{equation}
for some $z>0$. 
For the sequence $(T_t)_t$ and any $p\geq 1$, we have 
\begin{equation}\label{eq:boundT_t}
\begin{cases}
    \sup_t E[\,(r_tT_{t+1})^p\,] < \infty  & \text{for $w=1$ or $\beta=1$}\,;\\
    \sup_t E[\,r_t(T_{t+1})^p\,] < \infty  & \text{for $w<1$ and $\beta<1$}\,.\\
\end{cases}
\end{equation}
\end{lem}
As a consequence, 
we note that the sequences $(r_t\,Z_t)_t$ and $(r_t\,\overline{T}_t)_t$, 
are bounded in $L^p$ for any $p\geq 1$.

\begin{proof} We split the proof in some steps.\\
\indent {\em Step 1): } First notice that, for any $z>0$,
$$\begin{aligned}
E\bigl[e^{z\,T_{t+1}}\mid\mathcal{F}_t\bigr]& = 
E\bigl[e^{z\,\sum_{j\in\mathcal{O}_t}X_{t+1,j}+z\,N_{t+1}}\mid\mathcal{F}_t\bigr]\\&=
E\bigl[e^{z\,N_{t+1}}\mid\mathcal{F}_t\bigr]
\prod_{j\in\mathcal{O}_t}E\bigl[e^{z\,X_{t+1,j}}\mid\mathcal{F}_t\bigr]\\
&=e^{(e^z-1)\,\lambda_t}\prod_{j\in\mathcal{O}_t} \bigl(
e^z P_{t,j} + 1-P_{t,j}
\bigr)\,.
\end{aligned}$$
Then, denoting $g(z)=e^z-1$, since $g(z)\geq 0$ and $1+x\leq e^x$ for $x\geq 0$ and 
using~\eqref{eq-S} we have
 \begin{equation}\label{eq-proof-1}
   \begin{split}
  E\bigl[e^{z\,T_{t+1}}\mid\mathcal{F}_t\bigr]
  &=e^{g(z)\,\lambda_t}\,\prod_{j\in {\mathcal O}_t}\bigl(1+g(z)\,P_{t,j}\bigr)\\
  &\leq\exp{\Bigl(g(z)\,\lambda_t+g(z)\sum_{j\in {\mathcal O}_t}P_{t,j}\Bigr)}
 \\
 &= \exp \Bigl( g(z)\,\lambda_t + g(z)\, \tfrac{w\sum_{n=1}^t T_n}{\theta+t}
   \Bigr) 
    \\
    &= \exp \Bigl( g(z)\,\lambda_t + g(z)\,w\,\tfrac{t}{\theta+t}\overline{T}_t
      \Bigr) \quad\text{a.s.}\,.
   \end{split}
   \end{equation}
\indent {\em Step 2):}  
Define $\overline{T}_t^*=r_t\overline{T}_t$ and note that
\begin{equation}\label{eq:u_v}
\begin{split}
\overline{T}_{t+1}^*&= r_{t+1}\tfrac{\sum_{n=1}^t T_n + T_{t+1}}{t+1}
=
\tfrac{r_{t+1}}{r_t}\tfrac{t}{t+1}\overline{T}_t^*
+ T_{t+1}\tfrac{r_{t+1}}{t+1}\\
&= (1-v_t)\,\overline{T}_t^*+ u_t\, T_{t+1}\,,
\end{split}
\end{equation}
where by~\eqref{eq-zeta-1} and~\eqref{eq-zeta-3}
$$
u_t=\tfrac{r_{t+1}}{t+1}\rightarrow 0
\qquad\text{and}\qquad
v_t=1-\tfrac{u_{t}}{u_{t-1}}\geq 0\,.
$$
Hence, by \eqref{eq-proof-1} and \eqref{eq:u_v}, we get
\begin{equation}\label{eq-proof-2}
\begin{aligned}
  E\bigl[& e^{z\,\overline{T}^*_{t+1}}\bigr]
   = E\Bigl[\,\exp{\Bigl(z\,\overline{T}^*_t\,(1-v_t)\Bigr)}
\,E\bigl[e^{z\,u_t\,T_{t+1}}\,\mid\,\mathcal{F}_t\bigr]\,\Bigr]
  \\
&  \leq \exp{\Bigl(\lambda_t\,g(z\,u_t)\Bigr)}
  E\Bigl[\,\exp{\Bigl( g(z\,u_t)w\overline{T}_t+z\,(1-v_t)\,\overline{T}^*_t\Bigr) } \Bigr]
    \\
 &=\exp{\Bigl(\lambda_t\,g(z\,u_t)\Bigr)}
  E\Bigl[\,\exp{\Bigl( g(z\,u_t)\tfrac{w}{r_t}\overline{T}^*_t+z\,(1-v_t)\,\overline{T}^*_t\Bigr) } \Bigr]\,.
\end{aligned}\end{equation}
Since $v_t\geq 0$ and $w\leq 1$, it follows
\begin{gather*}
  E\bigl[ e^{z\,\overline{T}^*_{t+1}}\bigr]
  \leq \exp{\Bigl(\lambda_t\,g(zu_t)\Bigr)}\,\,
  E\Bigl[\exp{\Bigl(\Big(z+\tfrac{g(zu_t)}{r_t}\Big)\,\overline{T}^*_t\Bigr)}\Bigr].
\end{gather*}
Iterating this last inequality, we obtain
\begin{gather*}
  E\bigl[ e^{z\,\overline{T}^*_{t+1}}\bigr]\leq a_t(z)\,E\bigl[e^{b_t(z)\,
  \overline{T}^*_1}\bigr]
\end{gather*}
for suitable deterministic quantities $a_t(z)$ and $b_t(z)$. 
Since we have $\overline{T}^*_1=r_1T_1$ with $T_1\stackrel{d}=\,$Poi$(\alpha)$, then
$E\bigl[e^{b_t(z)r_1\,T_1}\bigr]<+\infty$ so that  
$E\bigl[e^{z\,\overline{T}^*_t}\bigr]<+\infty$
for all fixed $t\geq 1$ and $z>0$. 
This also implies $E[e^{z'\,\overline{T}_t}]<+\infty$ for all fixed $t\geq 1$ and $z'>0$ (take in the above relation $z=z'/r_t$), 
which by~\eqref{eq-proof-1} also entails 
$E[e^{zT_t}]<+\infty$ for all fixed $t\geq 1$ and $z>0$.
\\
\indent 
{\em Step 3):} Observe now that $g(x)\leq 2\,x$ and $g(x)-x\leq x^2$ for $x\in [0,\,1/2]$. Moreover, since $u_t\to0$, we know
$z u_t\leq u_t\leq 1/2$ for $z\in (0,1]$ and $t$ large enough.
These facts are used for $z\in(0,1]$ to write $g(z u_t)\leq 2 z u_t$ and $g(z u_t)\leq (z u_t)+(z u_t)^2$ when $t$ is large enough. Hence, by~\eqref{eq-proof-2}, 
we get for $z\in (0,1]$ and $t$ large enough:
\begin{equation*}
\begin{aligned}
  E\bigl[e^{z\,\overline{T}^*_{t+1}}\bigr]
   &\leq
  \exp{\Bigl(2z\,\lambda_t\,u_t\Bigr)}\,
    E\Bigl[\exp \Bigl(
    (z\,u_t)^2 \tfrac{w}{r_t}\overline{T}^*_t
     +\Big(1-v_t+u_t\tfrac{w}{r_t}\Big)z\overline{T}^*_t
    \Bigr) \Bigr]\\
    &\leq
  \exp{\Bigl(2z\,\lambda_t\,u_t\Bigr)}\,
    E\Bigl[\Big(1-v_t+(u_t+u_t^2)\tfrac{w}{r_t}\Big)z\overline{T}^*_t
    \Bigr) \Bigr],\\
\end{aligned}
\end{equation*}
where we have used that $z^2\leq z$ as $z\in (0,1]$. Recalling that $v_t=1-u_{t}/u_{t-1}$ and $u_{t-1}=r_{t}/t$, we have 
$$\begin{aligned}
 \pi_t&=   1-v_t+(u_t+u_t^2)\tfrac{w}{r_t} =
\tfrac{u_{t}}{u_{t-1}} + (1+u_t)\tfrac{u_{t}}{u_{t-1}}\tfrac{w}{t} \\
&= \tfrac{u_{t}}{u_{t-1}}\Big(1+ \tfrac{w}{t} +\tfrac{w}{t} u_t \Big) 
= \tfrac{u_{t}}{u_{t-1}}\Big(1+ \tfrac{w}{t}\Big)\Big(1 +\tfrac{w/t}{1+w/t} u_t \Big)\\
&= \tfrac{u_{t}}{u_{t-1}}\tfrac{\widetilde{\zeta}_t(w)}{\widetilde{\zeta}_{t-1}(w)}\Big(1 +\tfrac{w}{t+w} u_t \Big)\,,
\end{aligned}
$$
where $\widetilde{\zeta}_t(x)$ is defined as $\zeta_t(x)$ in~\eqref{eq-def-zeta} with $\theta=0$. 
This leads for $z\in (0,1]$ to
\begin{equation}\label{eq:sum_prod_0}
E\bigl[e^{z\,\overline{T}^*_{t+1}}\bigr]\leq
\exp{\Bigl(2z\,\lambda_t\,u_t\Bigr)}\,
    E\Bigl[\exp \Bigl(
    z\,\pi_t\,\overline{T}^*_t
    \Bigr) \Bigr].
\end{equation}
Now, note that for any $n<t$ 
\begin{equation}\label{eq:prod_telescopic}
    \prod_{j=n+1}^t \pi_j = \prod_{j=n+1}^t \tfrac{u_{j}}{u_{j-1}}\tfrac{\widetilde{\zeta}_j(w)}{\widetilde{\zeta}_{j-1}(w)}\big(1 +\tfrac{w}{j+w} u_j \big)= 
\tfrac{u_{t}}{u_{n}}\tfrac{\widetilde{\zeta}_t(w)}{\widetilde{\zeta}_{n}(w)}\prod_{j=n+1}^t\Big(1 +\tfrac{w}{j+w} u_j \Big)
\end{equation}
and since 
\begin{itemize}
    \item[(i)] $\limsup_{t\rightarrow\infty}u_t\widetilde{\zeta}_t(w)\leq 1$ and
    \item[(ii)] $C=\prod_{t\geq 1}
\Big(1 +\tfrac{w}{t+w} u_t \Big)<+\infty$ as $\sum_{t\geq 1} \tfrac{w}{t+w} u_t<+\infty$
\end{itemize}
we can choose $z_0\in(0,1)$ and $t_0$ large enough such that 
$$\sup_{t\geq t_0}\Big(\prod_{j=t_0}^t \pi_j\Big)z_0\leq 1\,.$$
This ensures that we can iterate the inequality~\eqref{eq:sum_prod_0}, so obtaining
\begin{multline}\label{eq:sum_prod_1}
\begin{aligned}
    E\bigl[e^{ z_0\,\overline{T}^*_{t+1}}\bigr]& \leq
\exp{\Bigl( 2z_0\,\sum_{n= t_0}^t \lambda_n\, u_n \prod_{j=n+1}^t \pi_j\Bigr)}
E\bigl[e^{ (\prod_{j=t_0}^{t} \pi_j)z_0\,\overline{T}_{t_0}^*}\bigr]
\\
& \leq
\exp{\Bigl( 2z_0\,\sum_{n= t_0}^t \lambda_n\, u_n \prod_{j=n+1}^t \pi_j\Bigr)}
E\bigl[e^{\overline{T}_{t_0}^*}\bigr]
\end{aligned}
\end{multline}
for each $t\geq t_0$.
Then, the second term in the right-hand product in~\eqref{eq:sum_prod_1} is always bounded by  Step 2) and so we can now focus on the first term, which by~\eqref{eq:prod_telescopic} becomes less or equal than
\begin{equation*}
\exp{\Bigl(2z_0\,C\,u_t \widetilde{\zeta}_t(w) \sum_{n= t_0}^t \tfrac{\lambda_n}{\widetilde{\zeta}_{n}(w)}\Bigr)}\sim
\exp{\Bigl( 2z_0\,C\,\tfrac{r_t}{t^{1-w}}\sum_{n= t_0}^t \tfrac{\alpha}{n^{1+w-\beta}}\Bigr)}\,,
\end{equation*}
that is finite in all cases as
\begin{itemize}
    \item[(a)] in the case $w>\beta$ we have $\tfrac{r_t}{t^{1-w}}\sim 1$ and $\sum_{n= t_0}^t \tfrac{\alpha}{n^{1+w-\beta}}<+\infty$;
    \item[(b)] in the case $w=\beta$ we have  $\tfrac{r_t}{t^{1-w}}\sim 1/\ln(t)$ and $\sum_{n= t_0}^t \tfrac{\alpha}{n^{1+w-\beta}}=O(\ln(t))$;
    \item[(c)] in the case $w<\beta$ we have $\tfrac{r_t}{t^{1-w}}\sim 1/t^{\beta-w}$ and $\sum_{n= t_0}^t \tfrac{\alpha}{n^{1+w-\beta}}=O( t^{\beta-w})$.
\end{itemize}
Therefore from~\eqref{eq:sum_prod_1} we can conclude that, for any $z\leq z_0$,  
$\sup_{t\geq 1} E[e^{z\,\overline{T}_t^*}]\leq
\sup_{t\geq 1} E[e^{z_0\,\overline{T}_t^*}]<+\infty$.
\\
\indent Finally, since by~\eqref{eq-S} we have $r_t Z_t=w\tfrac{t}{t+\theta}\overline{T}_t^*+r_t\lambda_t\leq \overline{T}_t^*+c$ for some suitable $c>0$, 
we also obtain $\sup_{t\geq 1} E[e^{z\,r_t Z_t}]<+\infty$ for any $z\leq z_0$. \\
The $L^p$-boundness of the sequences $(r_t\,Z_t)_t$ and $(r_t\,\overline{T}_t)_t$ follows immediately, as for non-negative random variables we have 
\begin{gather*}
\sup_t E( e^{z_0 r_t Z_t}) <\infty  
\iff \sup_t \sum_{p\geq 0} \tfrac{z_0^p E[\,(r_t Z_t)^p\,]}{p!} < +\infty \\
\Longrightarrow 
\sup_t E[\,(r_t Z_t)^p\,] < +\infty\; \forall p\geq 1.
\end{gather*}
and the same for $(r_t\,\overline{T}_t)_t$. For what concerns $({T}_t)_t$, from \eqref{eq-proof-1}, \eqref{eq-S}, \eqref{eq-Z} and 
the fact that $g(z)\leq 2z$ when $z\leq 1/2$,
 we obtain for any $t\geq 0$ and $z\leq 1/2$
\begin{equation}\label{eq-lemma-new}
E\bigl[e^{z\,T_{t+1}}\mid\mathcal{F}_t\bigr]\leq \exp ( g(z)\, (\lambda_t + S_t)) = \exp ( g(z)\, Z_t)
\leq 
\exp ( 2z\, Z_t)\,.
\end{equation}
Now, note that $r_t \equiv 1$ when $w<\beta=1$ or $\beta<w=1$, and in this case the two expressions in \eqref{eq:boundT_t} coincide. 
More generally, we can consider the following cases.
\\
\noindent\emph{Case $w=1$ or $\beta=1$:} we have $r_t \leq 1$ and so, by~\eqref{eq-lemma-new}, for each $t\geq 0$, we have 
 \begin{equation*}
   \begin{split}
  E[ e^{\tfrac{z_0}{2} r_t\,T_{t+1}}]
 &\leq E[e^{z_0\, r_t Z_t}]\,.
   \end{split}
   \end{equation*}
Then, by~\eqref{eq-tesi1-lemma},  we get $\sup_t E[e^{\tfrac{z_0}{2} r_t\,T_{t+1}}] <+\infty$, that implies as above 
 $\sup_t E[\,(r_tT_{t+1})^p\,] < +\infty$ for all $p\geq 1$.
\\
\noindent\emph{Case $w<1$ or $\beta<1$:} we have $r_t \geq 1$ for $t$ large enough, say $t\geq \bar{t}$. 
Hence, by~\eqref{eq-lemma-new}, for $t\geq \bar{t}$, we can write\\
 \begin{equation*}
   \begin{split}
  r_t(E[e^{\tfrac{z_0}{2}\,T_{t+1}}]-1)&\leq r_t ( E[e^{z_0\, Z_t}] -1) 
 =r_t ( E[e^{z_0\, \tfrac{ Z_t^*}{r_t}}] -1) \\
   &= \sum_{p\geq 1} \tfrac{z_0^p E[(Z^*_t)^p]}{(r_t)^{p-1}\,p!}
    \leq \sum_{p\geq 1} \tfrac{z_0^pE[(Z^*_t)^p]}{p!}
    = (E[e^{z_0\, Z^*_t}] -1)\,.
   \end{split}
   \end{equation*}
Hence, we can conclude that $\sup_{t\geq \bar{t}} \,r_t(E[e^{\tfrac{z_0}{2}\,T_{t+1}}]-1) <+\infty$, that reads
\[
\sup_{t\geq \bar{t}}\, \sum_{p\geq 1} \tfrac{(z_0)^p E[r_t(T_{t+1})^p]}{2^p \, p!} < +\infty\,,
\]
and so we can get  $\sup_{t\geq\bar{t}} \, E[\,r_t(T_{t+1})^p\,] < +\infty$, for all $p\geq 1$. 
Since, by Step 2), for all $p\geq 1$ and each fixed $t$ we have $E[(T_{t+1})^p]<+\infty$, and so $E[r_t (T_{t+1})^p]<+\infty$,  we can also affirm that 
$\sup_{t} \, E[\,r_t(T_{t+1})^p\,] < +\infty$ for all $p\geq 1$.
\end{proof}

\begin{lem} \label{lem-RsuS} 
We have
$$
r_t R_t \stackrel{a.s.}\longrightarrow
\begin{cases}
    \Rstar_\infty & \text{if $\iota=0$, $w=1$ and $\beta<1$}\,;
    \\
    0 & \text{otherwise}\,.
\end{cases}
$$
where $\Rstar_\infty$ is a real random variable such that $0<\Rstar_\infty <Z^*_\infty$. 
\end{lem}

\begin{proof}
From \eqref{eq-P-2}, since $D_t \overline{P}_t^2=\overline{P}_t S_t$ we get 
\begin{equation*}
\begin{split}
R_t&=\sum_{j\in {\mathcal O}_t} P_{t,j}^2
=\sum_{j\in {\mathcal O}_t}
\big(w(1-\iota)\tfrac{K_{t,j}}{(\theta+t)}+\iota\overline{P}_t\big)^2\\
&=(w(1-\iota))^2 \Rstar_t+ \iota^2 D_t \overline{P}_t^2 + 2 \iota(1-\iota) \overline{P}_t S_t\\
&= (w(1-\iota))^2 \Rstar_t+ \iota(2-\iota)\overline{P}_t S_t \,, 
\end{split}
\end{equation*}
where, by~\eqref{eq-S}, 
\begin{equation*}
0 \leq \Rstar_t=\tfrac{\sum_{j\in {\mathcal O}_t}K_{t,j}^2}{(\theta+t)^2}\leq 
\left(\tfrac{\sum_{j\in {\mathcal O}_t}K_{t,j}}{\theta+t}\right)^2=
S_t^2/w^2\,.
\end{equation*}
Since $\overline{P}_t\to 0$ a.s.\ and $r_t S_t = O(1)$ from~\eqref{eq-S} and Table~\ref{table-Z-T-P-K-medio}, then $r_t R_t \to 0$ a.s.\ when $\iota=1$, 
otherwise if and only if $ r_t \Rstar_{t} \to 0$ a.s.\ and this is always the case when $\max\{w,\beta\}<1$, because $r_t \Rstar_{t} \leq r_t S_t^2/w^2\to 0$ a.s..

From now on, we assume $\iota<1$ and $\max\{w,\beta\}=1$ (and so $r_t=1$ when $\beta\neq w$ and $r_t\to 0$ when $\beta=w$) and we are going to prove that 
$r_t \Rstar_{t}\to 0$ a.s.\ unless $\iota=0$, $w=1$ and $\beta<1$.
We observe that 
\begin{equation*}
\begin{split}
\Rstar_{t+1}=
\sum_{j\in {\mathcal O}_t}\big(\tfrac{K_{t+1,j}}{\theta+t+1}\big)^2+
\sum_{j\in {\mathcal O}_{t+1}\setminus{\mathcal O}_t}
\big(\tfrac{1}{\theta+t}\big)^2
=\sum_{j\in {\mathcal O}_t}\big(\tfrac{K_{t,j}+\Delta K_{t+1,j}}{\theta+t+1}\big)^2
+
\tfrac{N_{t+1}}{(\theta+t+1)^2}
\end{split}
\end{equation*}
and so, since $(\Delta K_{t+1,j})^2=\Delta K_{t+1,j}$, we have 
\begin{equation}\label{eq:R*pred}
\begin{split}
E[\Rstar_{t+1}\mid{\mathcal F}_t]
&=\sum_{j\in {\mathcal O}_t}
\tfrac{K_{t,j}^2+P_{t,j}+2K_{t,j}P_{t,j}}{(\theta+t+1)^2}
+
\tfrac{\lambda_t}{(\theta+t+1)^2} \,,
\end{split}
\end{equation}
while,  the martingale difference reads 
\begin{multline*}
\Rstar_{t+1}-E[\Rstar_{t+1}\mid{\mathcal F}_t]
\\
= \tfrac{1}{\theta+t+1}\Big( \sum_{j\in {\mathcal O}_t}
\big(\Delta K_{t+1,j}-P_{t,j}\big)\tfrac{1+2K_{t,j}}{\theta+t+1}
+\tfrac{N_{t+1}-\lambda_t}{(\theta+t+1)}\Big)
= 
\tfrac{\Delta M^{(\Rstar)}_{t+1}}{\theta+t+1}\,.
\end{multline*}
Since the random variables $(\Delta K_{t+1,j})_{j\in\mathcal{O}_t}$ are,   conditionally on $\mathcal{F}_t$, independent among each other, and $N_{t+1}$ is independent of them and ${\mathcal F}_t$
(see the description of the  model dynamics), we have
\begin{equation}\label{eq:R*martComp}
\begin{split}
E[(\Delta M^{(\Rstar)}_{t+1})^2\mid{\mathcal F}_t] 
&=\sum_{j\in {\mathcal O}_t}P_{t,j}(1-P_{t,j})\Big(\tfrac{1+2K_{t,j}}{\theta+t+1}\Big)^2
+\tfrac{\lambda_t}{(\theta+t+1)^2}\,.
\end{split}
\end{equation}
Now, we go on with the computations for $E[\Rstar_{t+1}\mid {\mathcal F}_t]$. From~\eqref{eq:R*pred}, recalling that $\Rstar_t=\sum_{j\in{\mathcal O}_t} K_{t,j}^2/(\theta+t)^2$ 
and using~\eqref{eq-S} and~\eqref{eq-P-2}, 
we get 
\begin{align*}
E[\Rstar_{t+1}\mid{\mathcal F}_t]
&=\Rstar_t+
\sum_{j\in {\mathcal O}_t}K_{t,j}^2\left(\tfrac{1}{(\theta+t+1)^2}-\tfrac{1}{(\theta+t)^2}\right)+\tfrac{S_t}{(\theta+t+1)^2}\\
&\qquad +
\tfrac{2}{(\theta+t+1)^2}\sum_{j\in{\mathcal O}_t} K_{t,j}\left((1-\iota)w\tfrac{K_{t,j}}{\theta+t}+\iota \overline{P}_t\right)
+\tfrac{\lambda_t}{(\theta+t+1)^2} \\
&=\Rstar_t+
\sum_{j\in {\mathcal O}_t}K_{t,j}^2\left(\tfrac{1}{(\theta+t+1)^2}-\tfrac{1}{(\theta+t)^2} +
\tfrac{2(1-\iota)w}{(\theta+t)(\theta+t+1)^2}\right)\\
&\qquad +\tfrac{S_t}{(\theta+t+1)^2}+
\tfrac{2\iota}{(\theta+t+1)^2}\overline{P}_t\sum_{j\in{\mathcal O}_t} K_{t,j}
+\tfrac{\lambda_t}{(\theta+t+1)^2} \\
&=\Rstar_t-\tfrac{2(\iota+1-w)}{\theta+t+1}\Rstar_t  
+
\tfrac{2\iota}{(\theta+t+1)^2}\overline{P}_t\sum_{j\in{\mathcal O}_t} K_{t,j}
+ \rho_t
\,.
\end{align*}
with $\rho_t$ being ${\mathcal F}_t$-measurable and such that 
$\sum_t|\rho_t|<+\infty$ since 
$\rho_t=O((\Rstar_t+S_t+1)/t^2)=O((S_t+1)/t^2)$. 
Using \eqref{eq-P-medio-2}, we finally obtain
\begin{equation*}
\begin{split}
E[\Rstar_{t+1}\mid{\mathcal F}_t]
&= \Rstar_t-\tfrac{2(\iota+1-w)}{\theta+t+1}R_t^*+
\tfrac{2\iota (\theta+t)D_t}{(\theta+t+1)^2{w}}\overline{P}_t^2
+\rho_t
\end{split}
\end{equation*}
and hence
\begin{multline}\label{eq:R*_in_proof}
\Rstar_{t+1}
= \Rstar_t + \tfrac{1}{\theta+t+1}\Big( \tfrac{2\iota (\theta+t)D_t\overline{P}_t^2}{(\theta+t+1)w} -2(\iota+1-w) \Rstar_t\Big) +
\tfrac{\Delta M^{(\Rstar)}_{t+1}}{\theta+t+1}
+\rho_t
\,.
\end{multline}
Now, we consider two cases. \\
{\em Case $\iota > 0$ or $w<1$ (and $\beta=1$)}: in this case we can apply Theorem~\ref{thm:app-cor} to \eqref{eq:R*_in_proof} with 
$X_t=\Rstar_t$, $A_t= \tfrac{2\iota(\theta+t) D_t\overline{P}^2_t}{(\theta+t+1)w}$, $b=2(\iota+1-w)>0$,
$\eta_t=(\theta+t+1)^{-1}$ and $\Delta M_{t+1}=\Delta M^{(\Rstar)}_{t+1}$ (and in the present case the reminder term $\rho_{t+1}$ in Theorem~\ref{thm:app-cor} coincides here with $\rho_t$ and so it is ${\mathcal F}_t$-measurable). Indeed, 
by \eqref{eq:R*martComp}, we have 
\begin{align*}
    \sum_t& \eta_t^2 E[(\Delta M_{t+1})^2\mid{\mathcal F}_t]\\ & = 
    \sum_t \tfrac{1}{(\theta+t+1)^2}\bigg( \sum_{j\in {\mathcal O}_t}P_{t,j}(1-P_{t,j})\Big(\tfrac{1+2K_{t,j}}{\theta+t+1}\Big)^2+\tfrac{\lambda_t}{(\theta+t+1)^2}\bigg) \\
    & \leq \sum_t \tfrac{1}{(\theta+t+1)^2} \bigg(\sum_{j\in {\mathcal O}_t} P_{t,j} 
    \Big(\tfrac{1+2t}{\theta+t+1}\Big)^2 +\tfrac{\lambda_t}{(\theta+t+1)^2}\bigg) \\
    & \leq \sum_t \Big(\tfrac{4S_{t}}{(\theta+t+1)^2} + \tfrac{\lambda_t}{(\theta+t+1)^4} \Big)<+\infty\,.
\end{align*}
In addition, we have  
\[
A_t = \tfrac{2\iota}{w}\tfrac{(\theta+t)}{(\theta+t+1)}D_t\overline{P}^2_t \stackrel{a.s.}\longrightarrow 0\,.
\]
Hence,  by Theorem~\ref{thm:app-cor}, we obtain $\Rstar_t\to 0$ a.s.\ and, since in this case $r_t=1$, this means $r_t\Rstar_t\to 0$ a.s..
\\
{\em Case $\iota = 0$ and $w=1$ (and $\beta<1$ or $=1$):} in this case $R_t=\Rstar_t$ and  $\Rstar_t$ is a non-negative almost super-martingale converging a.s. to a non-negative real random variable $\Rstar_\infty$ 
(see Theorem~\ref{th-almost-supermart}). Then, when $w=\beta=1$, since $r_t\to 0$ in this case, we trivially get  
$r_tR_t=r_t \Rstar_t\to 0$ a.s..\\
\indent 
Let us focus on the remaining sub-case, that is when $\iota=0$, $w=1$ and $\beta<1$ 
(and so $r_t=1$).  First, since $P_{t,j}^2\leq P_{t,j}$, we have 
 $\Rstar_t=R_t\leq S_t=(Z_t-\lambda_t)\to Z_\infty^*$ and so we obtain $0\leq\Rstar_\infty\leq Z^*_\infty$. Now, let $\tau_1=\inf\{t\colon N_t>0\}$ and note that in this case \eqref{inclusion-probab} reads
\[
P_{\tau_1+n}(1) = \tfrac{K_{\tau_1+n}}{\tau_1+n+\theta}\,,
\]
that evolves, conditionally on the value taken by $\tau_1$ (note that $\tau_1<+\infty$ a.s.), 
as the proportion of black balls in a Pòlya urn that starts with $1$ black ball and $(\tau_1+\theta)$ white balls. 
Then, conditionally on the value taken by $\tau_1$, we have  $P_{\tau_1+n}(1)\stackrel{a.s.}\to P_\infty(1) \stackrel{d}= Beta(1,\tau_1+\theta)$, so that we have $P_\infty(1)(1-P_\infty(1))>0$. 
Note that, since $\beta<w=1$ and $\iota=0$, we have 
\[
\Rstar_\infty = P_\infty(1)^2 + \text{a.s.-}\lim_t \sum_{j>1,j\in{\mathcal O}_t} P_{t,j}^2
\]
and
\[
Z^*_\infty = S^*_\infty= P_\infty(1) + \text{a.s.-}\lim_t \sum_{j>1,j\in{\mathcal O}_t} P_{t,j}
\]
and hence $$0 <P_\infty(1)^2\leq \Rstar_\infty < \Rstar_\infty + P_\infty(1)(1-P_\infty(1))\leq Z^*_\infty.
$$
Hence, we can conclude because in this last sub-case $r_t=1$ and so $r_t\Rstar_t=\Rstar_t$.
\end{proof}

\begin{lem}\label{lem:Delta K}
  When $0\leq w\iota < \beta$ we have 
\begin{equation}\label{eq:bound Delta K}
   \sup_{t\geq1} E[\,t^{1-w(1-\iota)}\Delta K_{t+1,j}\,] < + \infty\,.
\end{equation}
\end{lem}

\begin{proof}
    First, recall from Section~\ref{subsec-D} that $D_t=\sum_{n=1}^t N_{n}\stackrel{d}=$Poi$(\Lambda_t)$ with 
    $\Lambda_t=\sum_{n=1}^t \lambda_{n-1}\sim (\alpha/\beta) t^\beta$ when $\beta>0$.
    Then, consider the sequence $(r_t)$ as defined in \eqref{eq:r_t_appB} with $\delta=w$ e $\rho=\beta$.
   Since $E[\Delta K_{t+1,j}]=w(1-\iota) E[B_{t,j}]+ \iota E[\overline{P}_t]$, 
    the proof will be divided in the proof of the following two facts:
\begin{itemize}
    \item[(1)] $E[t^{1-w+w\iota}\overline{P}_t]=O(t^{-(\min\{\beta,w\}-w\iota)}\ln(t))\to 0$,
    \item[(2)] $\sup_{t\geq1} t^{1-w+w\iota} E[B_{t,j}] < +\infty$.
\end{itemize}
\noindent \emph{Proof of (1):}
Consider
$$
   E[t^{1-w+w\iota}\overline{P}_t] = \tfrac{t^{1-w}}{r_t}
     E\Big[\tfrac{r_t S_t}{t^{-w\iota}D_t}\Big]\,.
$$
Since $D_t\stackrel{d}=$Poi$(\Lambda_t)$, by Chernoff's bound (see~\cite{MitUpf17}), we have 
$$
P\big(D_t < (1-\epsilon)\Lambda_t\big) \leq e^{-\tfrac{\epsilon^2}{2}\Lambda_t}
$$
and so, using also Holder's inequality, we get
$$\begin{aligned}
   E\Big[\tfrac{r_t S_t}{t^{-w\iota}D_t}\Big]
   &= E\Big[\tfrac{r_tS_t}{t^{-w\iota}D_t}\ind_{\{D_t \geq (1-\epsilon)\Lambda_t\}}\Big] + 
      E\Big[\tfrac{r_tS_t}{t^{-w\iota}D_t}\ind_{\{D_t < (1-\epsilon)\Lambda_t\}}\Big]\\
   &\leq E\Big[\tfrac{r_tS_t}{t^{-w\iota}(1-\epsilon)\Lambda_t}\Big] + 
      E[(r_tS_t)^{2}]^{\frac{1}{2}}\,t^{w\iota} P\big(D_t < (1-\epsilon)\Lambda_t\big)^{\frac{1}{2}}\\
 &\leq E[r_tS_t] \tfrac{t^{w\iota}}{(1-\epsilon)\Lambda_t} + 
      E[(r_tS_t)^{2}]^{\frac{1}{2}} \,
      t^{w\iota}e^{-\tfrac{\epsilon^2}{4}\Lambda_t}\,.
\end{aligned}
$$
Now, we observe that, by Lemma~\ref{lem-unif-int-new} and the fact that $0\leq S_t=Z_t-\lambda_t\leq Z_t$, we have 
$\sup_{t} E[(r_tS_t)^p]\leq \sup_{t} E[(r_tZ_t)^p] <+\infty$ for $p=1,2$ and so the above inequality 
implies that 
$$E[t^{1-w+w\iota}\overline{P}_t]=\tfrac{t^{1-w}}{r_t}O\big(\tfrac{t^{w\iota}}{\Lambda_t}\big)=
O\big(t^{-(\min\{\beta,w\}-w\iota)}\ln(t)\big)\rightarrow 0\,.$$
\noindent \emph{Proof of (2):} Setting $m_{t}=E[B_{t}]$,  by~\eqref{eq:dynamics_B_tj} we have
\begin{equation*}
\begin{aligned}
    m_{t+1}=
\big(1-\tfrac{1}{\theta+t+1}\big) m_t +\tfrac{1}{\theta+t+1} E[P_{t,j}]
=\Big(1-\tfrac{1-w(1-\iota)}{\theta+t+1}\Big) m_t
+\tfrac{\iota}{\theta+t+1} E[\overline{P_t}]\,.
\end{aligned}
\end{equation*}
Then, setting $m^{*}_t=\zeta_t(1-w+w\iota)m_t$ and using~\eqref{eq-zeta-1}, we get
\begin{equation*}
\begin{aligned}
m^{*}_{t+1}&=\Big(1-\tfrac{(1-w(1-\iota))^2}{(\theta+t+1)^2}\Big) m^{*}_t
+\tfrac{\iota}{\theta+t+1}E[t^{1-w+w\iota}\overline{P}_t]\\
&\leq  m^{*}_t+\tfrac{\iota}{\theta+t+1}E[t^{1-w+w\iota}\overline{P}_t]\,,
\end{aligned}
\end{equation*}
which implies that $(m^*_t)_t$ is bounded because, by fact~$(1)$, we have  
$$
\sum_tE[t^{1-w+w\iota}\overline{P}_t]/(\theta+t+1)=\sum_t O(1/t^{1+\min\{\beta,w\}-w\iota}) < +\infty\,.
$$
\end{proof}

\subsection{Proof of the law of the iterated logarithm for $D_t$}\label{app-LIL}
We here provide the proof of the result for $D_t$ given in~Remark~\ref{LIL-D}.
\begin{proof}
Recall that $D_t - \Lambda_t = \sum_{n=1}^t (N_n -\lambda_{n-1})$, where $(N_n -\lambda_{n-1})_n$ is a sequence of zero-mean independent random variables and 
$Var[D_t] = \sum_{n=1}^t Var[N_n] = \sum_{n=1}^t\lambda_{n-1} = \Lambda_t$. Set $c_n = \sqrt{\tfrac{\Lambda_n}{{\ln (\Lambda_n)}}}$. 
For $n$ large enough, we have $\lambda_{n-1} < c_n$ and consequently  $P(N_n-\lambda_{n-1} \leq - c_n) = P(N_n\leq \lambda_{n-1}-c_n)=0$. 
Then, by the bounds for the tail probabilities of a Poisson random variable (e.g. \cite[Theorem~5.4]{Mitzenmacher_Upfal_2005}), we obtain for $n$ large enough 
\begin{equation*}
\begin{split}
P(|N_n-\lambda_{n-1}|\geq c_n)&= P(N_n \geq \lambda_{n-1} + c_n)
\leq
e^{-\lambda_{n-1}} \tfrac{ (e\lambda_{n-1})^{\lambda_{n-1}+c_n} }{(\lambda_{n-1}+c_n)^{\lambda_{n-1}+c_n}}
\\
&= \big(\tfrac{e\lambda_{n-1}}{c_n}\big)^{c_n} \big(\tfrac{\lambda_{n-1}}{c_n}\big)^{\lambda_{n-1}} 
\tfrac{ 1 }{(\tfrac{\lambda_{n-1}}{c_n}+1)^{c_n}(\tfrac{\lambda_{n-1}}{c_n}+1)^{\lambda_{n-1}}}
\\
&\leq C \big(\tfrac{e\lambda_{n-1}}{c_n}\big)^{c_n}\,,
\end{split}
\end{equation*}
where the last inequality is due to the fact that $\lambda_{n-1} = o(c_n)$. 
Moreover, we have  
\(
\sum_n \big(\tfrac{e\lambda_{n-1}}{c_n}\big)^{c_n} < +\infty
\). In fact, for $\beta=0$, we have $\lambda_{n-1} = \alpha / n $ and hence $\big(\tfrac{e\lambda_{n-1}}{c_n}\big)^{c_n} = \big(\tfrac{e\alpha}{nc_n}\big)^{c_n} = o(n^{-2})$
and for $\beta>0$, we have $\lambda_{n-1}=\alpha/n^{1-\beta}$ and so $c_n \geq n^{\beta/(2+\epsilon)}$ eventually, which, together with $\lambda_{n-1}/c_n\to 0$, implies 
$ \big(\tfrac{e\lambda_{n-1}}{c_n}\big)^{c_n}\leq 2^{- n^{\beta/(2+\epsilon)}}$ eventually. 
By Borel-Cantelli Lemma, we finally obtain 
$$
P\Big(\limsup_n\{|N_n-\lambda_{n-1}|\geq c_n\}\Big)=0
$$
and so 
\[
|N_n-\lambda_{n-1}| = O (c_n) = o \left(\sqrt{\tfrac{\Lambda_n}{\ln\ln\Lambda_n}}\right)\qquad 
\hbox{a.s.}\,.
\]
Therefore, \eqref{eq-lil} is a consequence of Kolmogorov's law of the iterated logarithm \cite{fisher} applied to the sequence $(N_n -\lambda_{n-1})_n$ 
(see \cite[Theorem~1, p.~343]{chow_teicher}). 
\end{proof}

\section{Useful recalls}

For the reader's convenience, 
we here recall some technical results for discrete-time real stochastic processes, used in the previous proofs.

\begin{lem}[{\cite[Theorem~46, p.~40]{del-mey}}]\label{del-mey-result-app}
  Let $\mathcal F=(\mathcal{F}_t)_t$ be a filtration and $(Y_t)$ be an $\mathcal
F$-adapted non-negative stochastic process. \\  
 \indent Then the set $\{\sum_t
  E[Y_{t+1}\vert \mathcal{F}_{t}]<+\infty\}$ is a.s. contained in
  the set $\{\sum_t Y_t<+\infty\}$. If, in addition, the random variables $Y_t$ are
  uniformly bounded by a constant, then these two sets are a.s. equal.
\end{lem}

\begin{lem}[{L\'evy's extension of Borel–Cantelli Lemma \cite[Sec.~12.15]{williams}}]\label{williams-lemma}
Let $(Y_t)_t$ be a sequence of Bernoulli random variables, adapted to a filtration $\mathcal{F}=(\mathcal{F}_t)_t$ 
and  such that $Z_t=P(Y_{t+1}=1\mid\mathcal{F}_t)$ and $\sum_t Z_t=+\infty$ a.s.. 
Then $\sum_{n=0}^t Y_{n+1} / \sum_{n=0}^{t} Z_n \stackrel{a.s.}\longrightarrow 1$.
\end{lem}

\begin{lem}[{\cite[Lemma~3.2]{pemantle-volkov-1999}}]\label{lemma-pemantle} 
Let $\mathcal{F}=({\mathcal F}_t)_t$ be a filtration and 
  let $(L_t)_t$ be an $\mathcal F$-adapted  stochastic process. Set
  $U_t=E[L_{t+1}-L_t\mid \mathcal{F}_t]$ and
  $V_t=E[(L_{t+1}-L_t)^2\mid \mathcal{F}_t]$. If $\sum_t U_t$ and
  $\sum_t V_t$ are a.s. convergent, then $(L_t)_t$ converges
  a.s. to a real random variable.
\end{lem}

\begin{thm} [{\cite[Non-negative almost super-martingale]{rob}}]
\label{th-almost-supermart}
Let $(Y_t)_t$ be a non-negative real almost super-martingale with respect to a filtration 
 $\mathcal{F}=(\mathcal{F}_t)_t$, that is an $\mathcal{F}$-adapted non-negative real 
 stochastic process  satisfying
$$
E[Y_{t+1}|\mathcal{F}_t]\leq (1+\Delta_t)Y_t+\rho_{1,t}-\rho_{2,t}\,,
$$
where $\Delta_t$, $\rho_{1,t}$, $\rho_{2,t}$ are all  $\mathcal{F}$-adapted  non-negative real  stochastic processes.
Then $(Y_t)_t$ a.s. converges to a real random variable 
and $\sum_t \rho_{2,t}<+\infty$ a.s. on the event 
$\{\sum_t\Delta_t<+\infty\,,\sum_t \rho_{1,t}<+\infty\}$.
\end{thm}

\begin{lem}[{\cite[Lemma 4.1]{cri-dai-min}}]\label{lemma-kro}
Let $\mathcal F$ be a filtration and $(Y_t)$ be an $\mathcal
F$-adapted  stochastic process such that each $Y_t$ is square-integrable and
$E[Y_{t+1}|{\mathcal F}_{t}]\to Y_\infty$ a.s.\ for some real random variable
$Y_\infty$. Moreover, let $(a_t)$ and $(b_t)$ be two sequences of strictly
positive real numbers such that
$$ b_t\uparrow +\infty,\quad
\sum_{t=1}^{\infty}\tfrac{E[Y_t^2]}{a_t^2\,b_t^2}<+\infty\,.
$$
Then we have:
\begin{itemize}
\item[a)] if $\tfrac{1}{b_t}\sum_{n=1}^t\tfrac{1}{a_n}\to \gamma$
for some constant $\gamma$, then 
$
\tfrac{1}{b_t}\sum_{n=1}^t \tfrac{Y_n}{a_n}\stackrel{a.s.}
\longrightarrow \gamma Y_\infty;
$
\item[b)] If $b_t\sum_{n\geq t} \tfrac{1}{a_n\,b_n^2}\to \gamma$
for some constant $\gamma$, then 
$
b_t\sum_{n\geq t} \tfrac{Y_n}{a_n\,b_n^2}\stackrel{a.s.}
\longrightarrow \gamma Y_\infty.
$
\end{itemize}
\end{lem}

\begin{thm}[{\cite[Theorem A.1]{crimaldi-2009}}]
\label{fam_tri_vet_as_inf}
 Let $(\mathcal{F}_{t})_{t\geq 0}$ be a complete 
  filtration and, for each $n\geq 1$, let  $(M_{n,t})_{t\geq 0}$ be 
  a martingale with respect
  to the filtration $({\mathcal{F}}_{n+(t-1)})_{t\geq 0}$, with $M_{n,0}=0$, and 
  converging in $L^1$ to a real (${\mathcal F}_\infty$-measurable) random variable $M_{n,\infty}$. Moreover, set
\begin{equation*}
X_{n,t}=M_{n,t}-M_{n,t-1}\quad\hbox{for } t\geq 1,\quad
U_n=\sum_{t\geq 1} X_{n,t}^2,\quad
X_n^*=\sup_{t\geq 1} |X_{n,t}|
\end{equation*}
and assume that the following conditions are satisfied:
\begin{itemize}
\item[i)] $X_n^*\stackrel{a.s.}\longrightarrow 0$; 
\item[ii)] $(X_n^*)_n$ is dominated in $L^1$;
\item[iii)] $(U_n)_n\stackrel{a.s.}\longrightarrow U$ for a real (${\mathcal F}_\infty$-measurable) random variable~$U$.
\end{itemize}  
Then, the sequence $(M_{n,\infty})_{n}$ converges to the
Gaussian kernel ${\mathcal{N}}(0,U)$  
in the sense of the a.s.
conditional convergence with respect to~$({\mathcal F}_{n})_n$ (and so also stably),  that is 
the a.s. convergence of the conditional distributions holds true: 
for almost every $\omega$, the conditional distribution $\mu_n(\omega,\cdot)$ of
$M_{n,\infty}$ given $\mathcal{F}_n$ converges weakly to the Gaussian distribution 
${\mathcal N}(0,U(\omega))$.
\end{thm}
The a.s.\ conditional convergence is a variant of the stable convergence. 
For more information on stable convergence and its variants, we refer the reader 
to~\cite{Ald-Eag-1978, crimaldi-2009, crimaldi-libro, CriLetPra, HH, Ren}. 
An useful summary can be found in the appendix of \cite{ale-cri-ghi}. 
We here provide only a lemma, regarding the fact that 
 the a.s. conditional convergence combines well with the a.s. convergence, that we need in the proof of Theorem~\ref{th-CLT-general}.

\begin{lem}\label{lem-new-as-cond-conv}
Let ${\mathcal F}=({\mathcal F})_t$ be a filtration, $K$ an ${\mathcal F}_\infty$-measurable kernel and $Y$ an ${\mathcal F}_\infty$-measurable real random variable. 
If $X_t\longrightarrow K$ in the sense of the a.s. conditional convergence with respect to a filtration $\mathcal F=({\mathcal F}_t)_{t\geq 0}$ and 
$Y_t\stackrel{a.s.}\longrightarrow Y$,  then
$$
[X_t,Y_t]\longrightarrow K\otimes \delta_Y
$$
in the sense of the a.s. conditional convergence with respect to ${\mathcal F}$.\\
\indent In particular, when $K$ is a Gaussian kernel ${\mathcal N}(M,U)$, we have 
$$
X_t+Y_t\longrightarrow {\mathcal N}(M+Y,U)\,.
$$
in the sense of the a.s. conditional convergence with respect to ${\mathcal F}$.
\end{lem}

\begin{proof} 
Since Cram\'er-Wold device and \cite[Lemma A.3]{crimaldi-2009}, it is enough to prove that, for each $a_1,a_2\in\mathbb{R}$, we have 
$$
E[e^{i(a_1X_t+a_2Y_t)}\mid{\mathcal F}_t]\stackrel{a.s.}\longrightarrow 
\left(\textstyle{\int} e^{ia_1x}K(\cdot)(dx)\right)e^{ia_2Y}\,.
$$
We first prove the above convergence under the additional assumption that $Y$ is integrable and then we will show how to get the result in the general case.  
Suppose $Y$ to be integrable and set $Z_{t}=E[Y\mid{\mathcal F}_{t}]$ so that 
we have $Z_{t}\stackrel{a.s.}\to Y$ (by the martingale convergence theorem for martingales bounded in $L^1$). Moreover, we have  
\begin{equation*}
\begin{split}
|D_t|&=\big|E[e^{i(a_1X_t+a_2Y_t)}\mid{\mathcal F}_t]-
\left(\textstyle{\int} e^{ia_1x}K(\cdot)(dx)\right)e^{ia_2Y}\big|\\
&\leq 
|T_{1,t}|+|T_{2,t}|\,,
\end{split}
\end{equation*}
where 
\begin{itemize}
\item $|T_{1,t}|=|E[e^{ia_1X_t}(e^{ia_2Y_t}-e^{ia_2Z_t})\mid {\mathcal F}_t]|$, which is smaller than or equal to 
$E[\,|e^{ia_2Y_t}-e^{ia_2Z_t}|\mid {\mathcal F}_t]$, that converges a.s.\ to zero,  because of 
 \cite[Lemma A.2(d)]{crimaldi-2009} (see also \cite[Th.~2]{blackwell-dubins}) and the a.s. convergence of both $Y_t$ and $Z_t$ to the same limit $Y$;
\item $|T_{2,t}|=|E[e^{i(a_1X_t+a_2Z_{t})}\mid{\mathcal F}_t]-
\left(\int e^{ia_1x}K(\cdot)(dx)\right)e^{ia_2Y}|$, which is equal to the difference 
$$
|e^{ia_2Z_{t}}E[e^{ia_1X_t}\mid{\mathcal F}_t]-
\left(\textstyle{\int} e^{ia_1x}K(\cdot)(dx)\right)e^{ia_2Y}|
$$
and converges a.s. to zero by the assumption on $X_t$ and the a.s. converge of 
$Z_t$ to $Y$. 
\end{itemize}
Therefore, 
we can conclude that $D_t\stackrel{a.s.}\to 0$ and so the result is proven when $Y$ is integrable.\\
\indent If $Y$ is not integrable, then, for any $N\in \mathbb{N}$, we can define 
$Y^{(N)}_t=(-N)\vee Y_t\wedge N$ and 
$Y^{(N)}=(-N)\vee Y\wedge N$ so that $|Y^{(N)}_t|\leq N$, $|Y^{(N)}|\leq N$, $Y^{(N)}_t\stackrel{a.s.}\to Y^{(N)}$ as $t\to +\infty$, $Y^{(N)}_t\stackrel{a.s.}\to Y_t$ 
and $Y^{(N)}\stackrel{a.s.}\to Y$ as $N\to +\infty$. For each fixed $N$, 
applying the result to $X_t\longrightarrow K$ and $Y^{(N)}_t\stackrel{a.s.}\longrightarrow Y^{(N)}$ (which is trivially integrable), we find for $t\to +\infty$
$$
D^{(N)}_t= E[ e^{ i (a_1X_t+a_2Y^{(N)}_t)} \mid{\mathcal F}_t]-
\left(\textstyle{\int} e^{ia_1x}K(\cdot)(dx)\right)e^{ia_2 Y^{(N)}}\stackrel{a.s.}\longrightarrow 0\,.
$$
Moreover, since \cite[Lemma A.2(d)]{crimaldi-2009} and the a.s.\ convergence of the difference 
$|e^{ia_2 Y_t}-e^{ia_2Y^{(N)}_t}|$ to 
$|e^{ia_2 Y}-e^{ia_2Y^{(N)}}|$ for $t\to +\infty$ (and each fixed $N$), we 
have $E[\,|e^{ia_2 Y_t}-e^{ia_2Y^{(N)}_t}|\,|{\mathcal F}_t]$ to 
$|e^{ia_2 Y}-e^{ia_2Y^{(N)}}|$ for $t\to +\infty$ (and each fixed $N$) and we get 
$$
\limsup_t |D_t-D_t^{(N)}|\leq 2|e^{ia_2Y}-e^{ia_2 Y^{(N)}}|\stackrel{a.s.}\longrightarrow 0\quad\mbox{as } N\to +\infty\,.
$$
Hence, $D_t\stackrel{a.s.}\to 0$ and the proof is so concluded.
\end{proof}



\begin{thebibliography}{10}

\bibitem{AS}
M.~Abramowitz and I.~A. Stegun, editors.
\newblock {\em Handbook of Mathematical Functions with Formulas, Graphs, and Mathematical Tables}.
\newblock Dover, New York, 1972.

\bibitem{Ald-Eag-1978}
D.~J. Aldous and G.~K. Eagleson.
\newblock On mixing and stability of limit theorems.
\newblock {\em Ann. Probab.}, 6(2):325--331, 1978.

\bibitem{ale-cri-ghi}
G.~Aletti, I.~Crimaldi, and A.~Ghiglietti.
\newblock Synchronization of reinforced stochastic processes with a network-based interaction.
\newblock {\em Ann. Appl. Probab.}, 27(6):3787--3844, 2017.

\bibitem{ale-cri-ghi-complete}
G.~Aletti, I.~Crimaldi, and A.~Ghiglietti.
\newblock Networks of reinforced stochastic processes: A complete description of the first-order asymptotics.
\newblock {\em Stochastic Process. Appl.}, 176:104427, 2024.

\bibitem{BAN-THA-2022}
A.~Bandyopadhyay and D.~Thacker.
\newblock A new approach to {P}\'olya urn schemes and its infinite color generalization.
\newblock {\em Ann. Appl. Probab.}, 32(1):46--79, 2022.

\bibitem{BCPR-IBP}
P.~Berti, I.~Crimaldi, L.~Pratelli, and P.~Rigo.
\newblock {Central Limit Theorems for an Indian Buffet Model with Random Weights}.
\newblock {\em Ann. Appl. Probab.}, 25(2):523--547, 2015.

\bibitem{blackwell-dubins}
D.~Blackwell and L.~Dubins.
\newblock Merging of opinions with increasing information.
\newblock {\em Ann. Math. Statist.}, 33(3):882--886, 1962.

\bibitem{bol-cri-mon-2016}
P.~Boldi, I.~Crimaldi, and C.~Monti.
\newblock A network model characterized by a latent attribute structure with competition.
\newblock {\em Inf. Sci.}, 354(C):236--256, 2016.

\bibitem{broderick2012beta}
T.~Broderick, M.~I. Jordan, and J.~Pitman.
\newblock Beta processes, stick-breaking, and power laws.
\newblock {\em Bayesian Anal.}, 7(2):439--476, 2012.

\bibitem{CFMB2024}
Federico Camerlenghi, Stefano Favaro, Lorenzo Masoero, and Tamara Broderick.
\newblock Scaled process priors for bayesian nonparametric estimation of the unseen genetic variation.
\newblock {\em Journal of the American Statistical Association}, 119(545):320--331, 2024.

\bibitem{chow_teicher}
Yuan~Shih Chow and Henry Teicher.
\newblock {\em Limit Theorems for Independent Random Variables}, pages 354--403.
\newblock Springer New York, New York, NY, 1997.

\bibitem{crimaldi-2009}
I.~Crimaldi.
\newblock An almost sure conditional convergence result and an application to a generalized {P}\'{o}lya urn.
\newblock {\em Int. Math. Forum}, 4(23):1139--1156, 2009.

\bibitem{crimaldi-libro}
I.~Crimaldi.
\newblock {\em Introduzione alla nozione di convergenza stabile e sue varianti (Introduction to the notion of stable convergence and its variants)}, volume~57.
\newblock Unione Matematica Italiana, Monograf s.r.l., Bologna, Italy, 2016.
\newblock Book written in Italian.

\bibitem{cri-dai-min}
I.~Crimaldi, P.~Dai~Pra, and I.~G. Minelli.
\newblock Fluctuation theorems for synchronization of interacting {P}\'{o}lya's urns.
\newblock {\em Stochastic Process. Appl.}, 126(3):930--947, 2016.

\bibitem{CriLetPra}
I.~Crimaldi, G.~Letta, and L.~Pratelli.
\newblock {\em A Strong Form of Stable Convergence}, volume 1899, pages 203--225.
\newblock Springer, 2007.

\bibitem{cri-dai-lou-min}
I.~Crimaldi, P.~Dai Pra, P.-Y. Louis, and I.~G. Minelli.
\newblock Synchronization and functional central limit theorems for interacting reinforced random walks.
\newblock {\em Stochastic Process. Appl.}, 129(1):70 -- 101, 2019.

\bibitem{del-mey}
C.~Dellacherie and P.-A. Meyer.
\newblock {\em Probabilities and potential. {B}}, volume~72 of {\em North-Holland Mathematics Studies}.
\newblock North-Holland Publishing Co., Amsterdam, 1982.
\newblock Theory of martingales, Translated from the French by J. P. Wilson.

\bibitem{dos-gha-correlated}
F.~Doshi-Velez and Z.~Ghahramani.
\newblock Correlated non-parametric latent feature models.
\newblock In {\em Proceedings of the 25th Conference on Uncertainty in Artificial Intelligence (UAI)}, UAI '09, pages 143--150, Arlington, Virginia, USA, 2009. AUAI Press.

\bibitem{DW-restricted}
F.~Doshi-Velez and S.~A. Williamson.
\newblock Restricted {Indian Buffet Processes}.
\newblock {\em Stat. Comput.}, 27(5):1205--1223, 2017.

\bibitem{dur-res}
R.~Durrett and S.~I. Resnick.
\newblock Functional limit theorems for dependent variables.
\newblock {\em Ann. Probab.}, 6(5):829--846, 1978.

\bibitem{fisher}
E.~Fisher.
\newblock {On the Law of the Iterated Logarithm for Martingales}.
\newblock {\em Ann. Probab.}, 20(2):675--680, 1992.

\bibitem{fortini-petrone-2020}
S.~Fortini and S.~Petrone.
\newblock Quasi-bayes properties of a procedure for sequential learning in mixture models.
\newblock {\em J. R. Stat. Soc. Ser. B Stat. Methodol.}, 82(4):1087--1114, 2020.

\bibitem{gershman2015distance}
S.~J. Gershman, P.~I. Frazier, and D.~M. Blei.
\newblock Distance-dependent infinite latent feature models.
\newblock {\em IEEE Trans. Pattern Anal. Mach. Intell.}, 37(2):334--345, 2015.

\bibitem{Gouet93}
R.~Gouet.
\newblock {Martingale functional central limit theorems for a generalized Polya urn}.
\newblock {\em Ann. Probab.}, 21(3):1624 -- 1639, 1993.

\bibitem{GG06}
T.~L. Griffiths and Z.~Ghahramani.
\newblock Infinite latent feature models and the {Indian Buffet Process}.
\newblock In {\em Advances in Neural Information Processing Systems (NIPS)}, volume~18, 2005.

\bibitem{GG11}
T.~L. Griffiths and Z.~Ghahramani.
\newblock {The Indian Buffet Process}: An introduction and review.
\newblock {\em J. Mach. Learn. Res.}, 12:1185--1224, 2011.

\bibitem{GJR}
D.~Görür, F.~Jäkel, and C.~E. Rasmussen.
\newblock A choice model with infinitely many latent features.
\newblock In {\em Proceedings of the International Conference on Machine Learning (ICML)}, volume~23, pages 361--368, 2006.

\bibitem{HH}
P.~Hall and C.~C. Heyde.
\newblock {\em Martingale Limit Theory and Its Application}.
\newblock Academic Press, New York, 1980.

\bibitem{heaukulani2020gibbs}
C.~Heaukulani and D.~M. Roy.
\newblock Gibbs-type {Indian Buffet Processes}.
\newblock {\em Bayesian Anal.}, 15(3):683--710, 2020.

\bibitem{KG}
D.~A. Knowles and Z.~Ghahramani.
\newblock Nonparametric {B}ayesian sparse factor models with application to gene expression modeling.
\newblock {\em Ann. Appl. Stat.}, 5(2B):1534--1552, 2011.

\bibitem{lijoi07}
A.~Lijoi, R.~H. Mena, and I.~Prünster.
\newblock Bayesian nonparametric estimation of the probability of discovering new species.
\newblock {\em Biometrika}, 94(4):769--786, 2007.

\bibitem{lijoi2010}
A.~Lijoi and I.~Prünster.
\newblock {\em Models beyond the Dirichlet process}.
\newblock Cambridge Series in Statistical and Probabilistic Mathematics. Cambridge University Press, 2010.

\bibitem{MGJ}
K.~T. Miller, T.~L. Griffiths, and M.~I. Jordan.
\newblock The phylogenetic {Indian Buffet Process}: A non-exchangeable nonparametric prior for latent features.
\newblock In {\em Proceedings of the Conference on Uncertainty in Artificial Intelligence (UAI)}, volume~24, pages 403--410, 2008.

\bibitem{MGJ2009}
K.~T. Miller, T.~L. Griffiths, and M.~I. Jordan.
\newblock Nonparametric latent feature models for link prediction.
\newblock In {\em Advances in Neural Information Processing Systems (NIPS)}, 2009.

\bibitem{Mitzenmacher_Upfal_2005}
M.~Mitzenmacher and E.~Upfal.
\newblock {\em Probability and Computing: Randomized Algorithms and Probabilistic Analysis}.
\newblock Cambridge University Press, 2005.

\bibitem{MitUpf17}
M.~Mitzenmacher and E.~Upfal.
\newblock {\em Probability and Computing: Randomization and Probabilistic Techniques in Algorithms and Data Analysis}.
\newblock Cambridge University Press, Cambridge, 2nd edition, 2017.

\bibitem{NG}
D.~J. Navarro and T.~L. Griffiths.
\newblock Latent features in similarity judgment: A nonparametric bayesian approach.
\newblock {\em Neural Comput.}, 20(11):2597--2628, 2008.

\bibitem{pemantle-volkov-1999}
R.~Pemantle and S.~Volkov.
\newblock Vertex-reinforced random walk on $\mathbf{Z}$ has finite range.
\newblock {\em Ann. Probab.}, 27(3):1368--1388, 1999.

\bibitem{pitman_1996}
J.~Pitman.
\newblock Some developments of the {B}lackwell-{M}acqueen urn scheme.
\newblock {\em Lecture Notes Monogr. Ser.}, 30:245--267, 1996.

\bibitem{rai-dau}
P.~Rai and H.~Daume.
\newblock The infinite hierarchical factor regression model.
\newblock In D.~Koller, D.~Schuurmans, Y.~Bengio, and L.~Bottou, editors, {\em Advances in Neural Information Processing Systems}, volume~21. Curran Associates, Inc., 2008.

\bibitem{Ren}
A.~R{\'e}nyi.
\newblock On stable sequences of events.
\newblock {\em Sankhy\=a Ser. A}, 25:293--302, 1963.

\bibitem{rob}
H.~Robbins and D.~Siegmund.
\newblock A convergence theorem for non negative almost supermartingales and some applications.
\newblock In J.~S. Rustagi, editor, {\em Optimizing Methods in Statistics}, pages 233--257. Academic Press, New York, 1971.

\bibitem{Sariev23}
H.~Sariev, S.~Fortini, and S.~Petrone.
\newblock {Infinite-color randomly reinforced urns with dominant colors}.
\newblock {\em Bernoulli}, 29(1):132 -- 152, 2023.

\bibitem{SCJ}
P.~Sarkar, D.~Chakrabarti, and M.~I. Jordan.
\newblock Nonparametric link prediction in dynamic networks.
\newblock In {\em Proceedings of the 29th International Conference on Machine Learning (ICML)}, 2012.

\bibitem{IBP2018hawkes}
X.~Tan, V.~Rao, and J.~Neville.
\newblock The {Indian Buffet Hawkes Process} to model evolving latent influence.
\newblock In {\em Proceedings of the 34th Conference on Proceedings of the Conference on Uncertainty in Artificial Intelligence (UAI)}, 2018.

\bibitem{TG}
Y.~W. Teh and D.~G\"or\"ur.
\newblock {Indian Buffet Processes} with power-law behaviour.
\newblock In {\em Advances in Neural Information Processing Systems (NIPS)}, volume~22, pages 1838--1846, 2009.

\bibitem{TJ}
R.~Thibaux and M.~I. Jordan.
\newblock Hierarchical beta processes and the {Indian Buffet Process}.
\newblock In {\em Proceedings of the International Conference on Artificial Intelligence and Statistics (AISTATS)}, volume~11, pages 564--571, 2007.

\bibitem{warr2021attraction}
R.~L. Warr, D.~B. Dahl, J.~M. Meyer, and A.~Lui.
\newblock {The Attraction Indian Buffet Distribution}.
\newblock {\em Bayesian Anal.}, 17(3):931 -- 967, 2022.

\bibitem{williams}
D.~Williams.
\newblock {\em Probability with Martingales}.
\newblock Cambridge University Press, 1991.

\bibitem{WOG}
S.~Williamson, P.~Orbanz, and Z.~Ghahramani.
\newblock Dependent {Indian Buffet Processes}.
\newblock In Yee~Whye Teh and Mike Titterington, editors, {\em Proceedings of the 13th International Conference on Artificial Intelligence and Statistics (AISTATS)}, volume~9 of {\em Proceedings of Machine Learning Research}, pages 924--931. PMLR, 2010.

\bibitem{WGG}
F.~Wood, T.~L. Griffiths, and Z.~Ghahramani.
\newblock A non-parametric bayesian method for inferring hidden causes.
\newblock In {\em Proceedings of the Conference on Uncertainty in Artificial Intelligence (UAI)}, volume~22, pages 536--543, 2006.

\bibitem{Zhang-2014}
L.-X. Zhang.
\newblock {A Gaussian process approximation for two-color randomly reinforced urns}.
\newblock {\em Electron. J. Probab.}, 19:1--19, 2014.

\bibitem{Zhang2016}
L.-X. Zhang.
\newblock Central limit theorems of a recursive stochastic algorithm with applications to adaptive designs.
\newblock {\em Ann. Appl. Probab.}, 26(6):3630--3658, 2016.

\end{thebibliography}

\end{document}